\documentclass[11pt.final]{article} 

\usepackage{amsfonts,amsthm,amsmath,amssymb, mathtools}
\usepackage{dsfont}
\usepackage{graphicx,color}
\usepackage{hyperref}
\usepackage{a4wide}
\usepackage{multirow}
\usepackage{cleveref}
\usepackage{enumitem}
\usepackage{tikz-cd}

\newtheorem{proposition}{Proposition}[section]
\newtheorem{theorem}[proposition]{Theorem}
\newtheorem{lemma}[proposition]{Lemma}
\newtheorem{corollary}[proposition]{Corollary}
\newtheorem{definition}[proposition]{Definition}
\newtheorem{remark}[proposition]{Remark}

\renewenvironment{proof}{\medskip\noindent{\textbf{Proof.}}%
  \hspace{1pt}}{\hspace{-5pt}{\nobreak\quad\nobreak\hfill\nobreak%
    $\square$\vspace{2pt}\par}\smallskip\goodbreak}

\newenvironment{proofof}[1]{\smallskip\noindent{\textbf{Proof~of~#1.}}%
  \hspace{1pt}}{\hspace{-5pt}{\nobreak\quad\nobreak\hfill\nobreak%
    $\square$\vspace{2pt}\par}\smallskip\goodbreak}

\numberwithin{equation}{section}
\numberwithin{figure}{section}
\numberwithin{table}{section}

\allowdisplaybreaks[4]

\usepackage{marginnote}

\renewcommand{\phi}{\varphi}
\renewcommand{\theta}{\vartheta}
\renewcommand{\epsilon}{{\color{blue}\varepsilon}}
\renewcommand{\L}[1]{\mathbf{L^#1}}
\renewcommand{\div}{\mathinner{\mathop{\rm div}}}
\renewcommand{\d}[1]{\mathinner{\mathrm{d}{#1}}}

\newcommand{\C}[1]{\mathbf{C^{#1}}}

\newcommand{\Cc}[1]{\mathbf{C_c^{#1}}}

\newcommand{\W}[2]{\mathbf{W^{#1,#2}}}
\newcommand{\BV}{\mathbf{BV}}

\newcommand{\modulo}[1]{{\left|#1\right|}}
\newcommand{\norma}[1]{{\left\|#1\right\|}}
\newcommand{\reali}{{\mathbb{R}}}
\newcommand{\naturali}{{\mathbb{N}}}

\newcommand{\Lip}{\mathop\mathbf{Lip}}
\newcommand{\tv}{\mathop\mathrm{TV}}

\renewcommand{\O}{\mathinner{\mathcal{O}(1)}}

\newcommand{\spt}{\mathop\mathrm{spt}}

\newcommand{\pint}[1]{\mathaccent23{#1}}
\newcommand{\sign}{\mathop{\rm sign}}

\newcommand{\eps}{\varepsilon}

\newcommand{\comp}{\mathop\bigcirc}

\newcommand{\BC}{\mathop\mathbf{BC}}     
\newcommand{\BCvarF}[2]{\mathbf{BC^{#1,#2}}}        

\newcommand{\N}{\mathbb{N}}

\newcommand{\mA}{\mathcal{A}}

\newcommand{\mD}{\mathcal{D}}

\newcommand{\hlf}{\frac{1}{2}}

\newcommand\spa{{\cal U}} \newcommand\spb{{\cal W}}
 
\newcommand\vra{u}  

\newcommand\doa{\mathcal{D}^{\mathcal{U}}}

\newcommand\Doa{\mathcal{A}_{\mathcal{U}}}

\definecolor{dblue}{rgb}{0,0,0.7}

\begin{document}

\title{On the Coupling of Well Posed Differential Models\\
  Detailed Version}

\author{R.M.~Colombo\footnotemark[1] \and M.~Garavello\footnotemark[2]
  \and M.~Tandy\footnotemark[3]}

\maketitle

\footnotetext[1]{Unit\`a INdAM \& Dipartimento di Ingegneria
  dell'Informazione, Universit\`a di Brescia, Italy.\hfill\\
  \texttt{rinaldo.colombo@unibs.it}}

\footnotetext[2]{Dipartimento di Matematica e Applicazioni,
  Universit\`a di Milano -- Bicocca, Italy.\hfill\\
  \texttt{mauro.garavello@unimib.it}}

\footnotetext[3]{Department of Mathematical Sciences, NTNU Norwegian
  University of Science and Technology, Norway.\hfill\\
  \texttt{matthew.tandy@ntnu.no}}

\begin{abstract}

  Consider the coupling of $2$ evolution equations, each generating a
  global process. We prove that the resulting system generates a new
  global process. This statement can be applied to differential
  equations of various kinds. In particular, it also yields the well
  posedness of a predator--prey model, where the coupling is in the
  differential terms, and of an epidemiological model, which does not
  fit previous well posedness results.

  \medskip


  \medskip

  \noindent\textbf{MSC~2020:} 34G20; 35M30; 35L65; 35F30.
\end{abstract}

\newpage

\tableofcontents

\newpage

\section{Introduction}
\label{sec:introduction}

A variety of models describing the evolution in time of real
situations is obtained coupling simpler models devoted to specific
subsystems. In this paper we provide a framework where the well
posedness of the \emph{``big''} model follows from that of its parts.

Predictive models consisting of couplings of evolution equations,
possibly of different types, are very common in the applications of
mathematics. Here we only note that their use ranges, for instance,
from epidemiology~\cite{MR4283943, mod, ColomboMarcelliniRossi}, to
traffic modeling~\cite{MR3264413, MR2765683}, to several specific
engineering applications~\cite{MR3229002, MR3559371}.

In this manuscript, the core result is set in a metric space, so that
linearity plays no role whatsoever. This also allows the range of
applicability of the general theorem to encompass, for instance,
ordinary, partial and measure differential equations. In each of these
cases, we obtain stability estimates tuned to the metric structure
typical of the specific evolution equation considered, which can be,
for example, the Euclidean norm in $\reali^n$, the $\L1$ norm in
spaces of $\BV$ functions or some Wasserstein type distance between
measures.

At the abstract level, the starting point is provided by the framework
of evolution equations in metric spaces, see~\cite{MR1230367, MR1712751,
  BressanCauchy, zbMATH05509403, LorenzBook, Panasyuk}. In this
setting, an evolution equation is well posed as soon as it generates a
\emph{Global Process}, i.e., a Lipschitz continuous solution operator,
see Definition~\ref{def:Global}. In other words, global processes
substitute, in the time dependent case, semigroups that, in the
autonomous case, have as trajectories the solutions to evolution
equations.

Assume that two evolution equations are given, each depending on a
parameter and each generating a global process, also depending on that
parameter. We now let the parameter in an equation vary in time
according to the other equation: a coupling between the two models is
thus obtained. Theorem~\ref{thm:Metric} ensures the well posedness of
this coupled model, in the sense that it generates a new global
process.

The assumptions required in this abstract construction are then
verified in $5$ sample situations: ordinary differential equations,
initial and boundary value problems for renewal equations, measure
valued balance laws and scalar conservation laws. Thus, we prove that
any coupling of these equations results in a well posed model. Indeed,
in each of these cases, we provide a full set of detailed stability
estimates compatible with the abstract results. Note that assumptions
ensuring global in time existence results are also provided.

Finally, we consider specific cases. First, we briefly show that
Theorem~\ref{thm:Metric} comprises the case of the traffic model
introduced in~\cite{MR2765683}, where a scalar conservation law is
coupled to an ordinary differential equation.

Then, we detail the case of a predator--prey model inspired
by~\cite{MR4030510}, namely
\begin{equation}
  \label{eq:45}
  \left\{
    \begin{array}{l}
      \displaystyle
      \partial_t \rho
      +
      \div_x \left(
      \rho \;
      V\!\left(t, x, p (t)\right)
      \right)
      =
      - \eta \left(\norma{p (t) - x} \right)\, \rho (t,x)
      \\
      \dot p = U \left(t, p, \rho (t)\right) \,.
    \end{array}
  \right.
\end{equation}
While we refer to~\S~\ref{subsec:consensus-game} for a detailed
explanation of the terms in~\eqref{eq:45}, here we remark that
in~\eqref{eq:45} the coupling is not only in the source term of the
partial differential equations, but also in the convective term, where
no nonlocal term is involved ($V$ is a function defined for
$t \in \reali$, $x \in \reali^n$ and $P (t) \in \reali^n$).

Then, we apply the general construction to a recent epidemiological
model presented in~\cite{ColomboMarcelliniRossi} whose well posedness,
to our knowledge, was not proved at the time of this writing. In this
case, the coupling involves a boundary value problem for a renewal
equation, see~\S~\ref{subsec:an-epid-model}.

\medskip

For all basic results on evolution equations in metric spaces, we
refer to the extended treatises~\cite{MR1230367, MR1712751,
  LorenzBook}, whose wide bibliographies also give a detailed view on
the whole field. Below, we follow the approach outlined
in~\cite{BressanCauchy, zbMATH05509403, Panasyuk}. The different
frameworks differ in their approaches but offer similar
results. Related to Theorem~\ref{thm:Metric} is, for instance,
\cite[Theorem~26]{LorenzBook}. However, here we follow a more
quantitative approach to the various stability estimates.

We expect that also other equations fit in the framework introduced in
Section~\ref{sec:defin-abstr-results}. Natural candidates are, for
instance, measure differential equations~\cite{MR3961299, MR4026977}
and their coupling with ordinary differential equations as considered
in~\cite{GongPiccoli2022}. A further class of couplings is that
in~\cite{MR3229002}, consisting of ordinary and partial differential
equations similar to those comprised in \S~\ref{subsec:IBVP2}. Very
likely to comply with the present structure is also the general class
of traffic models presented in~\cite{MR3556545}.

\medskip

This work is organized as
follows. Section~\ref{sec:defin-abstr-results}, once the basic
notation is introduced, presents the general result. Each of the
paragraphs in Section~\ref{sec:part-cauchy-probl} is devoted to a
particular evolution equation: its well posedness is proved obtaining
those estimates that allow the application of
Theorem~\ref{thm:Metric}. Specific models are then dealt with in
Section~\ref{sec:coupled-problems}. Finally, proofs are in the final
Section~\ref{sec:technical-details}.

\section{Definitions and Abstract Results}
\label{sec:defin-abstr-results}

Below we rely on the framework established in~\cite{BressanCauchy,
  zbMATH05509403, Panasyuk}, see~\cite{MR1230367, MR1712751,
  LorenzBook} for an alternative, essentially equivalent, setting. Let
$(X,d)$ be a metric space and $I$ be a real interval. First, a
\emph{local flow} on $X$ provides a sort of tangent vector field to
$X$.

\begin{definition}[{\cite[Definition~2.1]{zbMATH05509403}}]
  \label{def:local}
  Given $\delta>0$ and a closed set ${\mathcal D} \subseteq X$, a
  \emph{local flow} is a continuous map
  $F \colon [0,\delta] \times I \times {\mathcal D} \mapsto X$, such
  that $F\left(0,t_o\right)u=u$ for any
  $(t_o,u)\in I\times {\mathcal D}$ and which is Lipschitz in its
  first and third arguments uniformly in the second, i.e.~there exists
  a $\Lip(F) > 0$ such that for all $\tau,\tau' \in [0,\delta]$ and
  $u,u' \in {\mathcal D}$
  \begin{equation}
    d \left( F(\tau,t_o)u, F(\tau',t_o) u' \right)
    \leq
    \Lip(F) \cdot \left( d(u,u') + \modulo{\tau - \tau'}\right) \,.
  \end{equation}
\end{definition}

Given an evolution equation, a \emph{global process} is a candidate
for the solution operator, i.e., for the mapping assigning to initial
datum $u$ at time $t_o$ and to time $t$ the solution evaluated at time
$t$.

\begin{definition}[{\cite[Definition~2.5]{zbMATH05509403}}]
  \label{def:Global}
  Fix a family of sets ${\mathcal D}_{t_o}\subseteq{\mathcal D}$ for
  all $t_o\in I$, and a set
  \begin{equation}
    \label{eq:10}
    {\mathcal A}
    =
    \left\{
      (t,t_o,u) \colon t \geq t_o,\  t_o,t \in I
      \mbox{ and }u \in {\mathcal D}_{t_o}
    \right\}.
  \end{equation}
  A \emph{global process} on $X$ is a map
  $P \colon {\mathcal A} \mapsto X$ such that, for all
  $u \in {\mathcal D}_{t_o}$ and $t_o,t_1,t_2 \in I$ with
  $t_2 \geq t_1 \geq t_o$,
  \begin{eqnarray}
    \label{eq:process0}
    & &
        P(t_o,t_o) u = u
    \\
    \label{eq:process1}
    & &
        P(t_1,t_o)u\in{\mathcal D}_{t_1}
    \\
    \label{eq:process2}
    & &
        P ( t_2, t_1) \circ P(t_1,t_o)u
        =
        P(t_2,t_o)u.
  \end{eqnarray}
\end{definition}

In Theorem~\ref{thm:main} below, a global process is constructed from
a local flow by means of a suitable extension of \emph{Euler
  Polygonals} to metric spaces.

\begin{definition}[{\cite[Definition~2.3]{zbMATH05509403}}]
  \label{def:Euler-polygonals}
  Let $F$ be a local flow. Fix $u \in {\mathcal D}$, $t_o \in I$,
  $\tau \in [0,\delta]$ with $t_o+\tau \in I$. For every
  $\varepsilon > 0$, let $k = \lfloor\tau / \varepsilon\rfloor$, where
  the symbol $\lfloor \cdot \rfloor$ denotes the integer part. An
  Euler $\varepsilon$-polygonal is
  \begin{equation}
    \label{eq:polygonal}
    F^\varepsilon (\tau,t_o) \, u
    =
    F(\tau-k\varepsilon, t_o+k\varepsilon) \circ
    \comp_{h=0}^{k-1} F(\varepsilon, t_o+h\varepsilon) \, u
  \end{equation}
  whenever it is defined.
\end{definition}

\noindent Above, we used the notation
$\comp_{h=0}^{k} f_h = f_k \circ f_{k-1} \circ \ldots \circ
f_1 \circ f_0$.

For a local flow $F$, its corresponding Euler $\varepsilon$-polygonal
$F^\varepsilon$, and any $t_o \in I$, introduce the notation:
\begin{equation}
  \label{def:D3}
  \!\!\!
  {\mathcal D}^3_{t_o}
  =
  \left\{
    u\in{\mathcal D} \colon \!
    \begin{array}{l}
      F^{\varepsilon_3} (\tau_3,t_o+\tau_1+\tau_2) \circ
      F^{\varepsilon_2} (\tau_2,t_o+\tau_1) \circ
      F^{\varepsilon_1} (\tau_1,t_o)u
      \\
      \mbox{is in } {\mathcal D}
      \mbox{ for all }
      \varepsilon_1, \varepsilon_2, \varepsilon_3 \in
      \left]0,\delta \right]
      \mbox{ and all}
      \\ \tau_1, \tau_2, \tau_3\geq 0 \mbox{ such that }
      t_o+\tau_1+\tau_2+\tau_3\in I
    \end{array}\!
  \right\}.
\end{equation}

The next result provides the basis for our construction of solutions
to coupled problems.

\begin{theorem}[{\cite[Theorem~2.6]{zbMATH05509403}}]
  \label{thm:main}
  Let $(X,d)$ be a complete metric space and $\mathcal{D}$ be a closed
  subset of $X$. Assume that for the local flow
  $F \colon [0,\delta] \times I \times {\mathcal D} \mapsto X$ there
  exist
  \begin{enumerate}
  \item \label{it:first} a non decreasing map
    $\omega \colon [0,\delta] \to \reali_+$ with
    $\int_0^{\delta} \frac{\omega(\tau)}{\tau} \, d\tau < +\infty$
    such that
    \begin{equation}
      \label{eq:k}
      d \left(
        F ( k \tau, t_o+\tau) \circ F(\tau,t_o) u,
        F \left( (k+1) \tau, t_o \right)u
      \right)
      \leq
      k\, \tau \; \omega(\tau)
    \end{equation}
    whenever $\tau \in [0,\delta]$, $k\in \naturali$ and the left hand
    side above is well defined;
  \item \label{it:second} a positive constant $L$ such that
    \begin{equation}
      \label{eq:Stability}
      d \left( F^\varepsilon (\tau,t_o) u_1, F^\varepsilon (\tau,t_o) u_2 \right)
      \leq L \; d(u_1,u_2)
    \end{equation}
    whenever $\varepsilon \in \left]0, \delta\right]$,
    $u_1, u_2 \in {\mathcal D}$, $\tau\geq 0$, $t_o,t_o + \tau \in I$
    and the left hand side above is well defined.
  \end{enumerate}
  \noindent Then, there exists a family of sets ${\mathcal D}_{t_o}$,
  for $t_o \in I$, and a unique global process (as in
  Definition~\ref{def:Global}) $P \colon {\mathcal A} \to X$ with the
  following properties:
  \begin{enumerate}
  \item ${\mathcal D}^3_{t_o}\subseteq{\mathcal D}_{t_o}$ for any
    $t_o\in I$, with ${\mathcal D}^3_{t_o}$ as defined
    in~\eqref{def:D3};
  \item $P$ is Lipschitz continuous with respect to
    $(t,t_o,u) \in \mathcal{A}$;
  \item
    \label{fgsfssald3}
    $P$ is tangent to $F$ in the sense that for all
    $(t_o+\tau,t_o,u) \in {\mathcal A}$, with
    $\tau \in \left]0, \delta\right]$:
    \begin{equation}
      \label{eq:tangent}
      \frac{1}{\tau} \;
      d\left( P(t_o+\tau,t_o)u, F(\tau,t_o)u \right)
      \leq
      \frac{2L}{\ln (2)}
      \int_0^\tau \frac{\omega(\xi)}{\xi} \, \d{\xi} \,.
    \end{equation}
  \end{enumerate}
\end{theorem}

\noindent A general condition to ensure that $\mathcal{A}$ is non
empty is~\cite[Condition~\textbf{(D)}]{zbMATH05509403}. Below, in the
examples we consider, it explicitly stems out that
$\mathcal{A} \neq \emptyset$.

We now head towards considering processes depending on parameters.

\begin{definition}
  \label{def:LipProc}
  Let $(\mathcal{U}, d_{\mathcal{U}})$ and
  $(\mathcal{W},d_{\mathcal{W}})$ be metric spaces. A \emph{Lipschitz
    Process on $\mathcal{U}$ para\-metrized by $w \in \mathcal{W}$} is
  a family of maps $P^w \colon \mathcal{A}_{\mathcal{U}} \to \mathcal{U}$, with
  \begin{eqnarray*}
    \mathcal{I}
    & =
    & \left\{(t,t_o) \in I \times I \colon t \geq t_o\right\} \,,
    \\
    \Doa
    & =
    & \left\{
      (t, t_o, \vra) \colon
      (t,t_o) \in \mathcal{I} \,,\; \vra \in \doa_{t_o}
      \right\} \,,
    \\
    \mathcal{D}_t^{\mathcal{U}}
    & \subseteq
    & \mathcal{U} \,,
  \end{eqnarray*}
  such that for all $w \in \mathcal{W}$, $P^w$ is a Global Process in
  the sense of Definition~\ref{def:Global} and there exist positive
  constants $C_u,C_t,C_w$ such that
  \begin{eqnarray}
    \label{eq:6}
    d_{\mathcal{U}}\left(P^{w}{(t,t_o)} u_1 , P^{w}{(t,t_o)} u_2\right)
    & \leq
    & e^{C_u (t-t_o)}\; d_{\mathcal{U}}(u_1, u_2) \,,
    \\
    \label{eq:7}
    d_{\mathcal{U}}\left(P^{w}{(t_1,t_o)} u , P^{w}{(t_2,t_o)} u\right)
    & \leq
    & C_t \; \modulo{t_2 - t_1} \,,
    \\
    \label{eq:8}
    d_{\mathcal{U}}\left(P^{w_1}{(t,t_o)} u_o, P^{w_2}{(t,t_o)} u_o\right)
    & \leq
    & C_w \, (t-t_o) \; d_{\mathcal{W}}(w_1,w_2) \,.
  \end{eqnarray}
\end{definition}

We equip the product space $\mathcal{U} \times \mathcal{W}$ with the
distance
\begin{displaymath}
  d\left((u',w'), (u'',w'')\right)
  =
  d_{\mathcal{U}} (u',u'') + d_{\mathcal{W}} (w',w'').
\end{displaymath}

\begin{theorem}
  \label{thm:Metric}
  Let $(\mathcal{U},d_{\mathcal{U}})$ and
  $(\mathcal{W},d_{\mathcal{W}})$ be complete. Let
  $ P^w \colon \mA_\spa \to \spa $ be a Lipschitz Process on
  $\mathcal{U}$ parametrized by $w \in \mathcal{W},$ and let
  $P^u \colon \mA_\spb \to \spb $ be a Lipschitz Process on
  $\mathcal{W}$ parametrized by $\mathcal{U}$. Let $C_u,C_w,$ and $C_t$ be
  constants that satisfy~\eqref{eq:6}--\eqref{eq:7}--\eqref{eq:8} for
  both processes.  Then,
  \begin{enumerate}
  \item Introducing
    $\mathcal{A}_F = \left\{ \left(\tau,t_o,(u,w)\right) \colon \tau
      \geq 0,\ t_o,t_o+\tau \in I ,\,(u,w) \in {\mathcal
        D}_{t_o}^\mathcal{U} \times {\mathcal D}_{t_o}^{\mathcal{W}}
    \right\}$, the map
    \begin{equation}
      \label{eq:36}
      \begin{array}{ccccc}
        F
        & \colon
        &    \mathcal{A}_F
        & \to
        & \spa \times \spb
        \\
        &
        & \left(\tau, t_o, (u, w) \right)
        & \mapsto
        & \left(P^w(t_o+\tau, t_o)u, P^u(t_o+\tau, t_o) w \right)
      \end{array}
    \end{equation}
    is a local flow on $\mathcal{U}\times\mathcal{W}$.
  \item $F$ satisfies the assumptions of Theorem~\ref{thm:main} with
    \begin{equation}
      \label{eqn:tangM}
      L
      =
      e^{(C_u + C_w) T}
      \quad \mbox{ and } \quad
      \omega (\tau) = C_t \, C_u \, \tau
    \end{equation}
    hence $F$ generates a unique global process
    $P \colon \mathcal{A} \to \mathcal{U} \times \mathcal{W}$, for a
    suitable
    $\mathcal{A} \subseteq I \times I \times
    \mathcal{U}\times\mathcal{W}$, satisfying properties 1.,
    2.~and~3.~in Theorem~\ref{thm:main}.

  \item For all $t_o \in I$ and $\tau \geq 0$ with
    $t_o+\tau \geq t_o$, we have
    \begin{equation}
      \label{eq:77}
      F (\tau,t_o)
      (\mathcal{D}_{t_o}^{\mathcal{U}} \times \mathcal{D}_{t_o}^{\mathcal{W}})
      \subseteq
      (\mathcal{D}_{t_o+\tau}^{\mathcal{U}} \times \mathcal{D}_{t_o+\tau}^{\mathcal{W}})
    \end{equation}
    hence the process $P$ is defined on $\mathcal{A}$ with
    \begin{equation}
      \label{eq:78}
      \mathcal{A} \supseteq
      \left\{ \left(\tau,t_o,(u,w)\right) \colon \tau
        \geq 0,\ t_o,t_o+\tau \in I ,\,(u,w) \in {\mathcal
          D}_{t_o}^\mathcal{U} \times {\mathcal D}_{t_o}^{\mathcal{W}}
      \right\} \,.
    \end{equation}
  \end{enumerate}

\end{theorem}

\noindent The proof is deferred to~\S~\ref{subsec:proofs-sect-refs}.

An analogous result can be proved defining the local flow $F$ by means
of local flows $F_w^\spa$ and $F_u^\spb$, provided these local flows
satisfy the assumptions of Theorem~\ref{thm:main} and have a Lipschitz
continuous dependence on the parameter.

\begin{theorem}
  \label{thm:Metric2}
  Consider two complete metric spaces
  $ (\mathcal{U}, d_{\mathcal{U}})$ and
  $ (\mathcal{W}, d_{\mathcal{W}}) $. Let
  \begin{displaymath}
    F^w \colon [0, \delta] \times I \times \mathcal{D}^\mathcal{U} \to \mathcal{U},
    \quad \mbox{ and } \quad
    F^u \colon [0, \delta] \times I \times \mathcal{D}^\mathcal{W} \to \mathcal{W} \,,
  \end{displaymath}
  be local flows parametrized by $w \in \mathcal{W}$ and
  $u \in \mathcal{U}$, respectively, so that there exists
  $\mathcal{L}$ such that for all $ \tau \in [0,\delta]$ and
  $t \in I$,
  \begin{displaymath}
    \begin{array}{rcl@{\qquad}r@{\,}c@{\,}l@{\qquad}r@{\,}c@{\,}l}
      \displaystyle
      d_{\mathcal{U}}\left(
      F^{w_1} (\tau,t)u,
      F^{w_2} (\tau,t)u
      \right)
      & \leq
      & \mathcal{L} \; d_{\mathcal{W}}(w_1,w_2)
      & u
      & \in
      & \mathcal{D}^{\mathcal{U}}
      & w_1,w_2
      & \in
      & \mathcal{W}
      \\
      \displaystyle
      d_{\mathcal{W}}\left(
      F^{u_1} (\tau,t)w,
      F^{u_2} (\tau,t)w
      \right)
      & \leq
      & \mathcal{L} \; d_{\mathcal{U}}(u_1,u_2)
      & u
      & \in
      & \mathcal{D}^{\mathcal{W}}
      & u_1,u_2
      & \in
      & \mathcal{U}
    \end{array}
  \end{displaymath}
  Then, setting
  $\mathcal{D} =
  \mathcal{D}^{\mathcal{U}}\times\mathcal{D}^{\mathcal{W}}$, the
  coupling
  \begin{displaymath}
    \begin{array}{ccccc}
      \hat F
      & \colon
      & [0,\delta] \times I \times \mathcal{D}
      & \to
      & \mathcal{U} \times \mathcal{W}
      \\
      &
      & \left(\tau,t,(u,w)\right)
      & \mapsto
      & \left(F^w(t, t_o)u, F^u(t, t_o)w\right)
    \end{array}
  \end{displaymath}
  is a local flow in the sense of Definition~\ref{def:local}.  If
  moreover $F^w$ and $F^u$ satisfy assumptions~\ref{it:first}
  and~\ref{it:second} in Theorem~\ref{thm:main}, then $\hat F$ is
  tangent to the local flow $F$ defined in~\eqref{eq:36} by means of
  the processes $P^w$ and $P^u$ defined through
  Theorem~\ref{thm:main}.
\end{theorem}

As a direct consequence of Theorem~\ref{thm:Metric2}, by means
of~\cite[Theorem~2.9]{BressanLectureNotes}, we have that whenever
Theorem~\ref{thm:Metric} applies, if $\hat F$ generates a global
process $\hat P$, then $\hat P$ coincides with the process $P$
constructed in Theorem~\ref{thm:Metric}.

\section{General Cauchy Problems}
\label{sec:part-cauchy-probl}

In the paragraphs below we consider differential equations depending
on parameters that generate parametrized Lipschitz processes in the
sense of Definition~\ref{def:LipProc}. Thus, any coupling of the
processes below meets the requirements of Theorem~\ref{thm:Metric} and
generates a new Lipschitz process. Moreover, we verify that this new
process eventually yields solutions to the coupled problem.

Throughout, $\hat I$ is a real interval containing $0$. If
$x \in \reali^n$, $\norma{x}$ denotes its Euclidean norm, while
$\norma{x}_V$ is the norm of $x$ in the Banach space $V$. The open,
respectively closed, ball centered at $x$ with radius $r$ is
$B (x,r)$, respectively $\overline{B (x,r)}$.

\subsection{Ordinary Differential Equations}
\label{subsec:ordin-diff-equat}

This brief paragraph mainly serves as a paradigm for the subsequent
ones. Indeed, we begin by considering the classical Cauchy problem
for an ordinary differential equation
\begin{equation}
  \label{eq:21}
  \left\{
    \begin{array}{l@{\qquad\qquad}l}
      \dot u = f (t,u,w)
      & t \in \hat I
      \\
      u (t_o) = u_o
    \end{array}
  \right.
  \qquad \mbox{ with } \quad
  f \colon \hat I \times \reali^n \times \mathcal{W} \to \reali^n \,,
\end{equation}
where $t_o \in \hat I$, $u_o \in \reali^n$ and the parameter $w$ is
fixed in $\mathcal{W}$.

\begin{definition}
  \label{def:ODE}
  A map $u \colon I \to \reali^n$ is a solution to~\eqref{eq:21} if
  $t_o \in I \subseteq \hat I$, $u (t_o) = u_o$, for a.e.~$t \in I$,
  $u$ is differentiable at $t$ and
  $\dot u (t) = f\left(t, u (t), w\right)$.
\end{definition}

The well posedness of~(\ref{eq:21}) is an elementary result which we
state below to allow subsequent couplings of~(\ref{eq:21}) with other
equations within the framework of Theorem~\ref{thm:Metric}.

\begin{proposition}
  \label{prop:ODE}
  Let $R>0$. Define $\mathcal{D} = \overline{B (0, R)}$ in
  $\reali^n$ and consider the Cauchy problem~\eqref{eq:21} under the
  assumptions
  \begin{enumerate}[label=\bf{(ODE\arabic*)},
    ref=\textup{\textbf{(ODE\arabic*)}}, left=0pt]
  \item\label{it:ODE1} For all $u \in \mathcal{D}$ and all
    $w \in \mathcal{W}$, the map $t \mapsto f(t, u , w )$ is
    measurable.
  \item\label{it:ODE2} There exist positive $F_L$, $F_\infty$ such
    that for all $t \in \hat I$, $u_1,u_2 \in \mathcal{D}$ and
    $w_1,w_2 \in \mathcal{W}$
    \begin{eqnarray}
      \label{eq:22}
      \norma{f (t,u_1,w_1) - f (t,u_2,w_2)}
      & \leq
      & F_L  \left(
        \norma{u_1 - u_2}
        +
        d_{\mathcal{W}}(w_1,w_2)
        \right) \,;
      \\
      \label{eq:9}
      \sup_{w \in \mathcal{W}}
      \norma{f(\cdot, \cdot, w)}_{\L\infty(\hat I\times \hat {\mathcal D}; \reali^n)}
      & \leq
      & F_\infty \,.
    \end{eqnarray}
  \end{enumerate}
  \noindent Then, there exists $T>0$, such that
  $[0,T] \subseteq \hat I$, and a Lipschitz process on $\reali^n$
  pametrized by $\mathcal{W}$ in the sense of
  Definition~\ref{def:LipProc}, whose orbits solve~\eqref{eq:21}
  according to Definition~\ref{def:ODE}, with
  \begin{equation}
    \label{eq:27}
    \begin{array}{c}
      T \leq \left. R \middle/ (2 F_\infty)\right.
      \,,\quad
      C_u = F_L
      \,,\quad
      C_t = F_\infty
      \,,\quad
      C_w = F_L \, e^{F_L T}\,,
      \\[6pt]
      \mathcal{D}_t
      =
      \overline{B \left(
      0,
      R
      - (T - t) \, \sup_{w \in \mathcal{W}}
      \norma{f(\cdot, \cdot, w)}
      _{\L\infty(\hat I \times \hat{\mathcal{D}}; \reali^n)}
      \right)} \,.
    \end{array}
  \end{equation}
\end{proposition}

\noindent Long time existence is also available.

\begin{corollary}
  \label{cor:global-existence-ODE}
  Assume $\sup \hat I = +\infty$ and that, for every $R > 0$,
  \ref{it:ODE1} and \ref{it:ODE2} hold with $F_\infty = F_\infty(R)$
  satisfying
  \begin{equation*}
    \limsup_{R \to + \infty}\, \frac{F_\infty(R)}{R \ln(R)} < + \infty.
  \end{equation*}
  Then, for all $t_o \in \hat I$, the solution to~\eqref{eq:21} exists
  for every $t \geq t_o$.
\end{corollary}
\noindent The proof is deferred to \S~\ref{sub:ODE}.
We now verify that Theorem~\ref{thm:Metric} applies
to the coupling of~\eqref{eq:21} with other Lipschitz Processes.

\begin{proposition}
  \label{prop:ODEcoupled}
  Set $\mathcal{U} = \reali^n$. Assume
  that~\ref{it:ODE1}--\ref{it:ODE2} hold. Let $P^u$ be a Lipschitz
  Process on $\mathcal{W}$ parametrized by $u \in \mathcal{U}$. Call
  $P\colon \mathcal{A} \to \reali^n \times \mathcal{W}$, with
  $P \equiv (P_1,P_2)$, the Process constructed in
  Theorem~\ref{thm:Metric} coupling $P^w$, generated by~\eqref{eq:21},
  and $P^u$. If $([t_o,T],t_o,u_o,w_o) \subseteq \mathcal{A}$, then
  \begin{displaymath}
    \begin{array}{@{}c@{\,}c@{\,}c@{\,}c@{\,}c@{}}
      u
      & \colon
      & [t_o,T]
      & \to
      & \reali^n
      \\
      &
      &t
      & \mapsto
      & P_1 (t,t_o) (u_o , w_o)
    \end{array}
    \mbox{ solves }
    \left\{
      \begin{array}{l}
        \dot u
        =
        \bar f (t,u)
        \\
        u (t_o) = u_o
      \end{array}
    \right.
    \mbox{ where }
    \bar f (t,u) = f\left(t, u, P_2 (t,t_o) (u_o,w_o)\right)
  \end{displaymath}
  in the sense of Definition~\ref{def:ODE}.
\end{proposition}
\noindent The proof is deferred to \S~\ref{sub:ODE}.

A particular case of Proposition~\ref{prop:ODE} of interest is the
following.

\begin{corollary}
  \label{cor:ODE}
  Let $R>0$. Define $\hat{\mathcal{D}} = \overline{B (0, R)}$
  in $\mathcal{U} = \reali^n$. Choose
  $\mathcal{W} = \L1 (\reali^N;\reali^M)$ and fix
  $\eta \in \L\infty (\hat I \times \reali^N;\reali)$. Consider the
  Cauchy problem~\eqref{eq:21} with
  \begin{equation}
    \label{eq:11}
    f(t, u, w)
    =
    g\left(t, u, \int_{\reali^N} \eta (t,x) \; w (x) \d{x}\right)
  \end{equation}
  under the assumptions:
  \begin{enumerate}[label=\bf{(NL\arabic*)},
    ref=\textup{\textbf{(NL\arabic*)}}, left=0pt]
  \item \label{ip:(g)} For all $u \in \hat{\mathcal{D}}$ and
    $W \in \reali^M$, the map $t \mapsto g (t,u,W)$ is measurable.

  \item \label{ip:(g2)} There exist positive $L_g$ and $G_\infty$ such
    that for all $t \in \hat I$, $u_1,u_2 \in \hat{\mathcal{D}}$ and
    $W_1,W_2 \in \reali^M$
    \begin{eqnarray*}
      \norma{g (t,u_1,W_1) - g (t,u_2,W_2)}
      & \leq
      & G_L \left(\norma{u_1 - u_2} + \norma{W_1 - W_2}\right) \,;
      \\
      \sup_{\hat I \times \hat{\mathcal{D}} \times \reali^M} \norma{g (t,u,W)}
      & \leq
      & G_\infty \,.
    \end{eqnarray*}
  \end{enumerate}

  \noindent Then, given the interval $I=[0, T]$ with
  $T=\frac{R}{2 G_\infty}$ and, for every $t \in I$, the domain
  \begin{equation}
    \label{eq:dom-t-ODE}
    \mathcal{D}_t
    =
    \overline{B \left(0, R - (T - t) \norma{g}
        _{\L\infty(\hat I \times \hat{\mathcal{D}} \times \reali^M; \reali^n)}
      \right)} \,,
  \end{equation}
  problem~\eqref{eq:21}--\eqref{eq:11} generates a Lipschitz Process
  on $\reali^n$ pametrized by $w \in \mathcal{W}$, with constants
  in~\eqref{eq:6}--\eqref{eq:7}--\eqref{eq:8} given by
  \begin{equation}
    \label{eq:28}
    \begin{array}{c}
      C_u =
      G_L (1+\norma{\eta}_{\L\infty (\hat I\times\reali^N;\reali)})
      \,,\quad
      C_t = G_\infty
      \\
      C_w
      =
      G_L
      (1+\norma{\eta}_{\L\infty (\hat I\times\reali^N;\reali)})
      \exp \left(
      G_L
      (1+\norma{\eta}_{\L\infty (\hat I\times\reali^N;\reali)})
      \, \hat T
      \right).
    \end{array}
  \end{equation}
\end{corollary}

\noindent The proof is a direct consequence of
Proposition~\ref{prop:ODE} and is hence omitted. Note that also
Proposition~\ref{prop:ODEcoupled} is immediately extended to the case
of~\eqref{eq:11}.  The analog of \Cref{cor:global-existence-ODE} in
this setting is given by the following result, whose proof is omitted,
since it is identical to that of \Cref{cor:global-existence-ODE}.
\begin{corollary}
  \label{cor:global-existence-ODE-2}
  Assume $[0, +\infty) \subseteq \hat I$ and that, for every $R > 0$,
  \ref{ip:(g)} and \ref{ip:(g2)} hold with $G_\infty = G_\infty(R)$
  satisfying
  \begin{equation*}
    \limsup_{R \to + \infty}\, \frac{G_\infty(R)}{R \ln(R)} < + \infty.
  \end{equation*}
  Then the solution to~\eqref{eq:21}, with vector field~\eqref{eq:11},
  exists for every $t \geq t_o$.
\end{corollary}

\subsection{The Initial Value Problem for a Renewal Equation}
\label{subsec:IVP}

We examine the following initial value problem for a first order
partial differential equation
\begin{equation}
  \label{eq:REN}
  \left\{
    \begin{array}{l@{\qquad}r@{\,}c@{\,}l}
      \partial_t u + \div_x \left( v(t, x, w) \, u\right)
      = m(t, x, w)u + q(t, x, w)
      & (t,x)
      & \in
      & \hat I \times \reali^n,
      \\
      u (t_o, x) = u_o(x),
      &x
      & \in
      & \reali^n
    \end{array}
  \right.
\end{equation}
for $ u_o \in \L1 (\reali^n; \reali)$ and $t_o \in \hat I$. Proofs are
deferred until \S~\ref{subsubsec:proofs-related-s-IP}.

\begin{definition}
  \label{def:RENSol}
  For a fixed $w \in \mathcal{W}$, a function
  $u \in \C0\left([t_o, T]; \L1(\reali^n; \reali)\right)$, where
  $[t_o,T] \subseteq \hat I$, is a solution to~\eqref{eq:REN} if:
  \begin{enumerate}
  \item for any test function
    $\phi \in \Cc\infty (\mathopen]t_o, T\mathclose[ \times \reali^n;
    \reali)$,
    \begin{eqnarray*}
      \int_{t_o}^{T} \int_{\reali^n}
      \bigl(
      u (t,x) \, \partial_t\phi (t,x)
      +
      u (t,x)\, v (t, x, w) \cdot \nabla_x \phi (t,x)
      \qquad\qquad
      \\
      \qquad\qquad
      +
      \left(m (t,x,w)\, u (t,x) + q (t,x,w) \right)
      \phi (t,x)
      \bigr)\,
      \d{x}  \, \d{t} = 0;
    \end{eqnarray*}
  \item $u(t_o,x) = u_o(x)$ for a.e. $x \in \reali^n$.
  \end{enumerate}
\end{definition}

\begin{proposition}
  \label{prop:RE}
  Let $R>0$ and set $\mathcal{U} = \L1 (\reali^n;
  \reali)$. Define
  \begin{displaymath}
    \mathcal{D} =
    \left\{
      u \in \L1 (\reali^n; \reali)
      \colon
      \max
      \left\{
        \norma{u}_{\L1(\reali^n; \reali)},
        \norma{u}_{\L\infty(\reali^n; \reali)},
        \tv(u)
      \right\} \leq R \right\} \,.
  \end{displaymath}
  Consider the Cauchy problem~\eqref{eq:REN} under the assumptions
  \begin{enumerate}[label=\bf{(IP\arabic*)},
    ref=\textup{\textbf{(IP\arabic*)}}, left=0pt]
  \item \label{ip:(V)} For all $ w \in \mathcal{W} $,
    $v(\cdot, \cdot, w) \in \C0 (\hat I\times\reali^n;\reali^n)$,
    $v(t,\cdot, w) \in \C2(\reali^n;\reali^n)$ for all $t \in \hat I$
    and there exist positive constants $V_1$, $V_L$, $V_\infty$ such
    that for all $t \in \hat I$
    \begin{displaymath}
      \begin{array}{c}
        \norma{v(t, \cdot, w)}_{\L\infty(\reali^n;\reali^n)}
        \leq
        V_\infty \,;
        \qquad
        \norma{\nabla v(t, \cdot, w)}_{\L\infty(\reali^n;\reali^{n\times n})}
        \leq
        V_L \,;
        \\[6pt]
        \norma{\nabla \nabla \cdot v (t, \cdot, w)}_{\L1 (\reali^n; \reali^n)}
        \leq
        V_1 \,.
      \end{array}
    \end{displaymath}
    and, for all $w_1, w_2 \in \mathcal{W}$ and $t\in \hat{I}$,
    \begin{align*}
      \norma{v(t, \cdot, w_1) - v(t, \cdot, w_2)}_{\L\infty(\reali^n;\reali^n)}
      & \leq
        V_L \, d_{\mathcal{W}}(w_1, w_2),
      \\
      \norma{\nabla \cdot \left(v(t, \cdot, w_1) - v(t, \cdot, w_2)\right)}_{\L1(\reali^n;\reali)}
      &\leq
        V_L \, d_{\mathcal{W}}(w_1, w_2).
    \end{align*}

  \item \label{ip:(P)} For all $w \in \mathcal{W}$,
    $m (\cdot, \cdot,w) \in \C0 (\hat I \times \reali^n; \reali)$ and
    there exist positive constants $M_\infty$, $M_L$ such that for all
    $t \in \hat I$ and for all $w, w_1, w_2 \in \mathcal{W}$
    \begin{eqnarray*}
      \norma{m(t, \cdot, w)}_{\L\infty(\reali^n;\reali)}
      +
      \tv\left( m(t, \cdot, w) \right)
      & \leq
      & M_\infty \,;
      \\
      \norma{m(t, \cdot, w_1) - m(t, \cdot, w_2)}_{\L1(\reali^n;\reali)}
      & \leq
      & M_L \; d_{\mathcal{W}}(w_1, w_2) \,.
    \end{eqnarray*}

  \item \label{ip:(Q)} For all $w \in \mathcal{W}$,
    $q (\cdot,\cdot,w) \in \L1 \left(\hat I; \L\infty (\reali^n;
      \reali)\right)$ and there exist positive constants $Q_\infty$,
    $Q_1$, $Q_L$ such that for all $t \in \hat I$ and for all
    $w,w_1,w_2 \in \mathcal{W}$,
    \begin{eqnarray*}
      \norma{q(t, \cdot, w)}_{\L\infty(\reali^n;\reali)}
      +
      \tv\left(q(t,\cdot, w)\right)
      & \leq
      & Q_\infty \,;
      \\
      \norma{q(t, \cdot, w)}_{\L1(\reali^n; \reali)}
      & \leq
      & Q_1,
      \\
      \norma{q(t, \cdot, w) - q(t, \cdot, w_2)}_{\L1(\reali^n;\reali)}
      & \leq
      & Q_L \; d(w_1, w_2) \,.
    \end{eqnarray*}
  \end{enumerate}
  Then, there exists $T>0$, such that $[0,T] \subseteq \hat I$, and a
  Lipschitz process on $\mathcal{U}$ pametrized by $\mathcal{W}$ in
  the sense of Definition~\ref{def:LipProc}, whose orbits
  solve~\eqref{eq:REN} in the sense of Definition~\ref{def:RENSol},
  with
  \begin{equation}
    \label{eq:29}
    \begin{array}{c}
      C_u = M_\infty \,,\quad
      C_t =     V_\infty \, R \, e^{(M_\infty+2V_L)  T}
      + Q_1 \, e^{M_\infty  T}
      + (M_\infty+V_L) \, R \, e^{(M_\infty+V_L)  T}\,,
      \\
      C_w
      =
      \left[
      V_L
      (2R + Q_\infty)
      (1 + (V_1 + M_\infty){T})
      + (
      Q_L
      + (M_L + V_L) (R + Q_\infty {T})
      )
      \right]
      e^{(M_\infty + V_L)  T} \,,
      \\
      \mathcal{D}_t =
      \left\{
      u \in \mathcal{D}
      \colon
      \begin{array}{rcl}
	\norma{u}_{\L1(\reali^n; \reali)}
        & \leq
        & \alpha_1(t)
        \\
        \norma{u}_{\L\infty(\reali^n; \reali)}
        & \leq
        & \alpha_\infty(t)
        \\
        \tv(u)
        & \leq
        & \alpha_{\tv}(t)
      \end{array}
          \right\} \,,
    \end{array}
  \end{equation}
  where
  \begin{equation}
    \label{eq:65}
    \begin{array}{rcl}
      \alpha_1(t)
      & =
      & Re^{-M_\infty (T - t)} - Q_1(T- t)e^{M_\infty t} \,,
      \\
      \alpha_\infty(t)
      & =
      & R e^{-(M_\infty+V_L)(T-t)} - Q_\infty e^{(M_\infty+V_L)t} (T - t) \,,
      \\
      \alpha_{\tv}(t)
      & =
      & R e^{-(M_\infty+V_L)(T-t)} \left(1 - (M_\infty + V_1)(T - t)\right)
      \\
      &
      & - Q_\infty e^{(M_\infty+V_L)t}\left(1 + (M_\infty+V_1)t\right)(T- t) \,.
    \end{array}
  \end{equation}
\end{proposition}

\begin{corollary}
  \label{cor:global-existence-IVP}
  Assume $[0, +\infty) \subseteq \hat I$ and that \ref{ip:(V)},
  \ref{ip:(P)}, and~\ref{ip:(Q)} hold.  Then the solution
  to~\eqref{eq:REN} exists for every $t \geq t_o$.
\end{corollary}

Continuing now to the act of coupling this Lipschitz process with
another.

\begin{proposition}
  \label{prop:RENCoup}
  Set $\mathcal{U} = \L1 (\reali^n; \reali)$. Assume
  that~\ref{ip:(V)}--\ref{ip:(P)}--\ref{ip:(Q)} hold. Let $P^u$ be a
  Lipschitz process on $\mathcal{W}$, parametrised by
  $u \in \mathcal{U}$. Call
  $P\colon \mathcal{A} \to \L1 (\reali^n; \reali) \times \mathcal{W}$,
  with $P \equiv (P_1,P_2)$, the process generated in
  Theorem~\ref{thm:Metric} by the coupling of process $P^w$, found in
  Proposition~\ref{prop:RE}, with $P^u$. If
  $([t_o,T],t_o,u_o,w_o) \subseteq \mathcal{A}$, then the map
  \begin{displaymath}
    \begin{array}{@{}c@{\,}c@{\,}c@{\,}c@{\,}c@{}}
      u
      & \colon
      & [t_o,T]
      & \to
      & (\L1 \cap \BV) (\reali^n; \reali)
      \\
      &
      &t
      & \mapsto
      & P_1 (t,t_o) (u_o , w_o)
    \end{array}
  \end{displaymath}
  solves
  \begin{displaymath}
    \left\{
      \begin{array}{l@{\qquad}l}
        \partial_t u + \div_x \left( \bar{v}(t, x) \, u\right)
        = \bar m(t, x)u + \bar q(t, x)
        & (t,x) \in [t_o,T] \times \reali^n,
        \\
        u (t_o, x) = u_o(x),
        &x \in \reali^n
      \end{array}
    \right.
  \end{displaymath}
  in the sense of Definition~\eqref{eq:REN}, where
  \begin{displaymath}
    \begin{array}{c}
      \bar m (t,x)
      =
      m \left(t,x,P_2 (t,t_o) (u_o,w_o)\right),
      \quad \quad
      \bar q (t,x)
      =
      q\left(t,x,P_2 (t,t_o) (u_o,w_o)\right) \,,
      \\
      \bar{v} (t, x)
      =
      v \left(t, x, P_2 (t, t_o) (u_o, w_o)\right) \,.
    \end{array}
  \end{displaymath}
\end{proposition}

\subsection{The Boundary Value Problem for a Linear Balance Law}
\label{subsec:IBVP2}

Consider the model
\begin{equation}
  \label{eq:IBVP2Eq}
  \left\{
    \begin{array}{@{}l@{\qquad\quad}r@{\,}c@{\,}l@{}}
      \partial_t u + \partial_x \left(v(t, x) \, u\right)
      =
      m (t,x,w) \, u + q (t,x,w)
      & (t,x)
      & \in
      & \hat I \times \reali_+
      \\
      u (t,0) = b (t)
      & t
      & \in
      & \hat I
      \\
      u (t_o,x) = u_o (x)
      & x
      & \in
      & \reali_+ \,.
    \end{array}
  \right.
\end{equation}
where $u_o \in \L1 (\reali_+; \reali)$, $t_o \in \hat I$ and $w \in
\mathcal{W}$. Throughout, we choose left continuous representatives of
$\BV$ functions. Proofs are deferred to~\S~\ref{subs:IBVP2}.

\begin{definition}
  \label{def:IBVP}
  For a fixed $w \in \mathcal{W}$, a function
  $u \in \C0 \left([t_o,T]; \L1 (\reali_+; \reali)\right)$, with
  $[t_o,T] \subseteq \hat I$, such that
  $u (t) \in \BV (\reali_+; \reali)$ for a.e. $t \in [t_o,T]$ is a
  solution to~\eqref{eq:IBVP2Eq} if:
  \begin{enumerate}
  \item\label{item:IBVPsol1} For all
    $\phi \in \Cc\infty (\mathopen]t_o,T\mathclose[ \times
    \pint{\reali}_+; \reali)$
    \begin{eqnarray*}
      \int_{t_o}^T \int_{\reali_+}
      \Bigl(
      u (t,x) \, \partial_t \phi (t,x)
      +
      v (t,x) \, u (t,x) \, \partial_x \phi (t,x)\qquad\qquad
      \\
      +
      \left(m (t,x,w) \, u(t,x) + q (t,x,w)\right) \, \phi (t,x)
      \Bigr)
      \d{x} \d{t}
      & =
      & 0 \,.
    \end{eqnarray*}
  \item\label{item:IBVPsol2} For a.e.~$x \in \reali_+$,
    $u (t_o,x) = u_o (x)$.
  \item\label{item:IBVPsol3} For a.e.~$t \in [t_o, T]$,
    $\lim_{x\to 0+} u (t,x) = b (t)$.
  \end{enumerate}
\end{definition}

\begin{proposition}
  \label{prop:IBVP2}
  Let $\mathcal{U} = \L1 (\reali_+; \reali)$ and fix
  $b \in \BV (\hat I; \reali)$. For $R >0$, define
  \begin{equation}
    \label{eq:66}
    \mathcal{D}
    =
    \left\{u \in \mathcal{U} \colon
      \max\left\{\norma{u}_{\L1(\reali_+; \reali)} ,\,
        \norma{u}_{\L\infty(\reali_+;\reali)} ,\,
        \tv(u) + \modulo{b(\sup \hat I) - u(0)}\right\} \leq R
    \right\} \,.
  \end{equation}
  Assume
  \begin{enumerate}[label=\bf{(BP\arabic*)},
    ref=\textup{\textbf{(BP\arabic*)}}, left=0pt]
  \item\label{item:IBVP1} There exist positive constants
    $\check v, \hat v, V_1, V_\infty$ such that for all
    $v \in \C{0,1} (\hat I\times \reali_+; [\check v, \hat v])$ and
    for all $(t,x) \in \hat I \times \reali_+$
    \begin{eqnarray*}
      \tv\left(v (\cdot, x); \hat I\right)
      +
      \tv\left(v (t, \cdot)\right)
      & \leq
      & V_\infty \,,
      \\
      \tv\left(\partial_x v (t, \cdot)\right)
      +
      \norma{\partial_x v (t, \cdot)}_{\L\infty (\reali_+; \reali)}
      & \leq
      & V_L\,.
    \end{eqnarray*}
  \item\label{item:IBVP2} For all $w \in \mathcal{W}$,
    $m (\cdot, \cdot,w) \in \C0(\hat I \times \reali_+; \reali)$ and
    there exist $M_\infty, M_L$ such that for all $t \in \hat I$,
    $w,w_1,w_2 \in \mathcal{W}$,
    \begin{eqnarray*}
      \tv\left(m (t,\cdot,w)\right)
      +
      \norma{m (t, \cdot,w)}_{\L\infty (\reali_+; \reali)}
      & \leq
      & M_\infty \,,
      \\
      \norma{m (t, \cdot, w_1) - m (t, \cdot, w_2)}_{\L1 (\reali_+; \reali)}
      & \leq
      & M_L \; d_{\mathcal{W}} (w_1,w_2) \,.
    \end{eqnarray*}

  \item\label{item:IBVP3} For all $w \in \mathcal{W}$,
    $q (\cdot, \cdot,w) \in \C0 \left(\hat I; \L1(\reali_+;
      \reali)\right)$ and there exist $Q_1, Q_\infty$ such that for
    all $t \in \hat I$ and $w,w_1,w_2 \in \mathcal{W}$, and
    \begin{eqnarray*}
      \norma{q (t, \cdot, w)}_{\L1 (\reali_+; \reali)}
      & \leq
      & Q_1 \,,
      \\
      \tv\left(q (t, \cdot, w)\right)
      +
      \norma{q (t, \cdot, w)}_{\L\infty (\reali_+; \reali)}
      & \leq
      & Q_\infty \,,
      \\
      \norma{q (t, \cdot, w_1) - q (t, \cdot, w_2)}_{\L1 (\reali_+; \reali)}
      & \leq
      & Q_L \; d_{\mathcal{W}} (w_1,w_2) \,.
    \end{eqnarray*}
  \item\label{item:IBVP4}
    $b \in (\L1 \cap \L\infty \cap \BV) (\hat I; \reali)$, is left
    continuous, and there exist positive constants $B_1$ and
    $B_\infty$ such that
    \begin{eqnarray*}
      \norma{b}_{\L1 (\hat I; \reali)}
      & \leq
      & B_1 \,,
      \\
      \tv(b)
      +
      \norma{b}_{\L\infty (\hat I; \reali)}
      & \leq
      & B_\infty \,.
    \end{eqnarray*}
  \end{enumerate}
  Then, there exists $R,T>0$, such that $[0,T] \subseteq \hat I$, and
  a Lipschitz process on $\mathcal{U}$, parametrized by $\mathcal{W}$
  in the sense of Definition~\ref{def:LipProc}, whose orbits
  solve~\eqref{eq:IBVP2Eq} in the sense of Definition~\ref{def:IBVP},
  with
  \begin{equation}
    \label{eq:IBVP2Const}
    \begin{array}{c}
      C_u = M_\infty
      \,,\quad
      C_t = [
      \hat{v}(B_1 + 2R + R(M_\infty + V_L)T)
      + M_\infty R
      + Q_1
      ]e^{M_\infty T}
      \,,\quad\\
      C_w
      =\left[
      B_\infty M_L
      + \hat v \, Q_L
      +\frac{1}{2} \hat v \, Q_\infty \, M_L \, T
      + M_L \, R
      + Q_L
      + \frac{1}{2} \, M_L \, Q_\infty \, T
      \right]
      e^{M_\infty T}\,,
      \\
      \mathcal{D}_t
      =
      \left\{r \in \mathcal{U} \colon
      \begin{array}{l}
        \norma{u}_{\L1(\reali_+; \reali)} \leq \alpha_1(t)\,,\;
        \norma{u}_{\L\infty(\reali_+;\reali)} \leq \alpha_\infty(t)\,,
        \\
        \tv(u) + \modulo{b(t) - u(0)} \leq \alpha_{TV}(t)
      \end{array}
      \right\}
    \end{array}
  \end{equation}
  where
  \begin{eqnarray}
    \nonumber
    \alpha_1 (t)
    & =
    & Re^{-M_\infty (T - t)}
      - (\hat{v} B_\infty + Q_1) (T - t) e^{M_\infty t}
    \\
    \nonumber
    \alpha_\infty (t)
    & =
    & R e^{-M_\infty (T - t)}
      - Q_\infty (T - t)
    \\
    \nonumber
    \alpha_{\tv} (t)
    & =
    & R \left(1 - (M_\infty + V_L)(T - t)\right) e^{(M_\infty + V_L)(T - t)}
    \\
    \nonumber
    &
    &
      -2 Q_\infty(1 + (M_\infty + V_L)t)(T - t)e^{(M_\infty + V_L)t}
    \\
    \nonumber
    &
    & - B_\infty(M_\infty + V_L)(T - t)e^{(M_\infty + V_L)t}
      - \tv(b; [t, T])e^{(M_\infty + V_L)t} \,.
  \end{eqnarray}
\end{proposition}

A result entirely analogous to
Corollary~\ref{cor:global-existence-IVP} can be proved also in the
case of~\eqref{eq:IBVP2Eq}.

\begin{proposition}
  \label{prop:coupleIBVP2}
  Set $\mathcal{U} = \L1 (\reali_+;\reali)$.
  Assume~\ref{item:IBVP1}--\ref{item:IBVP2}--\ref{item:IBVP3}--\ref{item:IBVP4}.
  Let $P^u$ be a Lipschitz process on $\mathcal{W}$, parametrised by
  $u \in \mathcal{U}$.  Set
  $P\colon \mathcal{A} \to \mathcal{U} \times \mathcal{W}$, with
  $P \equiv (P_1, P_2)$, to be the process generated in
  Theorem~\ref{thm:Metric} by the coupling of the process $P^w$,
  constructed in Proposition~\ref{prop:IBVP2}, with $P^u$.  If
  $\left(t, t_o, ( u_o, w_o )\right) \in \mathcal{A}$, then
  \begin{equation}
    \label{eq:IBVP2Compr}
    \begin{array}{@{}c@{\,}c@{\,}c@{\,}c@{\,}c@{}}
      u
      & \colon
      & [t_o,T]
      & \to
      & \L1(\reali_+;\reali)
      \\
      &
      &t
      & \mapsto
      & P_{1} (t,t_o) \left(u_o, w_o\right)
    \end{array}
  \end{equation}
  is a solution to
  \begin{equation}
    \label{eq:IBVP2CompProb}
    \left\{
      \begin{array}{@{}l@{\qquad\quad}r@{\,}c@{\,}l@{}}
        \partial_t u + \partial_x \left(v (t,x) \,  u\right)
        =
        \bar m (t,x) \, u + \bar q (t,x)
        & (t,x)
        & \in
        & [t_o,T] \times \reali_+
        \\
        u (t,0) = b (t)
        & t
        & \in
        & [t_o,T]
        \\
        u (t_o,x) = u_o (x)
        & x
        & \in
        & \reali_+
      \end{array}
    \right.
  \end{equation}
  in the sense of Definition~\ref{def:IBVP}, where
  \begin{equation}
    \label{eq:IBVP2CompProb2}
    \bar m (t,x)
    =
    m\left(t,x,P_2 (t,t_o)\left(u_o, w_o\right)\right) \,,
    \quad
    \bar q (t,x)
    =
    q\left(t,x,P_2 (t,t_o)\left(u_o, w_o\right)\right) \,.
  \end{equation}
\end{proposition}

\subsection{Measure Valued Balance Laws}
\label{subsec:meas-valu-balance}

Following~\cite{MR2871800}, consider the following measure valued
balance law
\begin{equation}
  \label{eq:30}
  \left\{
    \begin{array}{l@{\qquad}l}
      \partial_t \mu
      +
      \partial_x \left(b (t,\mu,w) \, \mu\right)
      +
      c (t,\mu,w) \, \mu
      =
      \int_{\reali_+} \left(\eta (t,\mu,w)\right) (y) \, \d{\mu (y)}
      & t \in \hat I
      \\
      \mu (t_o)
      =
      \mu_o
    \end{array}
  \right.
\end{equation}
for $\mu_o \in \mathcal{M}^+ (\reali_+)$, the set of bounded, positive
Radon measures on $\reali_+$ equipped with the following distance,
induced by the dual norm of $\W{1}{\infty}(\reali_+;\reali)$,
see~\cite[\S~2]{MR2871800}:
\begin{equation}
  \label{eq:31}
  d_{\mathcal{M}}(\mu_1, \mu_2)
  =
  \sup
  \left\{
    \int_{\reali_+} \phi \, \d(\mu_1-\mu_2)
    \colon
    \phi \in \C1(\reali_+;\reali)
    \mbox{ and }
    \norma{\phi}_{\W{1}{\infty}} \leq 1
  \right\} \,.
\end{equation}
We refer to~\cite{MR4309603} for basic measure theoretic
results. Below, if $X$ is a Banach space, then $\BC(\hat I;X)$ is the
space of bounded continuous functions with the supremum norm.
$\BCvarF{\alpha}{1}(\hat I \times {\mathcal M}^+(\reali_+); X)$ is the
space of $X$ valued functions which are bounded with respect to the
$\norma{\cdot}_{X}$ norm, H\"older continuous with exponent $\alpha$
with respect to time and Lipschitz continuous in the measure variable
with respect to $d_{\mathcal{M}}$ in~\eqref{eq:31}. These spaces are
equipped with the norms
\begin{eqnarray*}
  \norma{f}_{\BC (\hat I; X)}
  & =
  & \sup_{t \in \hat I} \norma{f (t)}_X \,,
  \\
  \norma{f}_{\BCvarF{\alpha}{1} (\hat I \times \mathcal{M}^+ (\reali_+);X)}
  & =
  & \sup_{t\in \hat I, \mu\in{\mathcal{M}^+(\reali_+)}}
    \left(
    \norma{f(t,\mu)}_{X} +
    \Lip\left(f(t,\cdot)\right) +
    \mathbf{H}\left(f(\cdot,\mu)\right)
    \right) \,,
  \\
  \norma{f}_{(\BC \cap \W1\infty )(\reali_+;\mathcal{M}^+ (\reali_+))}
  & =
  & \sup_{x \in \reali_+}
    \norma{f(x)}_{\mathcal{M} (\reali_+)} +
    \Lip(f) \,,
\end{eqnarray*}
where, with a slight abuse of notation,
\begin{eqnarray*}
  \Lip \left(f (t, \cdot)\right)
  & =
  & \sup_{\underset{\mu_1\neq\mu}{\mu_1,\mu_2 \in \mathcal{M}^+ (\reali_+)}}
    \left(
    \norma{f(t,\mu_1) - f(t,\mu_2)}_{X}
		\middle/ d_{\mathcal{M}}(\mu_1,\mu_2)
    \right) \,,
  \\
  \mathbf{H}\left(f(\cdot,\mu)\right)
  & =
  & \sup_{s_1,s_2\in \hat I}
    \left(
    \norma{f(s_1,\mu) - f(s_2,\mu)}_{X}
    \middle/ \modulo{s_1 - s_2}^{\alpha}
    \right) \,,
  \\
  \Lip(f)
  & =
  & \sup_{\underset{x_1\neq x_2}{x_1,x_2 \in \reali_+}}
	\left(d_{\mathcal{M}} \left(f (x_1),f (x_2)\right)
    \middle/ \norma{x_2-x_1}\right) \,.
\end{eqnarray*}

\begin{definition}
  \label{def:MVBL}
  Given $T \in \hat I$ with $T> t_o$ and $w \in \mathcal{W}$, a
  function $\mu \colon [t_o,T] \to \mathcal{M}^+(\reali_+)$ is a
  \emph{weak solution} to~\eqref{eq:30} on the time interval $[t_o,T]$
  if $\mu$ is narrowly continuous with respect to time (i.e., for
  every bounded function $\psi \in \C0\left(\reali_+; \reali\right)$,
  the map $t \mapsto \int_{\reali_+} \psi(x) \d{\mu(t, x)} $ is
  continuous), and for all
  $\phi \in (\C1 \cap \W{1}{\infty}) \left([t_o,T]\times
    \reali_+;\reali\right)$, the following equality holds:
  \begin{eqnarray*}
    &
    & \int_{t_o}^T \int_{\reali_+}
      \left(
      \partial_t \phi(t,x) + \left(b(t, \mu, w) \right)(x) \;
      \partial_x \phi(t,x)
      -
      \left(c(t, \mu, w)\right)(x) \; \phi(t,x) \right)
      \d\mu(t,x) \d{t} \nonumber
    \\
    &
    & + \int_{t_o}^T \int_{\reali_+} \left( \int_{\reali_+} \phi(t,x)
      \d{\left[\eta(t, \mu, w)(y)\right]} (x) \right) \d\mu(t,y) \d{t}
    \\
    & =
    & \int_{\reali_+} \phi(T,x) \; \d{\mu}(T,x)
      -
      \int_{\reali_+}  \phi(t_o,x) \; \d{\mu_o}(x) \,.
  \end{eqnarray*}
\end{definition}

\begin{proposition}
  \label{prop:MVBL}
  Let $R>0$. Set $\mathcal{U} = \mathcal{M}^+ (\reali)$ and let
  $\mathcal{D} = \left\{\mu \in \mathcal{M}^+ (\reali_+) \colon \mu
    (\reali_+) \leq R\right\}$. Consider the Cauchy
  problem~\eqref{eq:30} under the assumptions, for some positive
  constant $\hat{L}$,

  \begin{enumerate}[label=\bf{(MVBL\arabic*)},
    ref=\textup{\textbf{(MVBL\arabic*)}}, left=0pt]
  \item \label{ip:MVBL1} For every $w \in \mathcal{W}$,
    $b(\cdot, \cdot, w) \in \BC^{\alpha, 1}( \hat{I} \times
    \mathcal{D}; \W1\infty\left(\reali_+;\reali)\right)$.  Further,
    for every $w, w_1, w_2 \in \mathcal{W}$, $t \in \hat{I}$, and
    $\mu \in \mathcal{D}$, $ b(t, \mu, w)(0) \geq 0 $, and, for some
 $B>0$,
    \begin{eqnarray*}
      \norma{b (t,\mu,w)}_{\W1\infty (\reali_+; \reali)}
      & \leq
      & B \,,
      \\
      \norma{b (\cdot,\mu,w_1) - b (\cdot,\mu,w_2)}_{\BC (\hat I; \W1\infty (\reali_+;\reali))}
      & \leq
      & \hat L \; d_{\mathcal{W}} (w_1,w_2) \,.
    \end{eqnarray*}
  \item \label{ip:MVBL2} For every $w \in \mathcal{W}$,
    $c(\cdot, \cdot, w) \in \BC^{\alpha, 1}( \hat{I} \times
    \mathcal{D}; \W1\infty(\reali_+;\reali))$.  Further, there
    exists a positive constant $C \geq 0$ such that, for all
    $w, w_1, w_2 \in \mathcal{W}$, $\mu \in \mathcal{D}$ and
    $t \in \hat{I}$,
    \begin{eqnarray*}
      \norma{c (t,\mu,w)}_{\W1\infty (\reali_+; \reali)}
      &\leq
      & C\,,
      \\
      \norma{c (\cdot,\mu,w_1) - c (\cdot,\mu,w_2)}_{\BC (\hat I; \W1\infty (\reali_+;\reali))}
      &\leq
      & \hat L \; d_{\mathcal{W}} (w_1,w_2) \,.
    \end{eqnarray*}
  \item \label{ip:MVBL3} For all $w \in \mathcal{W}$,
    $\eta (\cdot,\cdot,w) \in \BCvarF{\alpha}{1} \left(
      \hat{I}\times\mathcal{D};\; (\BC \cap \W1\infty) (\reali_+;
      \mathcal{M}^+ (\reali_+)) \right)$.  Further, there exists an
    $E>0$ such that, for all
    $w, w_1, w_2 \in \mathcal{W}$, $t \in \hat{I}$, and
    $\mu \in \mathcal{D}$,
    \begin{eqnarray*}
      \norma{\eta (t,\mu,w)}%
      _{(\BC \cap \W1\infty) (\reali_+; \mathcal{M}^+(\reali_+))}
      & \leq
      & E \,,
      \\
      \norma{\eta (\cdot,\mu,w_1) - h (\cdot,\mu,w_2)}
      _{ \BC (\hat I;\BC \cap \W1\infty)
      (\reali_+; \mathcal{M}^+(\reali_+))}
      &\leq
      & \hat L \; d_{\mathcal{W}} (w_1,w_2) \,.
    \end{eqnarray*}
  \end{enumerate}
  \noindent Then, there exist $T>0$, such that
  $[0,T] \subseteq \hat I$, and a Lipschitz Process on
  $\mathcal{M}^+ (\reali^n)$, pametrized by $\mathcal{W}$ in the sense
  of Definition~\ref{def:LipProc} whose orbits solve~\eqref{eq:30} in
  the sense of Definition~\ref{def:MVBL},
  with
  \begin{equation}
    \label{eq:35}
    \begin{array}{c}
      C_u = {3 (B+C+E)}
      \,,\quad
      C_t = (B+C+E) \, e^{2 (B+C+E) T} R \,,
      \\
      C_w
      =
      C^* ({T}, B, C, E) \; R \; L \;  e^{5 (B+C+E)  T} \,,
      \\
      \mathcal{D}_t
      =
      \left\{
      \mu \in \mathcal{D} \colon \mu (\reali_+) \leq R e^{-3 (B+C+E) ({T}-t)}
      \right\} \,.
    \end{array}
  \end{equation}
\end{proposition}

\noindent The proof is a direct consequence of~\cite[Theorem
2.10]{MR2871800} and, hence, it is omitted. In particular, $C^*$
in~(\ref{eq:35}) is the constant defined in~\cite[Item~(iv),
Theorem~2.10]{MR2871800}.

\begin{proposition}
  \label{prop:MEASBLCoup}
  Set $\mathcal{U} = \mathcal{M}^+ (\reali^n)$. Fix $T > 0$ and assume
  that~\ref{ip:MVBL1}--\ref{ip:MVBL2}--\ref{ip:MVBL3} hold. Let $P^u$
  be a Lipschitz process on $\mathcal{W}$, parametrised by
  $u \in \mathcal{U}$. Call
  $P\colon \mathcal{A} \to \reali^n \times \mathcal{W}$, with
  $P \equiv (P_1,P_2)$, the Process constructed in
  Theorem~\ref{thm:Metric} coupling $P^w$, found in
  Proposition~\ref{prop:MVBL}, and $P^u$. If
  $([t_o,T],t_o,u_o,w_o) \subseteq \mathcal{A}$, then the map
  \begin{equation}\label{def:muSoln}
    \begin{array}{ccccc}
      \mu
      & \colon
      & [t_o,T]
      & \to
      & \mathcal{M}^+ (\reali^n)
      \\
      &
      &t
      & \mapsto
      & P_1 (t,t_o) (\mu,w)
    \end{array}
  \end{equation}
  solves the measure valued balance law
  \begin{displaymath}
    \left\{
      \begin{array}{l@{\qquad}l}
        \partial_t \mu
        +
        \partial_x \left(\bar b (t,\mu) \, \mu\right)
        +
        \bar c (t,\mu) \, \mu
        =
        \int_{\reali_+} \left(\bar \eta (t,\mu)\right) (y) \, \d{\mu (y)}
        & t \in \hat I
        \\
        \mu (t_o)
        =
        \mu_o
      \end{array}
    \right.
  \end{displaymath}
  in the sense of Definition~\ref{def:MVBL}, where
  \begin{displaymath}
    \begin{array}{c}
      \bar b (t,\mu)
      =
      b\left(t, \mu, P_2 (t,t_o) (\mu_o,w_o)\right) \,,
      \qquad\qquad
      \bar c (t,\mu)
      =
      c\left(t, \mu, P_2 (t,t_o) (\mu_o,w_o)\right) \,,
      \\
      \bar \eta (t,\mu)
      =
      \eta\left(t, \mu, P_2 (t,t_o) (\mu_o,w_o)\right) \,.
    \end{array}
  \end{displaymath}
\end{proposition}

\noindent The proof is deferred to \S~\ref{subsec:proof-MVBL}.

\subsection{Scalar NonLinear Conservation Laws}
\label{subsec:scal-cons-laws}

We now consider the following scalar nonlinear conservation law in one
space dimension:
\begin{equation}
  \label{eqn:consLaw}
  \begin{cases}
    \partial_t u + \partial_x f(t, u, w) = 0
    & (t,x) \in \hat I \times \reali,
    \\
    u(t_o,x) = u_o (x)
    & x \in \reali
  \end{cases}
\end{equation}
for $t_o \in \hat I$, $u_o \in \L1(\reali;\reali)$,
$w \in \mathcal{W}$, with
$f \colon \hat I \times \reali \times \mathcal{W} \to \reali$ a
given function.

\begin{definition}
  \label{def:CL}
  Fix $w \in \mathcal{W}$ and $[t_o,T] \subseteq \hat I$. We say
  that a map $u \in \C0\left([t_o,T];\L1(\reali;\reali)\right)$
  is a solution to problem~\eqref{eqn:consLaw} if it is a Kru\v
  zkov-Entropy solution, i.e.
  \begin{multline}
    \label{eq:39}
    \int_{t_o}^T \int_\reali \left[ \modulo{u-k} \, \partial_t \phi +
      \sign(u-k) \, \left(f(t, u, w) - f(t, k, w) \right) \,
      \partial_x \phi \right] \d{x} \d{t}
    \\
    \geq \int_\reali \modulo{u(T, x) - k} \, \phi(T,x) \d{x} -
    \int_\reali \modulo{u_o(x) - k} \, \phi(t_o,x) \d{x},
  \end{multline}
  for all non-negative test functions
  $\phi \in \Cc\infty(\hat I \times \reali; \reali_+)$, and for all
  $k \in \reali$.
\end{definition}

\begin{proposition}
  \label{prop:CL}
  Let $R>0$ and $t_o,T$ be such that $[t_o,T] \subseteq \hat
  I$. Choose $\mathcal{U} = \L1 (\reali; \reali)$ and define
  $\mD = \left\{ u \in \mathcal{U} \colon \tv(u) \leq R
  \right\}$. Consider the Cauchy problem
  \begin{equation}
    \label{eq:38}
    \begin{cases}
      \partial_t u + \partial_x f(u, w) = 0 & (t,x) \in [t_o, T]
      \times \reali,
      \\
      u(t_o,x) = u_o (x) & x \in \reali
    \end{cases}
  \end{equation}
  under the assumptions
  \begin{enumerate}[label=\bf{(CL\arabic*)},
    ref=\textup{\textbf{(CL\arabic*)}}, left=0pt]

  \item \label{it:CL1} For all $w \in \mathcal{W}$, the map
    $u \mapsto f (u,w)$ is piecewise twice continuously
    differentiable.
  \item \label{it:CL2} There exists a positive $F_L$ such that for all
    $u_1,u_2 \in \reali$ and all $w,w_1,w_2 \in \mathcal{W}$
    \begin{eqnarray*}
      \modulo{f (u_1,w) - f (u_2,w)}
      & \leq
      & F_L \;\modulo{u_1-u_2}
      \\
      \Lip\left(f (\cdot, w_1) - f (\cdot, w_2)\right)
      & \leq
      & F_L \; d_{\mathcal{W}} (w_1,w_2)
    \end{eqnarray*}
  \end{enumerate}
  \noindent Then, there exists a Lipschitz Process on
  $\L1 (\reali; \reali)$, pametrized by $\mathcal{W}$, whose orbits
  are solutions to~\eqref{eqn:consLaw} in the sense of
  Definition~\ref{def:CL}, with constants
  in~\eqref{eq:6}--\eqref{eq:7}--\eqref{eq:8}
  \begin{displaymath}
    C_u = 0\,,\quad
    C_t = F_L \, R\,,\quad
    C_w = F_L \, R \,,\quad
    \mD_t = \mD \,.
  \end{displaymath}
\end{proposition}

\noindent The proof is classical and follows, for instance,
from~\cite[Theorem~2.14 and Theorem~2.15]{HoldenFT}.

\begin{remark}
  \label{rem:CL}
  \rm The present treatment is limited to \emph{homogeneous}, i.e.,
  with a flux independent of $x$, conservation laws. Note that general
  $2\times2$ systems of conservation laws can \emph{not} be approached
  by means of Theorem~\ref{thm:Metric} while, for instance, we do
  comprehend a non local coupling of the form
  \begin{displaymath}
    \left\{
      \begin{array}{l}
        \partial_t u + \partial_x f \left(u,\int_{\reali}w \d{x}\right) = 0
        \\
        u (0,x) = u_o (x)
      \end{array}
    \right.
    \qquad\qquad  \left\{
      \begin{array}{l}
        \partial_t w + \partial_x g\left(w, \int_{\reali} u\, \d{x}\right) = 0
        \\
        w (0,x) = w_o (x) \,.
      \end{array}
    \right.
  \end{displaymath}
\end{remark}

\begin{proposition}
  \label{prop:CLcoupling}
  Set $\mathcal{U} = \L1 (\reali; \reali)$.  Assume
  that~\ref{it:CL1}--\ref{it:CL2} hold. Let $P^u$ be a Lipschitz
  process on $\mathcal{W}$, parametrised by $u \in \mathcal{U}$.  Call
  $P\colon \mathcal{A} \to \reali^n \times \mathcal{W}$, with
  $P \equiv (P_1,P_2)$, the Process constructed in
  Theorem~\ref{thm:Metric} coupling $P^w$, generated by~\eqref{eq:38},
  to $P^u$. If $([t_o,T],t_o,u_o,w_o)$ $\subseteq \mathcal{A}$, then
  \begin{displaymath}
    \begin{array}{@{}c@{\,}c@{\,}c@{\,}c@{\,}c@{}}
      u
      & \colon
      & [t_o,T]
      & \to
      & \L1 (\reali;\reali)
      \\
      &
      &t
      & \mapsto
      & P_1 (t,t_o) (u_o , w_o)
    \end{array}
    \qquad \mbox{ solves }\qquad
    \left\{
      \begin{array}{l}
        \partial_t u + \partial_x \bar f(t,u) = 0
        \vspace{.2cm}
        \\
        u(t_o) = u_o,
      \end{array}
    \right.
  \end{displaymath}
  in the sense of Definition~\ref{def:CL}, where
  $\bar f (t,u) = f\left(u, P_2 (t,t_o) (u_o,w_o)\right)$.
\end{proposition}

\noindent The proof is left until \S~\ref{sec:proofs-s-refs-1}.

\section{Specific Coupled Problems}
\label{sec:coupled-problems}

The abstract framework developed in
Section~\ref{sec:defin-abstr-results}, thanks to the proofs in the
subsequent paragraphs, allows to prove the Lipschitz well posedness of
several models.

As a first example, consider the model introduced in~\cite{MR2765683},
where a large and slow vehicle positioned at $y = y (t)$ affects the
overall traffic density $\rho = \rho (t,x)$. The resulting
model~\cite[Formula~(2.1)]{MR2765683} consists in the coupling of the
Lighthill-Whitam~\cite{LighthillWhitham} and Richards~\cite{Richards}
macroscopic model decribing the evolution of $\rho$ coupled with an
ordinary differential equation for $y$, that is
\begin{equation}
  \label{eq:73}
  \left\{
    \begin{array}{l}
      \partial_t \rho + \partial_x f\left(x,y (t),\rho\right) = 0
      \\
      \dot y = w\left(\rho (t,y)\right)
    \end{array}
  \right.
\end{equation}
Clearly, this coupled problem fits in Theorem~\ref{thm:Metric} thanks
to Proposition~\ref{prop:CLcoupling} and
Proposition~\ref{prop:ODEcoupled}, once the functions $f$ and $w$ meet
reasonable requirements.

In the next paragraphs, we consider in particular the case of a
predator--prey system (\S~\ref{subsec:consensus-game}) and that of an
epidemiological model (\S~\ref{subsec:an-epid-model}). To our
knowledge, this latter well posedness is first proved here.

\subsection{Predators and Prey}
\label{subsec:consensus-game}

On the basis of the games introduced in~\cite{MR4030510} we consider
the following predator--prey model:
\begin{equation}
  \label{eq:32}
  \!\!\!
  \left\{
    \begin{array}{@{}l}
      \displaystyle
      \partial_t \rho
      +
      \div_x \left(
      \rho \;
      V\!\left(t, x, p (t)\right)
      \right)
      =
      - \eta \left(\norma{p (t) - x} \right)\, \rho (t,x)
      \\
      \rho (0,x) = \bar\rho (x)
    \end{array}
  \right.
  \; \mbox{ where } \;
  \left\{
    \begin{array}{@{}l}
      \dot p = U \left(t, p, \rho (t)\right)
      \\
      p (0) = \bar p
    \end{array}
  \right.
\end{equation}
We consider a specific example, letting $\rho = \rho (t,x)$ be the
density of some prey species moving in $\reali^N$ and $p = p (t)$ be
the position in $\reali^N$ of a predator hunting it. To escape the
predator, prey adopt a strategy defined by the speed
\begin{equation}
  \label{eq:55}
  V (t,x,p)
  =
  - \dfrac{p-x}{\alpha+\norma{p-x}^2} \;
  \psi \left(\norma{p-x}^2\right)
\end{equation}
where the term $\dfrac{p-x}{\alpha+\norma{p-x}^2}$ stands for the
escape direction of the prey. The positive term $\alpha$ in the
denominator smooths the normalization. The function $\psi$ describes
the relevance of the predator $p$ to the prey at $x$ as a function of
the distance $\norma{p-x}$. The function
$\eta = \eta \left(\norma{p-x}\right)$ describes the effect of the
feeding of the predator at $p$ on the prey at $x$.  On the other hand,
the predator hunts moving towards the region of highest (mean) prey
density, i.e., with speed
\begin{equation}
  \label{eq:56}
  U (t,p,\rho)
  =
  \left( \nabla \phi *\rho\right) (p) \,,
\end{equation}
where $\phi$ is an averaging kernel.

Here, we show that~\eqref{eq:32} fits in the general framework
presented in Section~\ref{sec:defin-abstr-results}. Indeed, with
reference to~\S~\ref{subsec:IVP}, set
\begin{equation}
  \label{eq:33}
  \begin{array}{rcl}
    \mathcal{U}
    & =
    & \L1 (\reali^N; \reali)\,,
    \\
    \mathcal{W}
    & =
    & \reali^N\,,
  \end{array}
  \qquad
  \begin{array}{rcl}
    u
    & =
    & \rho \,,
    \\
    w
    & =
    & p \,,
  \end{array}
  \qquad
  \begin{array}{rcl}
    v (t,x,w)
    & =
    & V (t,x,w) \,,
    \\
    m (t,x,w)
    & =
    & -\eta \left(\norma{w-x}\right) \,,
    \\
    q (t,x,w)
    & =
    & 0 \,,
  \end{array}
\end{equation}
while with reference to~\S~\ref{subsec:ordin-diff-equat}, set
\begin{equation}
  \label{eq:34}
  \begin{array}{rcl}
    \mathcal{U}
    & =
    & \reali^N \,,
    \\
    \mathcal{W}
    & =
    & \L1 (\reali^N; \reali) \,,
  \end{array}
  \qquad
  \begin{array}{rcl}
    u
    & =
    & p \,,
    \\
    w
    & =
    & \rho \,,
  \end{array}
  \qquad
  f (t,u,w) = U (t,u,w) \,.
\end{equation}

\begin{proposition}
  \label{lem:consensus}
  Fix positive $\alpha,r_\rho, r_p, r_\eta$ and mollifiers
  \begin{description}
  \item[(V)] Let $V$ be as in~\eqref{eq:55} with
    $\psi \in \Cc\infty (\reali^N;\reali_+)$, with
    $\spt \psi \subseteq B (0,r_\rho)$ and
    $\int_{B (0,r_\rho)} \psi \d\xi=1$.

  \item[(U)] Let $U$ be defined in~\eqref{eq:56} with
    $\phi \in \Cc\infty (\reali;\reali)$, positive, with
    $\spt \phi \subseteq [-r_p,r_p]$ in~\eqref{eq:56}.

  \item[($\mathbf\eta$)]$\eta \in \Cc\infty (\reali^N; \reali)$,
    positive, with $\spt \eta \subseteq B (0, r_\eta)$.
  \end{description}
  Then, conditions~\ref{ip:(V)}--\ref{ip:(P)}--\ref{ip:(Q)}
  and~\ref{it:ODE1}--\ref{it:ODE2} are all satisfied. Therefore,
  model~\eqref{eq:32} defines a unique global process in the sense of
  Definition~\ref{def:Global}.
\end{proposition}

\begin{proof}
  Consider first~\ref{ip:(V)}. By~\eqref{eq:55}, $V$ is a smooth
  function and the exponential factor ensures all the required
  boundedness conditions. We also have that
  $\norma{\nabla_p V}_{\L\infty (\reali_+ \times \reali^N \times
    \reali^N; \reali^{N\times N})}$ is bounded, proving the first
  Lipschitz requirement in~\ref{ip:(V)}. Prove now the latter
  inequality:
  \begin{eqnarray*}
    &
    & \int_{\reali^N}
      \modulo{\nabla_x \cdot \left(V (t,x,p_1) - V (t,x,p_2)\right)} \d{x}
    \\
    & =
    & \int_{\reali^N}
      \modulo{\nabla_x \cdot V (t,x,p_1) - \nabla_x \cdot V (t,x,p_2)} \d{x}
    \\
    & =
    & \int_{B (p_1,r_p)\cup B (p_2,r_p)}
      \modulo{\nabla_x \cdot V (t,x,p_1) - \nabla_x \cdot V (t,x,p_2)} \d{x}
    \\
    & \leq
    & \int_{B (p_1,r_p)\cup B (p_2,r_p)}
      \sup_{p\in \reali^N}\norma{
      \nabla_p \nabla_x \cdot V (t,x,p)
      }
      \d{x}
      \; \norma{p_2-p_1}
  \end{eqnarray*}
  proving also the latter requirement in~\ref{ip:(V)}.

  To prove~\ref{ip:(P)}, compute:
  \begin{eqnarray*}
    \norma{m(t, \cdot, w)}_{\L\infty(\reali^n;\reali)}
    +
    \tv\left( m(t, \cdot, w) \right)
    & =
    & \max_{B (0,r_\eta)} \modulo{\eta}
      +
      \norma{\eta'}_{\L1 (B (0,r_\eta);\reali)} \,;
    \\
    \norma{m(t, \cdot, w_1) - m(t, \cdot, w_2)}_{\L1(\reali^n;\reali)}
    & \leq
    & \int_{B (w_1,r_\eta) \cup B (w_2,r_\eta)}
      \sup _{B (0,r_\eta)}\modulo{\eta'} \norma{w_2-w_1}\d{x}
    \\
    & \leq
    & \O \norma{\eta'}_{\L\infty (B (0,r_\eta); \reali)} \, \norma{w_2-w_1} \,.
  \end{eqnarray*}
  Clearly, due to~\eqref{eq:33}, \ref{ip:(Q)} is immediate.

  The regularity required in~\ref{it:ODE1} is immediate. Pass to the
  Lipschitz estimate:
  \begin{eqnarray*}
    &
    & \norma{U (t,p_1,\rho_1) - U (t,p_2,\rho_2)}
    \\
    & \leq
    & \norma{U (t,p_1,\rho_1) - U (t,p_1,\rho_2)}
      + \norma{U (t,p_1,\rho_2) - U (t,p_2,\rho_2)}
    \\
    & =
    & \norma{\left(\nabla \phi * (\rho_1 - \rho_2)\right) (p_1)}
      +
      \norma{\left(\nabla\phi * \rho_2\right) (p_1) - \left(\nabla\phi * \rho_2\right) (p_2)}
    \\
    & \leq
    & \norma{\nabla \phi * (\rho_1 - \rho_2)}_{\L\infty (\reali^N;\reali^N)}
      +
      \norma{\nabla^2 \phi * \rho_2}_{\L\infty (\reali^N;\reali^{N\times N})} \;
      \norma{p_1 - p_2}
    \\
    & \leq
    & \norma{\nabla \phi}_{\L\infty (\reali^N;\reali^N)} \;
      \norma{\rho_1 - \rho_2}_{\L1 (\reali^N;\reali)}
      +
      \norma{\nabla^2 \phi * \rho_2}_{\L\infty (\reali^N;\reali^{N\times N})} \;
      \norma{p_1 - p_2} \,.
  \end{eqnarray*}
  Finally, the latter boundedness in~\ref{it:ODE2} is proved as
  follows:
  \begin{eqnarray*}
    \sup_{\rho \in \mathcal{D}_\rho} \norma{U (\cdot, \cdot, \rho)}
    & \leq
    & \sup_{\rho \in \mathcal{D}_\rho}
      \norma{\nabla \phi}_{\L\infty (\reali^N; \reali^N)}
      \;
      \norma{\rho}_{\L1 (\reali^N; \reali)}
  \end{eqnarray*}
  completing the proof by the definition of $\mathcal{D}_\rho$.

  By Proposition~\ref{prop:RE}, the balance law in~\eqref{eq:32}
  defines a global process $P_1$. Similarly,
  Proposition~\ref{prop:ODE} ensures that the ordinary differential
  equation in~\eqref{eq:32} generates a global process $P_2$. Now,
  Proposition~\ref{prop:RENCoup} and Proposition~\ref{prop:ODEcoupled}
  ensure that the global process $P$ obtained from $P_1$ and $P_2$
  through Theorem~\ref{thm:Metric} yields a solution to the coupled
  problem~\eqref{eq:32}.
\end{proof}

\subsection{Modeling Vaccination Strategies}
\label{subsec:an-epid-model}

Consider the model presented in~\cite[\S~2]{ColomboMarcelliniRossi}:
\begin{equation}
  \label{eq:12}
  \left\{
    \begin{array}{r@{\;}c@{\;}l@{}}
      \dot S
      & =
      & - \rho_S \, I \, S - p (t)
      \\
      \partial_t V + \partial_\tau V
      & =
      & - \rho_V \, I \, V
      \\
      \dot I
      & =
      & (\rho_S \, S + \int_0^{T_*} \rho_V \, V) I  - \theta \, I - \mu \, I
      \\
      \dot R
      & =
      & \theta \, I + V (t,T_*)
      \\
      V (t,0)
      & =
      & p(t) \,.
    \end{array}
  \right.
\end{equation}
It describes a population consisting of susceptibles, $S = S (t)$, of
infected that are also infective, $I = I (t)$, and recovered
individuals, $R = R (t)$. The vaccination rate is $p = p (t)$ and
vaccinated individuals need a time $T_*$ to get immunized. More
precisely, $V = V (t,\tau)$ is the number of individuals at time $t$
vaccinated at time $t-\tau$, for $\tau \in [0, T_*]$. Thus, at time
$T_*$, vaccinated individual enter the $R$ population.

The positive constants $\rho_S$, $\theta$ and $\mu$ quantify the
infectivity rate, the recovery rate and the mortality rate,
respectively. The function $\rho_V = \rho_V (\tau)$ describes the
infectivity rate of individuals vaccinated after time $\tau$ from
being dosed.

Note that model~\eqref{eq:12} is triangular, in the sense that the
evolution of the $R$ population results from that of the other ones,
without affecting them.

Model~\eqref{eq:12}, once the $R$ population is omitted, fits in the
abstract framework presented in
Section~\ref{sec:defin-abstr-results}. Indeed, with reference to the
notation used in~\S~\ref{subsec:ordin-diff-equat}, we pose
\begin{equation}
  \label{eq:13}
  \begin{array}{c}
    \mathcal{U}
    =
    \reali^2
    \,,\qquad
    \mathcal{W}
    =
    \L1 ([0,T_*];\reali)
    \,, \qquad
    u
    =
    \left[
    \begin{array}{c}
      S
      \\
      I
    \end{array}
    \right]
    ,\qquad
    w
    =
    V
    \,,
    \\
    f (t,u,w)
    =
    \left[
    \begin{array}{c}
      -\rho_S \, u_1 \, u_2 - p (t)
      \\
      \left(
      \rho_S\, u_1 + \int_0^{T_*}\rho_V (\tau) \, w (\tau) \d\tau
      -\theta-\mu
      \right)
      u_2
    \end{array}
    \right],
  \end{array}
\end{equation}
while with reference to~\S~\ref{subsec:IBVP2}, we set
\begin{equation}
  \label{eq:19}
  \begin{array}{r@{\,}c@{\,}l}
    \mathcal{U}
    & =
    & \L1 ([0,T_*];\reali)
    \\
    \mathcal{W}
    & =
    & \reali^2
  \end{array}
  \quad
  x = \tau
  \,,\quad
  u
  =
  V
  \,,\quad
  w
  =
  \left[
    \begin{array}{c}
      S
      \\
      I
    \end{array}
  \right]
  ,\quad
  \begin{array}{r@{\,}c@{\,}l}
    v (t,x)
    & =
    & 1
    \\
    m (t,x,w)
    & =
    & -\rho_V (x) \, w_2
    \\
    q (t,x,w)
    & =
    & 0
    \\
    b (t)
    & =
    & p (t) \,.
  \end{array}
\end{equation}

The well posedness of~\eqref{eq:12} now follows once we verify that
Proposition~\ref{prop:ODEcoupled}
and~Proposition~\ref{prop:coupleIBVP2} can be applied.

\begin{proposition}
  \label{prop:appl}
  Fix positive $r, T_*, \rho_S$ and choose
  $p \in \BV (\reali_+; \reali)$, $\rho_V \in \BV ([0,T_*];
  \reali)$. Then, problem~\eqref{eq:12} defines a unique global
  process $P$, in the sense of Definition~\ref{def:Global}, defined on
  all initial data
  \begin{equation}
    \label{eq:5}
    S_o, I_o, R_o \in [0,r]
    \quad \mbox{ and } \quad
    V_o \in \L1 ([0,T_*]; \reali_+)
    \mbox{ with } \tv (V_o) + \norma{V_o}_{\L\infty (\reali; \reali)}\leq r\,.
  \end{equation}
  $P$ is Lipschitz continuous as a function of time and of the initial
  data, with respect to the Euclidean norm in $(S_o,I_o,R_o)$ and to
  the $\L1$ norm in $V$.
\end{proposition}

\begin{proof}
  Verifying~\ref{it:ODE1} is immediate. The Lipschitz continuity
  required in~\ref{it:ODE2} follows from the boundedness
  $u \in \mathcal{D}_{\mathcal{U}}$, which is a closed ball in
  $\mathcal{U} = \reali^2$ and from the choice of $\rho_V$, see
  \S~\ref{subsec:ordin-diff-equat}. Hence, Proposition~\ref{prop:ODE}
  applies.

  Conditions~\ref{item:IBVP1} and~\ref{item:IBVP3} are immediate. The
  first requirement in~\ref{item:IBVP2} follows from the choice of
  $\rho_V$ and the boundedness of $\mathcal{D}_{\mathcal{U}}$. The
  second is ensured by the linearity of $m$ and the boundedness of
  $\rho_V$. Since $p$ has bounded variation, \ref{item:IBVP4} is
  satisfied on any bounded time interval. Hence, also
  Proposition~\ref{prop:IBVP2} can be applied.

  Then, Proposition~\ref{prop:ODEcoupled}
  and~Proposition~\ref{prop:coupleIBVP2}, through
  Theorem~\ref{thm:Metric}, ensure the well posedness of the coupled
  system~\eqref{eq:13}--\eqref{eq:19}.

  We now verify the well posedness of the $R$
  component. From~\eqref{eq:12}, using~\eqref{eq:rIBVP2}, we have
  \begin{displaymath}
    V (t,\tau)
    =
    \left\{
      \begin{array}{l@{\quad\mbox{ if }}r@{\,}c@{\,}l}
        V_o (\tau+t_o-t) \,
        \exp\left(- \int_{t_o}^t \rho_V (s) \, I (s) \d{s}\right)
        & t
        & \leq
        & \tau + t_o\,,
        \\
        p (t-\tau) \,
        \exp\left(- \int_{t-\tau}^t \rho_V (s) \, I (s) \d{s}\right)
        & t
        & >
        & \tau + t_o\,.
      \end{array}
    \right.
  \end{displaymath}
  This shows that the map $t \mapsto V (t,T_*)$ is sufficiently
  regular for the equation for $R$, namely
  $\dot R = \theta\, I (t) + V (t,T_*)$, to be explicitly solved:
  $R (t) = R_o + \int_0^t \left(I (s) + V (s,T_*)\right) \d{s}$. Thus,
  the full model~\eqref{eq:12} is well posed.
\end{proof}

\section{Technical Details}
\label{sec:technical-details}

\subsection{Proofs for Section~\ref{sec:defin-abstr-results}}
\label{subsec:proofs-sect-refs}

\begin{proofof}{Theorem~\ref{thm:Metric}}
  We begin by showing $F$ is a local flow in the sense of
  Definition~\ref{def:local}. $F$ is continuous as it is a pairing of
  two continuous functions. Further
  \begin{displaymath}
    F(0, t_o)(u, w) = \left(P^w(t_o, t_o)u, P^u(t_o, t_o)w
    \right) = \left(u, w \right).
  \end{displaymath}
  We prove the Lipschitz continuity in time and with respect to
  initial conditions of $F$:
  \begin{eqnarray*}
    &
    & d\left(F(\tau_1, t_o)(u_1, w_1), F(\tau_2, t_o)(u_2, w_2)\right)
    \\
    & \leq
    & d_\spa\left(P^{w_1}(t_o + \tau_1, t_o)u_1,
      P^{w_1}(t_o + \tau_1, t_o)u_2\right)
      +
      d_\spa\left(P^{w_1}(t_o + \tau_1, t_o)u_2,
      P^{w_2}(t_o + \tau_1, t_o)u_2\right)
    \\
    &
    & + d_\spa\left(P^{w_2}(t_o + \tau_1, t_o)u_2,
      P^{w_2}(t_o + \tau_2, t_o)u_2\right)
    \\
    &
    & + d_\spb\left(P^{u_1}(t_o + \tau_1, t_o)w_1,
      P^{u_1}(t_o + \tau_1, t_o)w_2\right)
      + d_\spb\left(P^{u_1}(t_o + \tau_1, t_o)w_2,
      P^{u_2}(t_o + \tau_1, t_o)w_2\right)
    \\
    &
    & + d_\spb\left(P^{u_2}(t_o + \tau_1, t_o)w_2,
      P^{u_2}(t_o + \tau_2, t_o)w_2\right)
    \\
    & \leq
    & e^{C_u \tau_1} \, d_\spa(u_1, u_2)
      +
      C_w\, \tau_1 \, d_\spb(w_1, w_2)
      +
      C_t \, \modulo{\tau_1 - \tau_2}
    \\
    &
    & + e^{C_u\tau_1} \, d_\spb(w_1, w_2)
      +
      C_w \, \tau_1 \, d_\spa(u_1, u_2)
      +
      C_t \, \modulo{\tau_1 - \tau_2}
    \\
    & \leq
    & (e^{C_u \delta} + C_w \, \delta) \;
      d\left((u_1, w_1), (u_2,w_2)\right)
      + 2\, C_t \; \modulo{\tau_1 - \tau_2} \,.
  \end{eqnarray*}
  Thus $F$ is indeed a local flow in the sense of
  Definition~\ref{def:local}, with
  $\Lip (F) = e^{C_u \delta} + C_w\, \delta + 2C_t$.

  We now show that $F$ satisfies the assumptions of
  Theorem~\ref{thm:main}. Consider~\eqref{eq:k}:
  \begin{eqnarray}
    \nonumber
    &
    & d\left(F(k \tau, t_o + \tau) \circ F(\tau, t_o)(u, w),
      F((k+1)\tau, t_o)(u, w)\right)
    \\
    \label{eq:1}
    & =
    & d_\spa
      \left(
      P^{P^u(t_o + \tau,t_o)w}(t_o + (k+1)\tau, t_o + \tau)
      P^w(\tau, t_o)u,
      P^w \left(t_o + (k+1)\tau, t_o\right)u
      \right)
    \\
    \label{eq:2}
    & +
    & \!\!\! d_\spb \!
      \left(
      P^{P^w(t_o + \tau,t_o)u}(t_o + (k{+}1)\tau, t_o + \tau)
      P^u(t_o + \tau, t_o)w,
      P^u\left(t_o + (k{+}1)\tau, t_o\right)w
      \right).
  \end{eqnarray}
  We consider only the term~\eqref{eq:1}, since the latter is entirely
  similar. By~\eqref{eq:process2}, we have
  \begin{displaymath}
    P^w \left(t_o + (k+1)\tau, t_o\right)u
    =
    P^w(t_o + (k+1)\tau, t_o + \tau) \; P^w(t_o + \tau, t_o)u \,,
  \end{displaymath}
  hence, via~\eqref{eq:7} and~\eqref{eq:8},
  \begin{eqnarray}
    \nonumber
    &
    & d_\spa
      \left(
      P^{P^u(t_o + \tau,t_o)w}(t_o + (k+1)t, t_o + \tau)
      P^w(t_o + \tau, t_o)u,
      P^w\left(t_o + (k+1)\tau, t_o\right)u
      \right)
    \\
    \nonumber
    & =
    & d_\spa
      \big(
      P^{P^u(t_o + \tau,t_o)w}(t_o + (k+1)\tau, t_o + \tau)
      P^w(t_o + \tau, t_o)u,
    \\
    \nonumber
    && \qquad
       P^w(t_o + (k+1)\tau, t_o+\tau)P^w(t_o+\tau, t_o)u
       \big)
    \\
    \nonumber
    & \leq
    & C_w \; k \, \tau \; d_\spb\left(P^u(t_o + \tau, t_o)w, w\right)
    \\
    \label{eq:3}
    & \leq
    & k \, \tau \; C_t \, C_w \tau \,.
  \end{eqnarray}
  Combining~\eqref{eq:3} with the analogous estimate
  bounding~\eqref{eq:2}, we end up with
  \begin{displaymath}
    d\left(
      F(k\tau, t_o + \tau) \circ F(\tau, t_o)(u, w),
      F\left((k+1)\tau, t_o\right)(u, w)
    \right)
    \leq
    k\, \tau \; \omega (\tau)
  \end{displaymath}
  where $\omega$ is as in~\eqref{eqn:tangM}. Thus~\eqref{eq:k} is
  satisfied.

  We consider the second condition in Theorem~\ref{thm:main},
  namely~\eqref{eq:Stability}.  Note that Euler polygonals for the
  local flow $F$, see \Cref{def:Euler-polygonals}, can be written
  recursively, as
  \begin{displaymath}
    F^\eps(\tau, t_o)(u,w)
    =
    F(\tau - k\eps, t_o + k\eps) \circ F^\eps(k\eps, t_o)(u,w) \,.
  \end{displaymath}
  For any $\tau \in [0,\delta]$ and for any $(u, w)$,
  $(\bar{u}, \bar{w})$ in $\mathcal{U} \times \mathcal{W}$, we have
  \begin{eqnarray*}
    d\left(F(\tau, t_o)(u, w), F(\tau, t_o)(\bar{u}, \bar{w})\right)
    & =
    & d_\spa(P^w(t_o + \tau, t_o)u, P^{\bar{w}}(t_o + \tau, t_o)\bar{u})
    \\
    &
    & \quad
      + d_\spb(P^u(t_o + \tau, t_o)w,
      P^{\bar{u}}(t_o + \tau, t_o)\bar{w}) \,.
  \end{eqnarray*}
  For the first of these summands, by the triangle inequality, we have
  {%
    \begin{eqnarray*}
      &
      & d_\spa
        \left(
        P^w(t_o + \tau, t_o)u,
        P^{\bar{w}}(t_o + \tau, t_o)\bar{u}
        \right)
      \\
      & \leq
      & d_\spa
        \left(
        P^w(t_o + \tau, t_o)u, P^w(t_o + \tau, t_o)\bar{u}
        \right)
        +
        d_\spa
        \left(P^w(t_o + \tau, t_o)\bar u,
        P^{\bar{w}}(t_o + \tau, t_o)\bar{u}\right)
      \\
      & \leq
      & e^{C_u\tau} \, d_\spa(u, \bar{u}) + C_w \, \tau \; d_\spb(w, \bar{w}) \,.
    \end{eqnarray*}%
  } The second term is estimated analogously, leading to
  \begin{equation}
    \label{eq:4}
    d\left(F(\tau, t_o)(u, w), F(\tau, t_o)(\bar{u}, \bar{w})\right)
    \leq
    \left(e^{C_u \tau} + C_w \, \tau\right) \,
    d\left((u, w), (\bar{u}, \bar{w})\right) \,.
  \end{equation}
  Estimate~\eqref{eq:4} is of use in the following:
  \begin{eqnarray*}
    &
    & d\left(F^\eps(\tau, t_o)(u, w), F^\eps(\tau, t_o)(\bar{u}, \bar{w})\right)
    \\
    & =
    & d\left(
      F(\tau - k\eps, t_o + k\eps) F^\eps(k\eps, t_o)(u, w),
      F(\tau - k\eps, t_o + k\eps) F^\eps(k\eps, t_o)(\bar{u}, \bar{w})
      \right)
    \\
    & \leq
    & \left(e^{C_u(\tau-k\eps)} + C_w \, (\tau-k\eps)\right) \,
      d\left(
      F^\eps(k\eps, t_o)(u, w), F^\eps(k\eps, t_o)(\bar{u}, \bar{w})
      \right) \,.
  \end{eqnarray*}
  It remains to estimate the distance in the latter right hand
  side. We have for any $k \in \naturali \setminus\{0\}$,
  \begin{displaymath}
    F^\eps(k\eps, t_o)(u, w) = F(\eps, t_o)
    \; F^\eps\left((k-1)\eps, t_o\right)(u, w),
  \end{displaymath}
  and thus using iteratively~\eqref{eq:4},
    \begin{eqnarray*}
      &
      & d\left(
        F^\eps(k\eps, t_o)(u, w),
        F^\eps(k\eps, t_o)(\bar{u}, \bar{w})\right)
      \\
      & \leq
      & \left(e^{C_u \eps} + C_w \, \eps\right) \,
        d\left(
        F^\eps((k-1)\eps, t_o)(u, w),
        F^\eps((k-1)\eps, t_o)(\bar{u}, \bar{w})
        \right)
      \\
      & \leq
      & \left(e^{C_u \eps} + C_w \, \eps\right)^k
        d\left((u, w), (\bar{u}, \bar{w})\right) \,.
    \end{eqnarray*}%
  Therefore,
  \begin{eqnarray*}
    &
    & d\left(F^\eps(\tau, t_o)(u, w), F^\eps(\tau, t_o)(\bar{u}, \bar{w})\right)
    \\
    & \leq
    & \left(e^{C_u(\tau-k\eps)} + C_w \, (\tau-k\eps)\right)
      \left(e^{C_u \eps} + C_w \, \eps\right)^k
      \;        d \left((u,w),(\bar u, \bar w)\right).
  \end{eqnarray*}
  Hence, \eqref{eq:Stability} is satisfied provided there exists a
  positive $L$ such that for all $ \eps > 0 $ and $t \in [0,T]$
  \begin{displaymath}
    \left(e^{C_u(\tau-k\eps)} + C_w \, (\tau-k\eps)\right)
    \left(e^{C_u \eps} + C_w \, \eps\right)^k    \leq
    L \,,
  \end{displaymath}
  where $k = \lfloor\frac{\tau}{\eps}\rfloor$. Indeed, since
  $e^a + b \leq e^{a+b}$ for all $a,b \in \reali_+$, we have
  \begin{displaymath}
    \left(e^{C_u(\tau-k\eps)} + C_w(\tau-k\eps)\right)
    \left(e^{C_u\eps} + C_w \, \eps\right)^k
    \leq
    e^{(C_u+C_w)(\tau-k\eps)}
    \left(e^{(C_u+C_w)\eps}\right)^k
    =
    e^{(C_u+C_w)\tau}
  \end{displaymath}
  so that $L = e^{(C_u+C_w)\delta}$.

  Finally, note that~\eqref{eq:77} directly follows from the
  definition~\eqref{eq:36} of $F$, together with the properties
  $P^w (t_o+\tau,t_o) \mathcal{D}_{t_o}^{\mathcal{U}} \subseteq
  \mathcal{D}_{t_o+\tau}^{\mathcal{U}}$, which holds for all
  $w \in \mathcal{W}$, and
  $P^u (t_o+\tau,t_o) \mathcal{D}_{t_o}^{\mathcal{W}} \subseteq
  \mathcal{D}_{t_o+\tau}^{\mathcal{W}}$, which holds for all
  $u \in \mathcal{U}$. Therefore, with reference to~\eqref{def:D3}, we
  have
  $\mathcal{D}_{t_o}^3 \supseteq (\mathcal{D}^{\mathcal{U}}_{t_o}
  \times \mathcal{D}^{\mathcal{W}}_{t_o})$ and Condition~1.~in
  Theorem~\ref{thm:main} completes the proof of~\eqref{eq:78}.
\end{proofof}

\begin{proofof}{Theorem~\ref{thm:Metric2}}
  The continuity of $\hat F$ is immediate. The Lipschitz continuity
  follows from the triangle inequality and a Lipschitz constant is
  $\Lip (\hat F) = \mathcal{L} + \max \{\Lip (F^w) \,,$
  $\, \Lip (F^u)\}$. Hence, $\hat F$ is a local flow according to
  Definition~\ref{def:local}.

  Concerning the tangency condition, compute
  \begin{eqnarray*}
    \dfrac1\tau \, d\left(\hat F (\tau,t_o) (u,w),F (\tau,t_o) (u,w)\right)
    & =
    & \dfrac1\tau \,
      d_{\mathcal{U}}\left(
      F^w (\tau,t_o) u,
      P^w (t_o+\tau,t_o) u
      \right)
    \\
    &
    & +
      \dfrac1\tau \,
      d_{\mathcal{W}}\left(
      F^u (\tau,t_o) w,
      P^u (t_o+\tau,t_o) w
      \right)
  \end{eqnarray*}
  and the first order tangency condition~\eqref{eq:tangent} allows to
  complete the proof.
\end{proofof}

{
  \subsection{Proofs for Section~\ref{subsec:ordin-diff-equat}}
  \label{sub:ODE}

\begin{proofof}{Corollary~\ref{cor:global-existence-ODE}}
  For $k \in \naturali$, define $R_k = 2^k$ and
  $\hat{\mathcal D}^k = \overline{ B\left(0, R_k\right)}$.
  Fix $u_o \in \reali^n$. There exists
  $\bar k \in \naturali \setminus \left\{0\right\}$ such that
  $\norma{u_o} \leq R_{\bar k - 1}$.  We proceed recursively.
  \begin{description}
  \item[For $k = \bar k$:] consider the process $P^w_k$, given by
    \Cref{prop:ODE}, according to the choice $R_k = 2^k$. By
    \Cref{prop:ODE} we know that $P^w_k(t, 0)u_o$ is defined for every
    $t \in [t_o, T_k]$, where
    $T_k = \frac{R_k}{2 F_\infty(R_k)}$.  Define
    $u_k = P^w_k(T_k, t_o)u_o \in \hat{\mathcal D}^k$.

  \item[For $k > \bar k$:] assume $u_{k-1} \in \hat{\mathcal D}^{k-1}$
    and consider the process $P^w_k$, given by \Cref{prop:ODE},
    according to the choice $R_k = 2^k$. By \Cref{prop:ODE} we know
    that $P^w_k(t, t_o)u_{k-1}$ is defined for every
    $t \in [t_o, T_k]$, where
    $T_k = \frac{R_k}{2 F_\infty(R_k)}$.  Define
    $u_k = P^w_k(T_k, t_o)u_{k-1} \in \hat{\mathcal D}^k$.
  \end{description}
  Define the function
  \begin{equation*}
    u(t) = \left\{
      \begin{array}{ll}
        P^w_{\bar k}(t, t_o)u_{o}
        & \quad \textrm{ if }\, t\in [t_o, T_{\bar k}]
          \vspace{.2cm}\\
        P^w_{k}(t - \sum_{h=\bar k}^{k-1} T_{h}, 0)u_{k-1}
        & \quad \textrm{ if }\, \sum_{h=\bar k}^{k-1} T_{h}<
          t \leq \sum_{h=\bar k}^k T_{h},
      \end{array}
    \right.
  \end{equation*}
  which clearly is a solution to~\eqref{eq:21}. Computing
  \begin{equation*}
    \sum_{k=\bar k}^{+\infty} T_k
    = \sum_{k=\bar k}^{+\infty} \frac{2^{k-1}}{F_\infty(2^k)}
    \geq O(1)\sum_{k=\bar k}^{+\infty} \frac{2^{k-1}}{2^k \log(2^k)}
    = O(1)\sum_{k=\bar k}^{+\infty} \frac{1}{k} = +\infty,
  \end{equation*}
  shows that the solution $u$ is defined for every $t \geq t_o$.
\end{proofof}

\begin{proofof}{Proposition~\ref{prop:ODEcoupled}}
  Let $F_1$ be the first component of the local flow $F$ defined
  in~\eqref{eq:36}.

  Let $t \in [0, T]$ be a Lebesgue point of the map
  $t \mapsto f\left(t,P(t,t_o)(u_o, w_o)\right)$. Choose $h$ small so
  that $t+h \in [0,T]$ and set $(u,w) = P(t,t_o)(u_o, w_o)$. Then,
  \begin{eqnarray*}
    &
    & \norma{\frac{P_1(t + h,t_o)(u_o, w_o) - P_1(t,t_o)(u_o, w_o)}{h} - f(t,P(t,t_o)(u_o, w_o))}
    \\
    & =
    & \norma{\frac{P_1(t+h,t)(u, w) - u}{h} - f(t,u,w)}
    \\
    & \leq
    & \norma{\frac{P_1(t+h, t)(u, w) - F_1(h, t) (u,w)}{h}}
      + \norma{\frac{F_1(h, t)(u, w) - u}{h} - f(t, u, w)}
    \\
    & =
    & R_1(h) + R_2(h) \,.
  \end{eqnarray*}
  Considering first the term $R_1$, we use
  estimate~\eqref{eq:tangent}, giving
  \begin{eqnarray*}
    R_1(h)
    & =
    & \norma{\frac{P_1(t+h, t)P(t, t_o)(u_o, w_o) - F_1(h, t)P(t, t_o)(u_o, w_o)}{h}}
    \\
    & \leq
    & \dfrac{2\, L}{\ln 2} \int_0^h \frac{\omega(\xi)}{\xi}\d\xi \longrightarrow0, \mbox{ as } h\to 0
  \end{eqnarray*}
  with $L$ and $\omega$ as in~\eqref{eqn:tangM}.  For $R_2$, we have
  \begin{eqnarray*}
    R_2(h)
    & =
    & \norma{\frac{1}{h} \int_0^h
      f(t+\tau, F_1(\tau, t)(u, w), w) \d\tau - f(t, u, w)}
    \\
    & =
    & \norma{\frac{1}{h} \int_0^h
      \left[
      f(t+\tau, F_1(\tau, t)(u, w), w) - f(t, u, w)
      \right]\ \d\tau
      }
    \\
    & \leq
    & \norma{\frac{1}{h} \int_0^h
      \left[
      f(t+\tau, F_1(\tau, t)(u, w), w)
      - f(t+\tau, P_1(t+\tau, t)(u, w), w)
      \right]\
      d\tau}
    \\
    & \quad
    &  + \norma{\frac{1}{h} \int_0^h
      \left[
      f(t+\tau, P_1(t+\tau, t)(u, w), w)
      - f(t+\tau, P_1(t+\tau, t)(u, w), P_2(t+\tau, t)(u, w))
      \right]
      \d\tau}
    \\
    & \quad
    &  + \norma{\frac{1}{h} \int_0^h
      \left[
      f(t+\tau, P_1(t+\tau, t)(u, w), P_2(t+\tau, t)(u, w)) - f(t, u, w)
      \right]
      \d\tau}
    \\
    & =
    & R_{2,1}(h) + R_{2, 2}(h) + R_{2, 3}(h).
  \end{eqnarray*}
  We have, as $f$ is Lipschitz
  continuous, and using~\eqref{eq:tangent}--\eqref{eqn:tangM}, that
  \begin{eqnarray*}
    R_{2, 1}(h)
    & \leq
    & \frac{L_f}{h} \int_0^{h}
      \norma{F_1(\tau, t)(u, w) - P_1(t+\tau, t)(u, w)} \d\tau
    \\
    & \leq
    & \dfrac{2\, L}{\ln 2} \;
      \frac{L_f}{h} \int_0^h \tau \int_0^\tau \frac{\omega(\xi)}{\xi} \d\xi\ d\tau
    \\
    & \to
    & 0 \mbox{ as } h \to 0+ \,;
    \\
    R_{2, 2}(h)
    & \leq
    & \frac{L_f}{h}\int_0^h
      \norma{P_2(t+\tau, t)(u, w) - P_2(t, t)(u, w)} \d\tau
    \\
    & \leq
    & \frac{L_f \cdot L_P}{h}
      \int_0^h \tau \, \d\tau
    \\
    & \to
    &0 \mbox{ as } h\to 0 + \,;
    \\
    R_{2,3}(h)
    & \leq
    & \int_0^h \frac{1}{h}
      \norma{
      f\left(t+\tau, P(t+\tau, t_o)(u_o, w_o)\right)
      -
      f(t, P(t, t_o)(u_o, w_o))}
      \d\tau
    \\
    & \to
    & 0 \mbox{ as } h\to 0 + \,,
  \end{eqnarray*}
  the latter convergence following from the choice of $t$ as a Lebesgue
  point.
\end{proofof}
}

\subsection{Proofs for~\S~\ref{subsec:IVP}}
\label{subsubsec:proofs-related-s-IP}

With reference to~\eqref{eq:REN} and~\eqref{eq:IBVP2Eq}, introduce for
$\bar t, t \in \hat I$ and $\bar x, x \in \reali_+$ the characteristics
\begin{equation}
  \label{def:BalLawChar}
  t \mapsto \mathcal{X} (t;\bar t,\bar x)
  \mbox{ solves }
  \left\{
    \begin{array}{l}
      \dot x = v (t,x,w)
      \\
      x (\bar t) = \bar x,
    \end{array}
  \right.
  \mbox{ and }
  t \mapsto \mathcal{T} (x;\bar x,\bar t)
  \mbox{ solves }
  \left\{
    \begin{array}{l}
      t' = 1 / v (t,x,w)
      \\
      t (\bar x) = \bar t \,,
    \end{array}
  \right.
\end{equation}
and in the sequel we omit the dependence on $w$. As is well known, see
for instance~\cite[Lemma~5]{MR4371486} and the references therein, the
unique solution to~\eqref{eq:REN} is
\begin{equation}
  \label{eq:64}
  u (t,x)
  =
  u_o \left(\mathcal{X}(t_o; t, x)\right) \,  \mathcal{E}_w(t_o,t,x)
  +
  \int_{t_o}^t q\left(s, \mathcal{X}(s; t, x),w\right) \,
  \mathcal{E}_w(s,t,x) \d{s}
\end{equation}
where the characteristics $\mathcal{X}$ are defined by
\eqref{def:BalLawChar} and
\begin{displaymath}
  \mathcal{E}_w(\tau,t,x)
  =
  \exp
  \int_\tau^t \left(
    m\left(s, \mathcal{X} (s;t,x), w\right)
    -
    \div v \left(s, \mathcal{X} (s;t,x)\right)
  \right)
  \d{s} \,.
\end{displaymath}
Below, we often use the substitution $y \leftrightarrow x$, where
\begin{equation}
  \label{eq:69}
  y = \mathcal{X} (t;t_o,x)
  \quad \mbox{ with Jacobian } \quad
  J (t,y)
  =
  \exp \left(
    \int_t^{t_o} \nabla \cdot v\left(s,\mathcal{X} (s;\tau,y)\right)\d{s}
  \right) \,,
\end{equation}
for more details see for instance~\cite[Proof of
Proposition~3]{MR4371486}.

\begin{lemma}
  \label{lem:BV}
  Assume~\ref{ip:(V)} holds and use the
  notation~\eqref{def:BalLawChar}. Let
  $u \in (\L1\cap \BV) (\reali^n; \reali)$. Then, for all
  $t_o,t \in \hat I$
  \begin{equation}
    \label{eq:67}
    \int_{\reali^n}
    \modulo{u\left(\mathcal{X} (t;t_o,x)\right) - u (x)} \d{x}
    \leq
    \dfrac{V_\infty}{V_L} \, \left(e^{V_L \modulo{t-t_o}} -1\right) \, \tv (u) \,.
  \end{equation}
\end{lemma}

\noindent This Lemma is an extension
of~\cite[Lemma~2.3]{BressanLectureNotes} to $\reali^n$.

{%
  \begin{proofof}{Lemma~\ref{lem:BV}}
    Along the same lines of~\cite[Lemma~3.24]{AmbrosioFuscoPallara},
    thanks to~\cite[Theorem~3.9]{AmbrosioFuscoPallara}, we assume that
    $u \in (\C1 \cap \BV) (\reali^n; \reali)$. Then, using the change
    of coordinates~\eqref{eq:69},
    \begin{eqnarray*}
      &
      & \int_{\reali^n}
        \modulo{u\left(\mathcal{X} (t;t_o,x)\right) - u (x)}  \d{x}
      \\
      & =
      & \int_{\reali^n} \modulo{
        \int_{t_o}^t
        \nabla u\left(\mathcal{X} (\tau;t_o,x)\right) \;
        v\left(\tau, \mathcal{X} (\tau;t_o,x)\right) \, \d\tau
        } \d{x}
      \\
      & \leq
      & \modulo{
        \int_{t_o}^t
        \int_{\reali^n}
        \norma{\nabla u\left(\mathcal{X} (\tau;t_o,x)\right)} \;
        \norma{v\left(\tau, \mathcal{X} (\tau;t_o,x)\right)} \,
        \d{x} \, \d\tau}
      \\
      & \leq
      & \modulo{
        \int_{t_o}^t \! \int_{\reali^n} \!\!
        \norma{\nabla u(y)} \,
        \norma{v(\tau, y)} \,
        \exp \modulo{
        \int_\tau^t \nabla \cdot v\left(\tau,\mathcal{X} (s;\tau,y)\right) \d{s}}
        \d{y} \d\tau}
      \\
      & \leq
      & V_\infty \, \norma{\nabla u}_{\L1 (\reali^n; \reali^n)}
        \modulo{\int_{t_o}^t e^{V_L (t-\tau)} \d\tau}
      \\
      & =
      & \dfrac{V_\infty}{V_L} \, \left(e^{V_L \modulo{t-t_o}} -1\right) \,
        \norma{\nabla u}_{\L1 (\reali^n; \reali^n)}\,,
    \end{eqnarray*}
    which yields~\eqref{eq:67}.
  \end{proofof}
}

Define the parameterized mapping $P^w$ by
\begin{equation}\label{def:PwIVP}
  \begin{array}{@{}c@{\,}c@{\,}c@{\,}c@{\,}c@{}}
    P^w
    & \colon
    & \mathcal{A}
    & \to
    & \mathcal{U}
    \\
    &
    & \left(t,t_o,u_o\right)
    & \mapsto
    & u(t)
  \end{array}
  \quad\mbox{ where }\quad
  \begin{array}{@{}l@{}}
    u (t)
    \mbox{ is given by~\eqref{eq:64};}
  \end{array}
\end{equation}
Below, by~\ref{ip:(V)} and~\ref{ip:(P)}, for all $t,\tau \in \hat I$,
$x \in \reali^n$ and $w \in \mathcal{W}$, we use the uniform estimate
\begin{equation}
  \label{eq:68}
  0
  \leq
  \mathcal{E}_w(\tau,t,x)
  \leq
  e^{(M_\infty+V_L)\modulo{t-\tau}} \,.
\end{equation}

\begin{lemma}\label{lem:PwProcessIVP}
  For all $w \in \mathcal{W}$, $P^w$ in~\eqref{def:PwIVP} is a global
  process according to Definition~\ref{def:Global}.
\end{lemma}

\begin{proofof}{Lemma~\ref{lem:PwProcessIVP}}
  That $P^w$ satisfies~\eqref{eq:process0} is an immediate consequence
  of its definition~\eqref{eq:64}. The uniqueness of the solution
  ensures that~\eqref{eq:process2} is satisfied.

  Fix $t_o, t \in I$, with $t_o \leq t$, and
  $r_o \in \mathcal{D}_{t_o}$.  It remains to
  show~\eqref{eq:process1}, that is,
  $u(t) = P^w(t, t_o)u_o \in \mathcal{D}_{t}$ for each
  $w \in \mathcal{W}$.

  \paragraph{1.}
  We begin by showing that, if
  $\norma{u_o}_{\L1(\reali^n;\reali)} \leq \alpha_1(t_o)$, then
  $ \norma{u(t)}_{\L1(\reali^n; \reali)} \leq \alpha_1 (t)$. Making
  use of~\ref{ip:(P)}--\ref{ip:(Q)}--\eqref{eq:65}%
  --\eqref{eq:64}--\eqref{eq:69}, see also~\cite[Proposition~3,
  \textbf{(H3)}]{MR4371486},
  \begin{eqnarray}
    \nonumber
    &
    & \norma{u(t)}_{\L1(\reali^n;\reali)}
    \\
    \label{eq:63}
    & \leq
    & \left(
      \norma{u_o}_{\L1(\reali^n; \reali)}
      +
      \norma{q (\cdot,\cdot,w)}_{\L1([t_o, t] \times \reali^n; \reali)}
      \right)
      \exp \left(
      \int_{t_o}^t \norma{m (\tau,\cdot,w)}_{\L\infty(\reali^n; \reali)} \d{\tau}
      \right)
    \\
    \nonumber
    & \leq
    & \left(
      \alpha_1(t_o)  + Q_1(t - t_o)
      \right) e^{M_\infty(t - t_o)}
    \\
    \nonumber
    & \leq
    & \left(
      Re^{-M_\infty (T - t_o)} - Q_1(T- t_o)e^{M_\infty t_o}  + Q_1(t - t_o)
      \right) e^{M_\infty(t - t_o)}
    \\
    \nonumber
    & \leq
    & Re^{-M_\infty (T - t)} - Q_1(T- t)e^{M_\infty t}
    \\
    \nonumber
    & =
    & \alpha_1(t) \,,%
  \end{eqnarray}
  as required.

\paragraph{2.}
Assuming now that
$ \norma{u_o}_{\L\infty(\reali^n;\reali)} \leq \alpha_\infty(t_o), $
we show that
$ \norma{u(t)}_{\L\infty(\reali^n;\reali)} \leq \alpha_\infty(t), $ We
use~\eqref{eq:65}--\eqref{eq:64}, see also~\cite[Proposition~3,
\textbf{(H4)}]{MR4371486}, together with~\ref{ip:(V)}, \ref{ip:(P)},
\ref{ip:(Q)} and~\eqref{eq:68}.  Then,
\begin{eqnarray*}
  \norma{u(t)}_{\L\infty(\reali^n;\reali)}
  & \leq
  & \left(
    \norma{u_o}_{\L\infty(\reali^n;\reali)}
    + \int_{t_o}^t \norma{q(s, \cdot,
    w)}_{\L\infty(\reali^n;\reali)} \d{s}
    \right)
  \\
  & \quad
  & \times \exp \left(
    \int_{t_o}^t
    \left(
    \norma{m(s, \cdot, w)}_{\L\infty(\reali^n;\reali)}
    +
    \norma{\nabla\cdot v(s)}_{\L\infty(\reali^n;\reali)}
    \right) \d{s}
    \right)
  \\
  & \leq
  & \left(
    \norma{u_o}_{\L\infty(\reali^n;\reali)} + Q_\infty(t-t_o)
    \right) e^{(M_\infty+V_L) (t-t_o)}
  \\
  & \leq
  & \left(
    \alpha_\infty (t_o) + Q_\infty(t-t_o)
    \right) e^{(M_\infty+V_L) (t-t_o)}
  \\
  & \leq
  & \left(
    R e^{-(M_\infty +V_L)(T-t_o)} - Q(T - t_o)e^{(M_\infty +V_L)t_o}
    \right)
    e^{(M_\infty +V_L) (t-t_o)}
  \\
  &
  & + Q (t-t_o) e^{(M_\infty +V_L) (t-t_o)}
  \\
  & \leq
  & R e^{-(M_\infty +V_L)(T- t)} - Q(T - t) e^{(M_\infty +V_L) t}
  \\
  & =
  & \alpha_\infty (t) \,,
\end{eqnarray*}
as required.

\paragraph{3.}
Finally, we show that, if $u_o \in \mathcal{D}_{t_o}$, then
$\tv\left(u(t)\right) \leq \alpha_{TV}(t)$.  We
use~\ref{ip:(V)}--\ref{ip:(P)}--\ref{ip:(Q)}--\eqref{eq:65}%
--\eqref{eq:64}--\eqref{eq:69}--\eqref{eq:68}, see
also~\cite[Formula~(31)]{MR4371486}:
\begin{eqnarray}
  \label{eq:18}
  \tv \left(u(t)\right)
  & \leq
  & \bigg[
    \tv(u_o)
    + \int_{t_o}^{t} \tv\left(q(s, \cdot, w)\right) \d{s}
  \\
  \nonumber
  &
  & +\left(
    \norma{u_o}_{\L\infty(\reali^n;\reali)}
    +
    \int_{t_o}^{t} \norma{q(s, \cdot, w)}_{\L\infty(\reali^n;\reali)} \d{s}
    \right)
  \\
  \nonumber
  &
  & \times
    \int_{t_o}^t
    \left(
    \tv\left(m(s, \cdot, w)\right)
    + \norma{\nabla \nabla \cdot v(s)}_{\L1(\reali^n;\reali^n)}
    \right) \d{s}
    \bigg]
    e^{(M_\infty+V_L)\modulo{t-\tau}}
\end{eqnarray}
Since $u_o \in \mD_{t_o}$, by~\eqref{eq:29},
$ \tv(u_o) \leq \alpha_{\tv} (t_o)$ and we have
that~\eqref{eq:18} becomes
\begin{eqnarray*}
  &
  & \tv\left(u(t)\right)
  \\
  & \leq
  & \bigg[
    \alpha_{\tv} (t_o) + Q_\infty(t-t_o)
  \\
  &
  & + \left(
    R e^{-(M_\infty+V_L)(T-t_o)} - Q_\infty e^{(M_\infty+V_L)t_o} (T- t_o)
    + Q_\infty(t-t_o)
    \right)(M_\infty+V_1) (t-t_o)
    \bigg]
  \\
  &
  & \qquad \times e^{(M_\infty+V_L) (t-t_o)}
  \\
  & \leq
  & R \left(1 - (M_\infty + V_1)(T - t)\right)
    e^{-(M_\infty+V_L) (T-t)}
        - Q_\infty \left(1+ (M_\infty+V_1)t\right) (T-t_o) \,
        e^{(M_\infty+V_L) t}
  \\
  &
  & + Q_\infty (t-t_o) \left(1+ (M_\infty+V_1) (t-t_o)\right)
    e^{(M_\infty+V_L) (t-t_o)}
  \\
  & \leq
  & R \left(1 - (M_\infty + V_1)(T - t)\right)
    e^{-(M_\infty+V_L) (T-t)}
    - Q_\infty \left(1+ (M_\infty+V_1)t\right) (T-t) \,
    e^{(M_\infty+V_L) t}
  \\
  & =
  & \alpha_{\tv} (t) \,,%
\end{eqnarray*}
completing the proof of~\eqref{eq:process1}.
\end{proofof}

\begin{proofof}{Proposition~\ref{prop:RE}}
  We define the mapping $P^w$ by~\eqref{def:PwIVP}.  That this defines
  a process is a consequence of Lemma~\ref{lem:PwProcessIVP}.

  It remains to show the three Lipschitz continuity
  estimates~\eqref{eq:6}, \eqref{eq:7}, and \eqref{eq:8}.

  \paragraph{1. Lipschitz continuity w.r.t initial data} By the linear
  structure of~\eqref{eq:REN}, from~\eqref{eq:63} we immediately have
  \begin{displaymath}
    \norma{P^w(t, t_o)(u_o - \bar{u}_o)}_{\L1(\reali^n;\reali)}
    \leq
    e^{M_\infty(t - t_o)} \; \norma{u_o - \bar{u}_o}_{\L1(\reali^n;\reali)}
  \end{displaymath}
  which is compatible with the choice of $C_u$ in~\eqref{eq:29}.

  \paragraph{2. Lipschitz continuity in time} By direct computations
  based on~\eqref{eq:64}, for $t \geq t_o$:
  \begin{eqnarray*}
    &
    & \norma{P^w (t,t_o)u_o - u_o}_{\L1(\reali_+;\reali)}
    \\
    & \leq
    & \int_{\reali^n} \modulo{u_o\left(\mathcal{X}(t_o; t, x)\right) - u_o(x)}
      \, \mathcal{E}_w(t_o,t,x)
      \d{x}
    \\
    &
    & +
      \int_{\reali^n} \int_{t_o}^t
      \modulo{q\left(\tau, \mathcal{X}(\tau; t, x),w\right)}  \,
      \mathcal{E}_w(\tau,t,x)
      \d{\tau} \d{x}
      +
      \int_{\reali^n}
      \modulo{u_o(x)} \, \modulo{\mathcal{E}_w(t_o,t,x) -1} \d{x}
  \end{eqnarray*}
  and we consider the latter three terms separately. First,
  use~\eqref{eq:68} and Lemma~\ref{lem:BV}, for $t \geq t_o$,
  \begin{eqnarray*}
    \int_{\reali^n} \modulo{u_o\left(\mathcal{X}(t_o; t, x)\right) - u_o(x)}
    \mathcal{E}_w(t_o,t,x)
    \d{x}
    & \leq
    & \int_{\reali^n} \modulo{u_o\left(\mathcal{X}(t_o; t, x)\right) - u_o(x)} \d{x}
      \, e^{(M_\infty+V_L) (t-t_o)}
    \\
    & \leq
    & \dfrac{V_\infty}{V_L} \, \left(e^{V_L \modulo{t-t_o}} -1\right) \, \tv (u_o)
      \, e^{(M_\infty+V_L) (t-t_o)}
    \\
    & \leq
    & V_\infty \, \tv (u_o) \, e^{(M_\infty+2V_L) (t-t_o)} \, (t-t_o) \,.
  \end{eqnarray*}
  {To deal with the second term, use the change of
    coordinates~\eqref{eq:69} and~\ref{ip:(P)}--\ref{ip:(Q)}:
    \begin{eqnarray*}
      &
      & \int_{\reali^n} \int_{t_o}^t
        \modulo{q\left(\tau, \mathcal{X}(\tau; t, x),w\right)} \,
        \mathcal{E}_w(\tau,t,x)
        \d{\tau} \d{x}
      \\
      & =
      & \int_{\reali^n} \int_{t_o}^t
        \modulo{q(\tau,y,w)}
        \,
        \exp
        \left(
        \int_\tau^t m\left(s,\mathcal{X} (s;\tau,y), w\right) \d{s}
        \right)
        \d{\tau} \d{y}
      \\
      & \leq
      & Q_1 \, e^{M_\infty (t-t_o)} \, (t-t_o) \,.
    \end{eqnarray*}
  }%
  Finally, the third term is treated as follows,
    by~\eqref{eq:68}:
    \begin{eqnarray*}
      \int_{\reali^n}
      \modulo{u_o(x)} \, \modulo{\mathcal{E}_w(t_o,t,x) -1} \d{x}
      & \leq
      & \int_{\reali^n}
        \modulo{u_o (x)} \,
        e^{(M_\infty+V_L) (t-t_o)}
        (M_\infty + V_L) (t-t_o)
        \d{x}
      \\
      & \leq
      & (M_\infty+V_L) \,
        \norma{u_o}_{\L1 (\reali^n; \reali)} \,
        e^{(M_\infty+V_L) (t-t_o)} (t-t_o) \,.
    \end{eqnarray*}
   Adding up, we have
  \begin{eqnarray*}
    \norma{P^w (t,t_o)u_o - u_o}_{\L1(\reali_+;\reali)}
    & \leq
    & V_\infty \, \tv (u_o) \, e^{(M_\infty+2V_L) (t-t_o)} \, (t-t_o)
    \\
    &
    & + Q_1 \, e^{M_\infty (t-t_o)} \, (t-t_o)
    \\
    &
    & + (M_\infty+V_L) \, \norma{u_o}_{\L1 (\reali^n; \reali)} \,
      e^{(M_\infty+V_L) (t-t_o)} (t-t_o)
  \end{eqnarray*}
  which agrees with the choice of $C_t$ in~\eqref{eq:29}.

  \paragraph{3. Lipschitz continuity w.r.t parameters}
  {From~\cite[{\bf(H5)}]{MR4371486}, using~\ref{ip:(V)},
    \ref{ip:(P)}, and \ref{ip:(Q)},
    \begin{eqnarray*}
      &
      & \norma{P_w(t, t_o)u_o - P_{\bar{w}}(t, t_o)u_o}_{\L1(\reali^n:\reali)}
      \\
      & \leq
      & \int_{t_o}^{t}
        \norma{v(\tau, \cdot, w_1) - v(\tau, \cdot, w_2)}_{\L\infty(\reali^n; \reali^n)}
        \d{\tau}
      \\
      &
      & \times \bigg[
        \norma{u_o}_{\L\infty(\reali^n;\reali^n)}
        + \tv(u_o)
      \\
      &
      & + \int_{t_o}^{t}
        \left(
        \max_{\omega = w_1, w_2} \norma{q(\tau, \cdot, \omega)}_{\L\infty(\reali^n;\reali)}
        + \max_{\omega = w_1, w_2} \tv(q(\tau, \cdot, \omega))
        \right)
        \d{\tau}
        \bigg]
      \\
      &
      & \times
        \exp\left(
        \int_{t_o}^{t}
        \left(
        \max_{\omega = w_1, w_2} \norma{m(\tau, \cdot, \omega)}_{\L\infty(\reali^n;\reali)}
        + \max_{\omega = w_1, w_2} \norma{\nabla v(\tau, \cdot, \omega)}_{\L\infty(\reali^n;\reali^{n\times n})}
        \right)
        \d{\tau}
        \right)
      \\
      &
      &
        \times \left[
        1
        + \int_{t_o}^{t}
        \max_{\omega = w_1, w_2} \left(
        \norma{\nabla \nabla \cdot v(\tau, \cdot, \omega)}_{\L\infty(\reali^n;\reali^n)}
        + \max_{\omega = w_1, w_2} \tv(m(\tau, \cdot, \omega))
        \right)
        \d{\tau}
        \right]
      \\
      &
      & +\bigg[
        \int_{t_o}^{t}
        \norma{q(\tau, \cdot, w) - q(\tau, \cdot, \bar{w})}_{\L1(\reali^n;\reali)}
        \d\tau
      \\
      &
      &  + \int_{t_o}^{t}
        \left(
        \norma{m(\tau, \cdot, w) - m(\tau, \cdot, \bar{w})}_{\L1(\reali^n;\reali)}
        + \norma{\nabla \cdot (v(\tau, \cdot, w) - v(\tau, \cdot, \bar{w}))}_{\L1(\reali^n;\reali)}
        \right)
        \d\tau
      \\
      &
      &
        \times \bigg(
        \norma{u_o}_{\L\infty(\reali^n;\reali)}
        + \int_{t_o}^{t}
        \max_{\omega = w, \bar{w}}
        \norma{q(\tau, \cdot, \omega)}_{\L\infty(\reali^n;\reali)} \d\tau
        \bigg)
        \bigg]
      \\
      &
      & \times \exp\left(
        \int_{t_o}^{t}
        \max_{\omega = w, \bar{w}}
        \norma{m(\tau, \cdot, \omega)}_{\L\infty(\reali^n;\reali)} \d\tau
        \right)
      \\
      & \leq
      &
        \left[
        V_L
        (2R + Q_\infty)
        \left(1 + (V_1 + M_\infty)(t - t_o)\right)
        + (
        Q_L
        + (M_L + V_L) \left(R + Q_\infty(t - t_o)\right)
        )
        \right]
      \\
      &
      & \times
        e^{(M_\infty + V_L)(t - t_o)}
        (t - t_o) d_\mathcal{W}(w_1, w_2)
      \\
      & \leq
      &
        \left[
        V_L
        (2R + Q_\infty)
        (1 + (V_1 + M_\infty)\hat{T})
        + (
        Q_L
        + (M_L + V_L) (R + Q_\infty \hat{T})
        )
        \right]
      \\
      &
      & \times
        e^{(M_\infty + V_L)\hat{T}}
        (t - t_o) d_\mathcal{W}(w_1, w_2)
    \end{eqnarray*}
    in agreement with the choice of $C_w$ in~\eqref{eq:29}.}

  \paragraph{Choice of $T$.}
  The time $T$ has to be chosen so that $\alpha_1 (0) > 0$,
  $\alpha_\infty (0) > 0$ and $\alpha_{\tv} (0) > 0$. Clearly,
  by~\eqref{eq:65}, for $T$ sufficiently small, these requirements are
  all met.
\end{proofof}

\begin{proofof}{Corollary~\ref{cor:global-existence-IVP}}
  Note that the constants defined in \ref{ip:(V)}, \ref{ip:(P)},
  and~\ref{ip:(Q)} do not depend on $R$. Moreover $T$ has to be
  chosen such that $\alpha_1(0) > 0$, $\alpha_\infty(0) > 0$ and
  $\alpha_{\tv}(0) > 0$, which are equivalent to
  \begin{equation*}
    \left\{
      \begin{array}{l}
        R e^{-M_\infty T} - Q_1 T > 0
        \\
        R e^{-(M_\infty + V_L)T} - Q_\infty T > 0
        \\
        R e^{-(M_\infty + V_L)T}\left(1 - \left(M_\infty + V_1\right)T\right)
        - Q_\infty T > 0.
      \end{array}
    \right.
  \end{equation*}
  The proof ends setting
  $T = \min\left\{\frac{1}{2(M_{\infty} +V_1)}, \frac{\ln(2)}{M_\infty
      + V_L}\right\}$, provided $R$ is sufficiently big.
\end{proofof}

\begin{proofof}{Proposition~\ref{prop:RENCoup}}
  The Lipschitz continuity of $P$ ensured by Theorem~\ref{thm:Metric}
  shows that $P_1$ is $\L1$--Lipschitz continuous, and hence in
  $\C0([t_o, T];\L1(\reali^n;\reali))$ as required.

  We focus our attention now on the first item in
  Definition~\ref{def:RENSol}, the second being immediate.  To ease
  reading, for any test function
  $\phi \in \Cc\infty (\mathopen]t_o, T\mathclose[ \times \reali^n;
  \reali)$ we introduce the notation
  \begin{equation}
    \label{eqn:RENIphi}
    \mathcal{I}_\phi (u,w)
    =
    u \, \partial_t\phi
    +
    u\, v \cdot \nabla_x \phi
    +
    \left(m (\cdot, \cdot,w)\, u + q (\cdot,\cdot,w) \right)
    \phi.
  \end{equation}
  We want to prove that, for any
  $\phi \in \Cc\infty (\mathopen]t_o, T\mathclose[ \times \reali^n;
  \reali)$,
  \begin{displaymath}
    \int_{\reali^n} \int_{t_o}^T \mathcal{I}_\phi  \left(P (t,t_o) (u_o,w_o)\right) \d{t} \, \d{x} = 0.
  \end{displaymath}
  We begin by discretising the time domain. For a given
  $k \in \naturali \setminus \left\{0\right\}$ and $i=0, \ldots, k$,
  introduce $t_i = t_o + i (T-t_o)/k$ and
  $({\tilde u_i}, {\tilde w}_i) = P (t_{i-1},t_o)
  (u_o,w_o)$. Splitting the integral then gives
  \begin{eqnarray}
    \nonumber
    &
    & \int_{t_o}^T \int_{\reali^n}
      \mathcal{I}_\phi  \left(P (t,t_o) (u_o,w_o)\right) \d{x} \, \d{t}
    \\
    \nonumber
    & =
    & \sum_{i=1}^k
      \int_{t_{i-1}}^{t_i} \int_{\reali^n}
      \left(
      \mathcal{I}_\phi  \left(P (t,t_{i-1}) ({\tilde u_i}, {\tilde w}_i)\right)
      -
      \mathcal{I}_\phi  \left(
      F (t-t_{i-1},t_{i-1}) ({\tilde u_i}, {\tilde w}_i)
      \right)
      \right)
      \d{x} \, \d{t}
    \\
    \label{eq:20}
    &
    & + \sum_{i=1}^k
      \int_{t_{i-1}}^{t_i} \int_{\reali^n}
      \mathcal{I}_\phi  \left(
      F (t-t_{i-1},t_{i-1}) ({\tilde u_i}, {\tilde w}_i)
      \right) \d{x} \, \d{t} \,.
  \end{eqnarray}
  We compute the terms on the last two lines separately, our goal is
  to show that they both converge to zero as $k \to \infty$.

  For the first,
  \begin{eqnarray}
    \nonumber
      &
      & \mathcal{I}_\phi  \left(P (t,t_{i-1}) ({\tilde u_i}, {\tilde w}_i)\right)
        -
        \mathcal{I}_\phi  \left(
        F (t-t_{i-1},t_{i-1}) ({\tilde u_i}, {\tilde w}_i)
        \right)
    \\
    \label{eq:14}
      & =
      & \partial_t \phi
        \left(
        P_1 (t,t_{i-1}) ({\tilde u_i}, {\tilde w}_i)
        -
        F_1 (t-t_{i-1},t_{i-1}) ({\tilde u_i}, {\tilde w}_i)
        \right)
    \\
    \nonumber
      &
      &
        + \bigl(
        P_1 (t,t_{i-1}) ({\tilde u_i}, {\tilde w}_i)
        v(t, x, P_2(t, t_{i-1})({\tilde u_i}, {\tilde w}_i))
    \\
    \label{eqn:DxPhiDiff}
      &
      & \qquad
        -
       	F_1 (t-t_{i-1},t_{i-1}) ({\tilde u_i}, {\tilde w}_i)
        v(t, x, F_2(t-t_{i-1}, t_{i-1})({\tilde u_i}, {\tilde w}_i))
        \bigr)
        \cdot \nabla_x \phi
    \\
    \label{eq:15}
      &
      & + \bigl(
        m\left(t,x, P_2(t,t_{i-1}) ({\tilde u_i}, {\tilde w}_i)\right)
        P_1(t,t_{i-1}) ({\tilde u_i}, {\tilde w}_i)
    \\
    \label{eq:16}
      &
      & \qquad
        -
        m \left(t,x, F_2(t-t_{i-1},t_{i-1}) ({\tilde u_i}, {\tilde w}_i)\right)
        F_1(t-t_{i-1},t_{i-1}) ({\tilde u_i}, {\tilde w}_i)
        \bigr) \phi
    \\
    \label{eq:17}
      &
      & +
        \left(
        q\left(t,x,P_2(t,t_{i-1}) ({\tilde u_i}, {\tilde w}_i)\right)
        -
        q\left(t,x,F_2(t-t_{i-1},t_{i-1}) ({\tilde u_i}, {\tilde w}_i)\right)
        \right) \phi.
  \end{eqnarray}
  Recall that the tangency condition~\eqref{eq:tangent} ensures
  \begin{eqnarray*}
    \frac{1}{t-t_{i-1}} \,
    \norma{
    P_1(t,t_{i-1})({\tilde u_i}, {\tilde w}_i)
    -
    F_1(t - t_{i-1},t_{i-1})({\tilde u_i}, {\tilde w}_i)}_{\L1 (\reali^n; \reali)}
    & \leq
    & \frac{2L}{\ln (2)}
      \int_{0}^{t-t_{i-1}} \frac{\omega(\xi)}{\xi} \, \d{\xi}
    \\
    \frac{1}{t-t_{i-1}} \,
    d_\mathcal{W} \left(
    P_2(t,t_{i-1})({\tilde u_i}, {\tilde w}_i)
    ,
    F_2(t - t_{i-1},t_{i-1})({\tilde u_i}, {\tilde w}_i)
    \right)
    & \leq
    & \frac{2L}{\ln (2)}
      \int_{0}^{t-t_{i-1}} \frac{\omega(\xi)}{\xi} \, \d{\xi}
  \end{eqnarray*}
  with $L$ and $\omega$ defined as in~\eqref{eqn:tangM}, so that,
  considering~\eqref{eq:14},
  \begin{eqnarray}
    \nonumber
    &
    & \modulo{
      \int_{t_{i-1}}^{t_i} \int_{\reali^n}
      \left(\partial_t \phi\ \right)
      \left(
      P_1 (t,t_{i-1}) ({\tilde u_i}, {\tilde w}_i)
      -
      F_1 (t-t_{i-1},t_{i-1}) ({\tilde u_i}, {\tilde w}_i)
      \right)
      \d{x} \, \d{t}
      }
    \\
    \nonumber
    & \leq
    & \norma{\partial_t \phi\ }_{\L\infty ([0,T]\times\reali^n; \reali)}
      \int_{t_{i-1}}^{t_i}
      \norma{
      P_1(t,t_{i-1})({\tilde u_i}, {\tilde w}_i)
      -
      F_1(t - t_{i-1},t_{i-1})({\tilde u_i}, {\tilde w}_i)}_{\L1 (\reali^n; \reali)}
      \d{t}
    \\
    \label{eq:23}
    & \leq
    & \frac{L}{\ln (2)} \,
      \norma{\partial_t \phi}_{\L\infty ([0,T]\times\reali^n; \reali)} \;
      (t_i-t_{i-1})^2
      \int_{0}^{t_i-t_{i-1}} \frac{\omega(\xi)}{\xi} \, \d{\xi} \,.
  \end{eqnarray}
  Considering the next term~\eqref{eqn:DxPhiDiff},
  \begin{eqnarray}
    \nonumber
    &
    & \int_{t_{i-1}}^{t_i} \int_{\reali^n}^{}
      \bigl[
      P_1 (t,t_{i-1}) ({\tilde u_i}, {\tilde w}_i)
      v(t, x, P_2(t, t_{i-1})({\tilde u_i}, {\tilde w}_i))
    \\
    \nonumber
    &
    & \hphantom{\int_{t_{i-1}}^{t_i} \int_{\reali^n}^{} \bigl[ }
      - F_1 (t - t_{i-1},t_{i-1}) ({\tilde u_i}, {\tilde w}_i)
      v(t, x, F_2(t - t_{i-1}, t_{i-1})({\tilde u_i}, {\tilde w}_i))
      \bigr]
      \cdot \nabla_x \phi
      \d{t} \d{x}
    \\
    \label{eqn:DxPhiDiff1}
    & =
    &
      \int_{t_{i-1}}^{t_i} \int_{\reali^n}^{}
      \bigl[
      P_1 (t,t_{i-1}) ({\tilde u_i}, {\tilde w}_i)
      - F_1 (t - t_{i-1} ,t_{i-1}) ({\tilde u_i}, {\tilde w}_i)
      \bigr]
    \\
    \nonumber
    &
    & \qquad \times
      v(t, x, P_2(t, t_{i-1})({\tilde u_i}, {\tilde w}_i))
      \cdot \nabla_x \phi
      \d{t} \d{x}
    \\
    \nonumber
    &
    & +
      \int_{t_{i-1}}^{t_i} \int_{\reali^n}^{}
      F_1 (t - t_{i-1},t_{i-1}) ({\tilde u_i}, {\tilde w}_i)
    \\
    \label{eqn:DxPhiDiff2}
    &
    & \qquad
      \bigl[
      v(t, x, P_2(t, t_{i-1})({\tilde u_i}, {\tilde w}_i))
      - v(t, x, F_2(t - t_{i-1}, t_{i-1})({\tilde u_i}, {\tilde w}_i))
      \bigr]
      \cdot \nabla_x \phi
      \d{t} \d{x}.
  \end{eqnarray}
  For~\eqref{eqn:DxPhiDiff1}, using~\ref{ip:(V)} and the same approach
  as for~\eqref{eq:23}, we get
  \begin{eqnarray}
    \nonumber
    \modulo{%
    \int_{t_{i{-}1}}^{t_i} \! \int_{\reali^n} \!
    \bigl[
    P_1 (t,t_{i{-}1}) ({\tilde u_i}, {\tilde w}_i)
    {-} F_1 (t {-} t_{i{-}1} ,t_{i{-}1}) ({\tilde u_i}, {\tilde w}_i)
    \bigr]
    v(t, x, P_2(t, t_{i{-}1})({\tilde u_i}, {\tilde w}_i))
    \cdot \nabla_x \phi
    \d{t} \d{x}
    }
    \\
    \label{eqn:DxPhhiDiff1fin}
    \leq
    \frac{L}{\ln(2)}\,
    V_\infty \,
    \norma{\nabla_x \phi}_{\L\infty([0, T] \times \reali^n;\reali^n)} \,
    (t_i - t_{i-1})^2
    \int_{0}^{t_i - t_{i-1}} \frac{\omega(\xi)}{\xi} \d{\xi}.
    \qquad
  \end{eqnarray}
  For the second term~\eqref{eqn:DxPhiDiff2}, using~\ref{ip:(V)}
  again, we have,
  \begin{eqnarray}
    \nonumber
    &
    &
      \bigl|
      \int_{t_{i-1}}^{t_i} \int_{\reali^n}^{}
      F_1 (t - t_{i-1},t_{i-1}) ({\tilde u_i}, {\tilde w}_i)
    \\
    \nonumber
    &
    &  \qquad \times
      \bigl[
      v(t, x, P_2(t, t_{i-1})({\tilde u_i}, {\tilde w}_i))
      - v(t, x, F_2(t - t_{i-1}, t_{i-1})({\tilde u_i}, {\tilde w}_i))
      \bigr]
      \cdot \nabla_x \phi
      \d{t} \d{x}
      \bigr|
    \\
    \nonumber
    & \leq
    &
      \int_{t_{i-1}}^{t_i}
      \norma{F_1 (t - t_{i-1},t_{i-1}) ({\tilde u_i}, {\tilde w}_i)}_{\L1(\reali^n;\reali)}
      \norma{\nabla_x \phi}_{\L\infty(\reali^n;\reali^n)}
    \\
    \nonumber
    &
    &\qquad \times
      V_L d_{\mathcal{W}}\left(
      P_2(t, t_{i-1})({\tilde u_i}, {\tilde w}_i),
      F_2(t - t_{i-1}, t_{i-1})({\tilde u_i}, {\tilde w}_i)
      \right)
      \d{t} \d{x}
    \\
    \label{eqn:DxPhhiDiff2fin}
    & \leq
    &
      \frac{L}{\ln(2)}
      R
      \norma{\nabla_x \phi}_{\L\infty([0, T] \times \reali^n;\reali^n)}
      V_L (t_i - t_{i-1})^2
      \int_0^{t_{i} - t_{i-1}}
      \frac{\omega(\xi)}{\xi}
      \d{\xi}.
  \end{eqnarray}
  Pass to~\eqref{eq:15}--\eqref{eq:16} and using again~\eqref{eq:14}:
  \begin{eqnarray}
    \nonumber
    &
    & \int_{t_{i{-}1}}^{t_i} \int_{\reali^n}
      \bigl|
      \bigl(
      m\left(t,x, P_2(t,t_{i{-}1}) ({\tilde u_i}, {\tilde w}_i)\right)
      P_1(t,t_{i{-}1}) ({\tilde u_i}, {\tilde w}_i)
    \\
    \nonumber
    &
    & \qquad
      -
      m \left(t,x, F_2(t-t_{i{-}1},t_{i{-}1}) ({\tilde u_i}, {\tilde w}_i)\right)
      F_1(t-t_{i{-}1},t_{i{-}1}) ({\tilde u_i}, {\tilde w}_i)
      \bigr) \phi
      \bigr| \d{x} \d{t}
    \\
    \nonumber
    & \leq
    & \int_{t_{i{-}1}}^{t_i}
      \norma{
      m\left(t,\cdot, P_2(t,t_{i{-}1}) ({\tilde u_i}, {\tilde w}_i)\right)
      -
      m\left(t,\cdot, F_2(t-t_{i{-}1},t_{i{-}1}) ({\tilde u_i}, {\tilde w}_i)\right)
      }_{\L1 (\reali^n; \reali)}
    \\
    \nonumber
    &
    & \qquad \times
      \norma{P_1(t,t_{i{-}1}) ({\tilde u_i}, {\tilde w}_i)}_{\L\infty (\reali^n;\reali)}
      \norma{\phi}_{\L\infty (\reali^n;\reali)}
      \d{t}
    \\
    \nonumber
    &
    & +
      \int_{t_{i{-}1}}^{t_i}
      \norma{m\left(t,\cdot, F_2(t-t_{i{-}1},t_{i{-}1}) ({\tilde u_i}, {\tilde w}_i)\right)}_{\L\infty (\reali^n; \reali)}
    \\
    \nonumber
    &
    & \qquad \times
      \norma{
      P_1(t,t_{i{-}1}) ({\tilde u_i}, {\tilde w}_i)
      -
      F_1(t-t_{i{-}1},t_{i{-}1}) ({\tilde u_i}, {\tilde w}_i)
      }_{\L1 (\reali^n; \reali)}
      \norma{\phi}_{\L\infty (\reali^n; \reali)}
      \d{t}
    \\
    \nonumber
    & \leq
    & M_L \, R \, \norma{\phi}_{\L\infty (\reali^n; \reali)}
      \int_{t_{i{-}1}}^{t_i}
      d_\mathcal{W} \left(
      P_2(t,t_{i{-}1}) ({\tilde u_i}, {\tilde w}_i)
      ,
      F_2(t-t_{i{-}1},t_{i{-}1}) ({\tilde u_i}, {\tilde w}_i)
      \right)
      \d{t}
    \\
    \nonumber
    &
    & +
      M_\infty \norma{\phi}_{\L\infty (\reali^n; \reali)} \!
      \int_{t_{i{-}1}}^{t_i} \!\!\! %
      \norma{
      P_1(t,t_{i{-}1}) ({\tilde u_i}, {\tilde w}_i)
      -
      F_1(t{-}t_{i{-}1},t_{i{-}1}) ({\tilde u_i}, {\tilde w}_i)
      }_{\L1 (\reali^n; \reali)} \d{t}
    \\
    \label{eq:24}
    & \leq
    & \frac{L}{\ln (2)} \,
      (M_L \, R + M_\infty) \norma{\phi}_{\L\infty (\reali^n; \reali)} \;
      (t_i-t_{i{-}1})^2
      \int_{0}^{t_i-t_{i{-}1}} \frac{\omega(\xi)}{\xi} \, \d{\xi} \,.
  \end{eqnarray}
  Concerning~\eqref{eq:17}, the tangency condition~\eqref{eq:tangent}
  implies
  \begin{eqnarray}
    \nonumber
    &
    & \bigg|
      \int_{t_{i-1}}^{t_i} \int_{\reali^n}  [
      q(t, x, P_2(t, t_{i-1})(\tilde{u}, \tilde{w}))
      - q(t, x, F_2(t - t_{i-1}, t_{i-1})(\tilde{u}, \tilde{w}))
      ]\phi(t)\d{x}\d{t}
      \bigg|
    \\
    \nonumber
    & \leq
    & Q_L \norma{\phi}_{\L\infty([t_o, T] \times \reali^n)}
      \int^{t_i}_{t_{i-1}}
      d_\mathcal{W} \left(
      P_2(t, t_{i-1})(\tilde{u}, {\tilde w_i}),
      F_2(t - t_{i-1}, t_{i-1})(\tilde{u}, {\tilde w_i})
      \right) \d{t}
    \\
    \label{eq:25}
    & \leq
    & \frac{L}{\ln(2)}\, Q_L\,
      \norma{\phi}_{L^\infty([t_o, T] \times \reali^n)}
      (t_i - t_{i-1})^2
      \int_{0}^{t_i - t_{i-1}} \frac{\omega(\xi)}{\xi}\d\xi \,.
  \end{eqnarray}
  Computing the sum over all time intervals, we get:
  \begin{eqnarray*}
    &
    & \sum_{i=1}^k
      \int_{t_{i-1}}^{t_i} \int_{\reali^n}
      \left(
      \mathcal{I}_\phi  \left(P (t,t_{i-1}) ({\tilde u_i}, {\tilde w}_i)\right)
      -
      \mathcal{I}_\phi  \left(
      F (t-t_{i-1},t_{i-1}) ({\tilde u_i}, {\tilde w}_i)
      \right)
      \right)
      \d{x} \, \d{t}
    \\
    & \leq
    & \sum_{i=1}^k
      [\mbox{\eqref{eq:23}}]
      + [\mbox{\eqref{eqn:DxPhhiDiff1fin}}]
      + [\mbox{\eqref{eqn:DxPhhiDiff2fin}}]
      + [\mbox{\eqref{eq:24}}]
      +[\mbox{\eqref{eq:25}}]
    \\
    & \leq
    & \frac{L}{\ln(2)} \,
      \mathcal{C} \,
      \int_{0}^{(T-t_o)/k} \frac{\omega(\xi)}{\xi} \d{\xi} \;
      \sum_{i=1}^k
      (t_i - t_{i-1})^2
    \\
    & =
    & \frac{L}{\ln(2)} \,
      \mathcal{C}
      \int_{0}^{(T-t_o)/k} \frac{\omega(\xi)}{\xi} \d{\xi} \; \dfrac{(T-t_o)^2}k
    \\
    & \underset{k\to+\infty}{\longrightarrow} 0\,,
  \end{eqnarray*}
  where $\mathcal{C}$ depends on the test function $\phi$ and the
  constants from~\ref{ip:(V)}-\ref{ip:(P)}-\ref{ip:(Q)}.

  Pass now to estimate~\eqref{eq:20}.  Temporarily, for
  $i=0, \ldots, k$, define
  $(u_i(t), w_i(t)) = F(t - t_{i-1}, t_{i-1}) ({\tilde u_i}, {\tilde
    w}_i)$.  Then $u_i(t) = P^{{\tilde w_i}}(t, t_{i-1}){\tilde u_i}$,
  and thus it satisfies
  \begin{equation}
    \label{prf:eq:uisoln}
    \int_{t_{i-1}}^{t_i} \int_{\reali^n}
    \mathcal{I}_\psi(u_i(t), {\tilde w_i})\d{x}\d{t} = 0
    \qquad\qquad \forall\,\psi \in \Cc\infty(\mathopen]t_{i-1},
    t_i\mathclose[ \times \reali^n; \reali)\,.
  \end{equation}
  Then, each summand in~\eqref{eq:20} can be estimated as follows:
  \begin{eqnarray}
    \nonumber
      &
      & \int_{t_{i-1}}^{t_i} \int_{\reali^n}
        \mathcal{I}_\phi\left(F(t - t_{i-1}, t_{i-1}) ({\tilde u_i}, {\tilde w_i})\right)\d{x}\d{t}
    \\
    \nonumber
      & =
      & \int_{t_{i-1}}^{t_i} \int_{\reali^n}
        \mathcal{I}_\phi\left(u_i(t), {\tilde w_i}\right)\d{x}\d{t}
    \\
    \nonumber
      & +
      & \int_{t_{i-1}}^{t_i}
        \int_{\reali^n} \big[
        \left(m(t, x, {\tilde w_i}) - m(t, x, w_i(t)) \right)u_i(t)
        + \left(q(t, x, {\tilde w_i}) - q(t, x, w_i(t)) \right)
        \big]\phi(t ,x)\d{x}\d{t}
    \\
    \nonumber
      & +
      & \int_{t_{i-1}}^{t_i} \int_{\reali^n}^{}
        u_i(t)
        \left(
        v(t, x, w_i(t)) - v(t, x, \tilde{w}_i)
        \right)
        \cdot \nabla_x \phi
        \d{x} \d{t}
    \\
    \nonumber
      & \leq
      & \int_{t_{i-1}}^{t_i} \int_{\reali^n}
        \mathcal{I}_\phi\left(u_i(t), {\tilde w_i}\right)\d{x}\d{t}
        +\norma{\phi}_{\L\infty([t_o, T] \times \reali^n; \reali^n)}
        \int_{t_{i-1}}^{t_i}
        (M_L\, R + Q_L) \,
        d_\mathcal{W} \left( {\tilde w_i}, w_i(t) \right) \d{t}
    \\
    \nonumber
      & +
      & \norma{\nabla_x \phi}_{\L\infty([t_o, T] \times \reali^n; \reali^n)}
        \int_{t_{i-1}}^{t_i}
        V_L R \,
        d_\mathcal{W} \left( {\tilde w_i}, w_i(t) \right) \d{t}
    \\
    \nonumber
      & \leq
      & \int_{t_{i-1}}^{t_i} \int_{\reali^n}
        \mathcal{I}_\phi\left(u_i(t), {\tilde w_i}\right)\d{x}\d{t}
        +\norma{\phi}_{\L\infty([t_o, T] \times \reali^n; \reali^n)}
        \hlf (M_L \, R + Q_L)\, \mathcal{C} \, (t_i - t_{i-1})^2
    \\
    \label{eq:26}
      & +
      & \norma{\nabla_x \phi}_{\L\infty([t_o, T] \times \reali^n; \reali^n)}
        \, \hlf \, V_L \, R\, \mathcal{C} \, (t_i - t_{i-1})^2 \,,
  \end{eqnarray}
  where $\mathcal{C}$ is the Lipschitz constant of $t \mapsto w (t)$
  and we used the equality $w (t_{i-1}) = \tilde w_i$.  The latter two
  summands in~\eqref{eq:26} are treated as the terms above.

  {Concerning the first summand, consider
    $\chi_\eps \in \Cc\infty(\mathopen]t_{i-1}, t_i\mathclose[;
    [0,1])$ satisfying $\chi_\eps(t) = 1$, for
    $t \in \mathopen]t_{i-1} + \eps, t_i - \eps\mathclose[$, and
    define $\phi_\eps = \phi \cdot \chi_\eps$. Then,
    \begin{displaymath}
      \int_{t_{i-1}}^{t_i} \! \int_{\reali^n}
      \mathcal{I}_\phi \! \left(u_i(t), {\tilde w_i}\right)
      \d{x}\d{t}
      =
      \int_{t_{i-1}}^{t_i} \! \int_{\reali^n} \!
      \mathcal{I}_{\phi - \phi_\eps} \! \left(u_i(t), {\tilde w_i}\right)
      \d{x}\d{t}
      + \int_{t_{i-1}}^{t_i} \! \int_{\reali^n}\!
      \mathcal{I}_{\phi_\eps} \! \left(u_i(t), {\tilde w_i}\right)
      \d{x}\d{t} \,.
    \end{displaymath}
    The second term here vanishes, by~\eqref{prf:eq:uisoln}. We then
    have
    \begin{eqnarray*}
      &
      & \int_{t_{i-1}}^{t_i} \int_{\reali^n}
        \mathcal{I}_{\phi - \phi_\eps} \left(u_i(t), {\tilde w_i}\right)\d{x}\d{t}
      \\
      & =
      & \int_{t_{i-1}}^{t_i} \int_{\reali^n}
        \bigg[
        u_i \, \partial_t (\phi - \phi_\eps)
        + u_i \, v(t, x, \tilde{w}_i) \cdot \nabla_x (\phi(t, x) - \phi_\eps(t, x))
      \\
      &
      & \qquad\qquad\qquad
        + \left( m(t, x, {\tilde w_i}) \, u_i + q(t, x, {\tilde w_i})
        \right)(\phi(t, x) - \phi_\eps(t, x))
        \bigg]\d{x}\d{t} \,.
    \end{eqnarray*}
    Via a use of the Dominated Convergence Theorem, the last two terms
    here tend to zero as $\eps \to 0$, since $\chi_\varepsilon \to 1$
    a.e.~on $[t_{i-1},t_i]$.  For the first term, by the construction
    of $\chi_\varepsilon$ and the $\L1$ continuity in time of $u_i$,
    \begin{displaymath}
      \int_{t_{i-1}}^{t_i} \int_{\reali^n} u_i \; \partial_t (\phi - \phi_\eps)\d{x}\d{t}
      \underset{\varepsilon\to 0}{\longrightarrow}
      \int_{\reali^n} \left(
        u_i(t_i, x) \; \phi(t_i, x) - u_i(t_{i-1}, x) \;\phi(t_{i-1}, x)
      \right) \d{x} \d{t} \,.
    \end{displaymath}}

  Passing to the sum~\eqref{eq:20}, and remembering that
  $u_i(t_{i-1}, x) = \tilde{u}_i = P_1(t_{i-1}, t_o)(u_o, w_o)$,
{\begin{eqnarray*}
      &
      & \sum_{i=1}^k
        \int_{t_{i-1}}^{t_i} \int_{\reali^n}
        \mathcal{I}_\phi\left(u_i(t), {\tilde w_i}\right)\d{x}\d{t}
      \\
      & =
      & \sum_{i=1}^{k-1} \int_\reali\Big[
        F_1(t_i - t_{i-1}, t_{i-1})P(t_{i-1}, t_o)(u_o, w_o)
        - P_1(t_i, t_o)(u_o, w_o)
        \Big]\phi(t_i, x)\d{x}\d{t}
      \\
      & \leq
      & \sum_{i=1}^{k-1} (t_i - t_{i-1}) \frac{2L}{\ln(2)}
        \int_0^{t_i - t_{i-1}} \frac{\omega(\xi)}{\xi} \d\xi \;
        \norma{\phi(t_i)}_{\L\infty(\reali^n; \reali)}
      \\
      & \leq
      & \frac{2L}{\ln(2)} \;
        \norma{\phi}_{\L\infty([t_o,T];\reali^n; \reali)} \;
        (T-t_o) \;
        \int_0^{(T-t_o)/k} \frac{\omega(\xi)}{\xi} \d{\xi}
      \\
      &     \underset{k \to +\infty}{\longrightarrow}
      & 0,
    \end{eqnarray*}}
  as required.
\end{proofof}

\subsection{Proofs for~\S~\ref{subsec:IBVP2}}
\label{subs:IBVP2}

Similar to the previous sections, for each $w \in \mathcal{W}$ the
unique solution to~\eqref{eq:IBVP2Eq} in the sense of
Definition~\ref{def:IBVP} is
\begin{equation}
  \label{eq:rIBVP2}
  u (t,x)
  =
  \left\{
    \begin{array}{@{\,}l@{\qquad}r@{\,}c@{\,}l@{}}
      \displaystyle
      u_o\left(\mathcal{X} (t_o;t,x)\right)
      \, \mathcal{E}_w (t_o,t,x)
      \\
      \displaystyle
      \qquad\qquad+
      \int_{t_o}^t
      q\left(\tau,\mathcal{X} (\tau;t,x),w\right)
      \, \mathcal{E}_w (\tau,t,x) \d\tau
      & x
      & \geq
      & \mathcal{X} (t;t_o,0)
      \\
      \displaystyle
      b\left(\mathcal{T} (0;t,x)\right)
      \, \mathcal{E}_w\left(\mathcal{T} (0;t,x),t,x\right)
      \\
      \displaystyle
      \qquad\qquad+ \int_{\mathcal{T} (0;t,x)}^t
      q\left(\tau, \mathcal{X} (\tau;t,x), w \right)
      \, \mathcal{E}_w (\tau,t,x) \d\tau
      & x
      & <
      & \mathcal{X} (t;t_o,0)
    \end{array}
  \right.
\end{equation}
where now
\begin{equation}
  \label{eq:80}
  \mathcal{E}_w (\tau,t,x)
  =
  \exp
  \int_\tau^t \left(
    m\left(s, \mathcal{X} (s;t,x), w\right)
    -
    \partial_x v \left(s, \mathcal{X} (s;t,x)\right)
  \right)
  \d{s} \,.
\end{equation}
Working under the assumptions of Proposition~\ref{prop:IBVP2}, we
define the parametrised mapping $P^w$, which we propose is a process,
by
\begin{equation}
  \label{def:PwIBVP2}
  \begin{array}{@{}c@{\,}c@{\,}c@{\,}c@{\,}c@{}}
    P^w
    & \colon
    & \mathcal{A}
    & \to
    & \mathcal{U}
    \\
    &
    & \left(t,t_o,u_o\right)
    & \mapsto
    & u(t)
  \end{array}
  \quad\mbox{ where }\quad
  \begin{array}{@{}l@{}}
    u (t)
    \mbox{ is given by~\eqref{eq:rIBVP2};}
  \end{array}
\end{equation}
where $\mathcal{A}$ is generated by the sets $\mathcal{D}_t$ as given
by~\eqref{eq:IBVP2Const}.

\begin{lemma}
  \label{lem:IBVP2Process}
  The mapping $P^w$ as defined in~\eqref{def:PwIBVP2} is a process in
  the sense of Definition~\ref{def:Global}.
\end{lemma}

{\begin{proofof}{Lemma~\ref{lem:IBVP2Process}} Fix
    $w \in \mathcal{W}$.  Conditions~\eqref{eq:process0} and
    \eqref{eq:process2} are an immediate consequence
    of~\eqref{def:PwIBVP2}. It remains to show~\eqref{eq:process1}. As
    the choice of $w \in \mathcal{W}$ has no impact on this result, we
    omit references to $w$.

    Define $\sigma(t) = X(t; t_o, 0)$, and for a fixed
    $t \in I$, $J_1 = \mathopen[0, \sigma(t)\mathclose[$, and
    $J_2 = [\sigma(t), +\infty\mathclose[$.

    \paragraph{1.} We first show that, if
    $\norma{u_o}_{\L1(\reali_+;\reali)} \leq \alpha_1(t_o)$, then
    $\norma{u(t)}_{\L1(\reali_+;\reali)} \leq \alpha_1(t)$.

    To begin, we have
    \begin{equation}
      \begin{aligned}[b]\label{prf:IBVP2Proc-1}
        \norma{u(t)}_{\L1(\reali_+;\reali)} & \leq \int_0^{\sigma(t)}
        | b(\mathcal{T}(0; t, x)) \, \mathcal{E}(\mathcal{T}(0; t, x),
        t, x) | \d{x}
        \\
        & \quad\ + \int_0^{\sigma(t)} \int_{\mathcal{T}(0;t, x)}^t |
        q(\tau, \mathcal{X}(\tau; t, x)) \, \mathcal{E}(\tau, t, x) |
        \d{\tau} \d{x}
        \\
        & \quad\ + \int_{\sigma(t)}^{+\infty} | u_o(\mathcal{X}(t_o;
        t, x)) \mathcal{E}(t_o, t, x) | \d{x}
        \\
        & \quad\ + \int_{\sigma(t)}^{+\infty} \int_{t_o}^t | q(\tau,
        \mathcal{X}(\tau; t, x)) \mathcal{E}(\tau; t, x) | \d{\tau}
        \d{x}
        \\
        & = \int_{t_o}^t \modulo{v(\eta, 0)} \, \modulo{b(\eta)} \,
        \exp\int_{\eta}^{t} \, m(s, \mathcal{X}(s;0,\eta))\ \d{s}
        \d{\eta}
        \\
        & \quad\ + \int_{t_o}^t \int_0^{\sigma(\tau)} |q(\tau, \xi)|
        \exp \int_\tau^t m(s, \mathcal{X}(s; t, 0)) \d{s} \d{\xi}
        \d{\tau}
        \\
        & \quad\ + \int_{0}^{+\infty} \modulo{u_o(\xi)} \exp
        \int_{t_o}^t m(s, \mathcal{X}(s; t_o, \xi)) \d{s} \d{\xi}
        \\
        & \quad\ + \int_{t_o}^t \int_{\sigma(t)}^{+\infty} |q(\tau,
        \xi)| \exp \int_{\tau}^t m(s, \mathcal{X}(s; \tau, \xi)) \d{s}
        \d{\xi} \d{\tau}
        \\
        & \leq \left( \norma{u_o}_{\L1(\reali_+;\reali)} + (\hat{v}
          B_\infty + Q_1)(t - t_o) \right) e^{M_\infty (t - t_o)}.
      \end{aligned}
    \end{equation}
		Inserting the fact that
    $\norma{u_o}_{\L1(\reali_+;\reali)} \leq \alpha_1(t_o)$
    into~\eqref{prf:IBVP2Proc-1}, we have
    \begin{align*}
      \norma{u(t)}_{\L1(\reali_+;\reali)}
      & \leq \left(
        \norma{u_o}_{\L1(\reali_+;\reali)}
        + (\hat{v} B_\infty + Q_1)(t - t_o)
        \right)
        e^{M_\infty (t - t_o)}\\
      & \leq \left(
        R e^{-M_\infty (T - t_o)}
        - (\hat{v} B_\infty + Q_1) (T - t_o) e^{M_\infty t_o}
        + (\hat{v} B_\infty + Q_1)(t - t_o)
        \right)
        e^{M_\infty (t - t_o)}\\
      & \leq Re^{-M_\infty (T - t)}
        - (\hat{v} B_\infty + Q_1) (T - t) e^{M_\infty t}\\
      & = \alpha_1(t)
    \end{align*}

    \paragraph{2.} We show that if
    $\norma{u_o}_{\L\infty(\reali_+; \reali)} \leq \alpha_\infty(t_o)$
    and $B_\infty \leq \alpha_\infty(t_o)$, then
    $ \norma{u(t)}_{\L\infty(\reali_+;\reali)} \leq \alpha_\infty(t)
    $.

    We have, directly from \eqref{eq:rIBVP2},
    \begin{align*}
      \norma{u(t)}_{\L\infty(\reali_+; \reali)}
      & \leq \left(
        \max \left\{
        \norma{u_o}_{\L\infty(\reali_+; \reali)},
        B_\infty
        \right\}
        + Q_\infty (t - t_o)
        \right) e^{M_\infty (t - t_o)}\\
      & \leq \left(
        \alpha_\infty(t_o)
        + Q_\infty (t - t_o)
        \right) e^{M_\infty (t - t_o)}\\
      & \leq \left(
        R e^{-M_\infty (T - t_o)}
        - Q_\infty (T - t_o)
        + Q_\infty(t - t_o)
        \right) e^{M_\infty (t - t_o)}\\
      & \leq R e^{-M_\infty (T - t)}
        - Q_\infty (T - t)\\
      & = \alpha_\infty(t).
    \end{align*}

    \paragraph{3.} Finally, we demonstrate that if
    $\tv(u_o) + \modulo{u_o(0) - b(t_o)} \leq \alpha_{TV}(t_o)$, then
    $\tv(u) + \modulo{u(t, 0) - b(t)} \leq \alpha_{TV}(t)$.

    The left continuity of $b$ implies the right continuity of
    $u(t, \cdot)$ at $0$, and hence
    \begin{eqnarray}
      \nonumber
      \tv\left(u(t)\right)
      & =
      & \tv\left(u(t); \mathopen]0, +\infty\mathclose[\right)
      \\
      & \leq
      & \tv\left(u(t); \mathopen]0, \sigma(t)\mathclose[\right)
        \label{prf:IBVP2tv1}
      \\
      &
      & \quad\ + \modulo{u(t, \sigma(t)-) - u(t, \sigma(t)+)}
        \label{prf:IBVP2tv2}
      \\
      &
      & \quad\ + \tv\left(u(t); \mathopen]\sigma(t), +\infty\mathclose[\right) \,.
        \label{prf:IBVP2tv3}
    \end{eqnarray}
    We calculate the three terms~\eqref{prf:IBVP2tv1},
    \eqref{prf:IBVP2tv2} and~\eqref{prf:IBVP2tv3} separately.

    Beginning with~\eqref{prf:IBVP2tv1}, we have
    \begin{eqnarray*}
      \tv\left(u(t); \mathopen]0, \sigma(t)\mathclose[\right)
      & \leq
      & \tv\left( b(\mathcal{T}(0; t, x))
        \mathcal{E}(\mathcal{T}(0; t, x), t, x)
        ;\mathopen]0, \sigma(t)\mathclose[
        \right)
      \\
      &
      & + \tv\left(
        \int_{\mathcal{T}(0; t, x)}^t
        q(\tau, \mathcal{X}(\tau; t, x))
        \mathcal{E}(\tau, t, x)
        \d{\tau}
        ;
        ]0, \sigma(t)[
        \right)
      \\
      & \leq
      & \left(
        \tv(b; ]t_o, t[)
        + \norma{b}_{\L\infty([t_o, t]; \reali)}
        (M_\infty + V_L)(t - t_o)
        \right)e^{(M_\infty + V_L) (t - t_o)}
      \\
      &
      & + Q_{\infty} (t - t_o) (1 + (M_\infty + V_L)(t - t_o))
        e^{(M_\infty + V_L) (t - t_o)}
    \end{eqnarray*}

    For the second term \eqref{prf:IBVP2tv2},
    \begin{eqnarray}
      \nonumber
      &
      & \modulo{u\left(t, \sigma(t)+\right) - u\left(t, \sigma(t)-\right)}
      \\
      \nonumber
      & \leq
      &  \modulo{
        u_o\left(\mathcal{X} \left(t_o;t,\sigma (t)+\right)\right)
        \, \mathcal{E} \left(t_o,t,\sigma (t)+\right)
        -
        b\left(\mathcal{T} \left(0;t,\sigma (t)-\right)\right)
        \, \mathcal{E}\left(\mathcal{T} \left(0;t, \sigma (t)-\right),t,\sigma (t)-\right)}
      \\
      \nonumber
      &
      &
        +
        \left|
        \int_{t_o}^t
        q\left(\tau,\mathcal{X} \left(\tau;t,\sigma (t)+\right),w\right)
        \, \mathcal{E} \left(\tau,t,\sigma (t)+\right) \d\tau
        \right.
      \\
      \nonumber
      &
      & \qquad\qquad
        \left.
        - \int_{\mathcal{T} (0;t,\sigma (t)-)}^t
        q\left(\tau, \mathcal{X} \left(\tau;t,\sigma (t)-\right), w \right)
        \, \mathcal{E} \left(\tau,t,\sigma (t)-\right) \d\tau
        \right|
      \\
      \nonumber
      & =
      &  \modulo{ u(t_o,0+) -  b(t_o+)} \, \mathcal{E}\left(t_o,t,\sigma (t)-\right)
      \\
      \nonumber\texttt{}
      & \leq
      & \left(
        \modulo{ u(t_o,0) -  b(t_o)} + \modulo{b(t_o-) - b(t_o+)}
        \right)\, e^{(M_\infty + V_L) (t-t_o)} \,.
    \end{eqnarray}

    Note that
    $\tv\left(b; \mathopen]t_o, t\mathclose[\right) + \modulo{b(t_o-)
      - b(t_o+)} = \tv(b; \mathopen[t_o, t\mathclose[)$ from the left
    continuity of $b$.

    For the final term~\eqref{prf:IBVP2tv3}, we find
    \begin{align*}
      \tv\left(u(t); \mathopen]\sigma(t), +\infty\mathclose[\right)
      \leq
      \big(
      &
        \tv(u_o; \mathopen]0, +\infty\mathclose[)
        + \norma{u_o}_{\L\infty(\reali_+;\reali)}(M_\infty + V_L)(t - t_o)
      \\
      & \quad\ + Q_\infty (1 + (M_\infty + V_L)(t - t_o)) (t - t_o)
        \big)
        e^{(M_\infty + V_L)(t - t_o)}
    \end{align*}
    Finally, notice that, as $u(t, 0) = b(t)$, we have
		\begin{equation*}
			\tv\left(u(t, \cdot);\mathopen]0,+\infty\mathclose[\right) +
			\modulo{u(t, 0) - b(t)} = \tv\left(u(t, \cdot)\right),
		\end{equation*}
		and thus
    we need only to show
    $\tv\left(u(t, \cdot)\right) \leq \alpha_{TV}(t)$. Using these
    three estimates, we obtain
    \begin{align*}
      \tv\left(u(t)\right)
      \leq \big(
      &
        \tv\left(u_o; \mathopen]0, +\infty\mathclose[\right)
        +
        \modulo{u(t, 0) - b(t_o)}
      \\
      & + \norma{u_o}_{\L\infty(\reali_+;\reali)} (M_\infty + V_L) (t - t_o)
      \\
      & + \tv(b; \mathopen[t_o, t\mathclose[) + B_\infty (M_\infty + V_L) (t - t_o)
      \\
      & + 2Q_\infty (1 + (M_\infty + V_L)(t - t_o))(t - t_o)
        \big) e^{(M_\infty + V_L)(t - t_o)}
      \\
      \leq \big(
      & \alpha_{TV}(t_o)
        + \norma{u_o}_{\L\infty(\reali_+;\reali)} (M_\infty + V_L) (t - t_o)
      \\
      & +  \tv(b; [t_o, t[) + B_\infty (M_\infty + V_L) (t - t_o)
      \\
      & + 2Q_\infty (1 + (M_\infty + V_L)(t - t_o))(t - t_o)
        \big) e^{(M_\infty + V_L)(t - t_o)}
      \\
      \leq \big(
      & R(1 - 2(M_\infty + V_L)(T - t_o))e^{(M_\infty + V_L)(T - t_o)}
      \\
      & -2 Q_\infty(1 + (M_\infty + V_L)t_o)(T - t_o)e^{(M_\infty + V_L)t_o}
      \\
      & - B_\infty(M_\infty + V_L)(T - t_o)e^{(M_\infty + V_L)t_o}
      \\
      & - \tv(b; [t_o, T])e^{(M_\infty + V_L)t_o}
      \\
      & + \big(
        Re^{-(M_\infty + V_L)(T - t_o)}
        - Q_\infty(T - t)e^{(M_\infty + V_L)t}
        \big)(M_\infty + V_L)(t - t_o)
      \\
      & +  \tv(b; \mathopen[t_o, t\mathclose[) +
        B_\infty (M_\infty + V_L) (t - t_o)
      \\
      & + 2Q_\infty (1 + (M_\infty + V_L)(t - t_o))(t - t_o)
        \big) e^{(M_\infty + V_L)(t - t_o)}
      \\
      \leq \hphantom{\big(}
      &
        R(1 - (M_\infty + V_L)(T - t))e^{(M_\infty + V_L)(T - t)}
      \\
      & -2 Q_\infty(1 + (M_\infty + V_L)t)(T - t)e^{(M_\infty + V_L)t}
      \\
      & - B_\infty(M_\infty + V_L)(T - t)e^{(M_\infty + V_L)t}
      \\
      & - \tv(b; [t, T])e^{(M_\infty + V_L)t}
      \\
      = \hphantom{\big(}
      & \alpha_{TV}(t) \,,
    \end{align*}
    as required.
  \end{proofof}
}

\begin{proofof}{Proposition~\ref{prop:IBVP2}}
  The mapping $P^w$, as given by~\eqref{def:PwIBVP2}, is a process for
  any $w \in \mathcal{W}$ by Lemma~\ref{lem:IBVP2Process}.  It remains
  to show that $P^w$ is a Lipschitz process on $\mathcal{U}$
  parametrised by $w \in \mathcal{W}$, i.e., it
  satisfies~\eqref{eq:6}, \eqref{eq:7}, and \eqref{eq:8}, with
  $C_u, C_t$ and $C_w$ given by~\eqref{eq:IBVP2Const}.

\paragraph{1. Lipschitz Continuity w.r.t.~Initial Data.}
Consider two initial data $u_1, u_2 \in \mathcal{D}$, $t_o, t \in I$
with $t_o < t$, and $w \in \mathcal{W}$.

To begin, assume that $x \in \mathopen[0, \sigma(t)\mathclose[$. Then,
it is easy to see from~\eqref{eq:rIBVP2} that
\begin{displaymath}
  \modulo{P^w(t, t_o)u_1 - P^w(t, t_o)u_2}(x) = 0\,,
\end{displaymath}
as $b, q$ and $m$ are independent of the choice of initial data $u_o$.
Similarly, for $x \in \mathopen[\sigma(t), +\infty\mathclose[$,
\begin{displaymath}
  \modulo{P^w(t, t_o)u_1 - P^w(t, t_o)u_2}(x)
  = \modulo{u_1(\mathcal{X}(t_o; t, x))
    - u_2(\mathcal{X}(t_o; t, x))} \, \mathcal{E}_w(t_o, t, x) \,.
\end{displaymath}
Thus, using the substitution $y = \mathcal{X}(t_o; t, x)$,
\begin{eqnarray*}
  d_{\mathcal{U}} \left(
  P^w(t, t_o)u_1,
  P^w(t, t_o)u_2
  \right)
  & =
  & \int_{\sigma(t)}^{+\infty}
    \modulo{u_1(\mathcal{X}(t_o; t, x)) - u_2(\mathcal{X}(t_o; t, x))}
    \mathcal{E}_w(t_o, t, x) \d{x}
  \\
& =
  & \int_{0}^{+\infty}
    \modulo{u_1(y) - u_2(y)}
    e^{\int_{t_o}^t m(s, \mathcal{X}(s; t_o, y), w) \d{s}}
    \d{y}
  \\
  & \leq
  & e^{M_\infty ( t - t_o )}\norma{u_1(0) - u_2(0)}_{\L1(\reali_+;\reali)} \,.
\end{eqnarray*}

\paragraph{2. Lipschitz Continuity w.r.t. Time.}
Consider $u_o \in \mathcal{D}$, $t_o, t \in I$, and
$w \in \mathcal{W}$.

We have
\begin{equation}
  \label{eq:dUt1}
  \begin{array}{rcl}
    d_\mathcal{U}(P^w(t, t_o)u_o, u_o)
    & \leq
    & \norma{P^w(t, t_o)u_o - u_o}_{\L1([0, \sigma(t)[; \reali_+)}
    \\
    &
    & + \norma{P^w(t, t_o)u_o - u_o}_{\L1([\sigma(t), +\infty[;
      \reali_+)} \,.
  \end{array}
\end{equation}
Focusing on the first term of \eqref{eq:dUt1},
using~\eqref{eq:rIBVP2}, \ref{item:IBVP1}, \ref{item:IBVP2},
\ref{item:IBVP3}, \ref{item:IBVP4}, and that $ u_o \in \mathcal{D} $,
\begin{eqnarray*}
  &
  & \norma{P^w(t, t_o)u_o - u_o}_{\L1([0, \sigma(t)[; \reali_+)}
  \\
  & \leq
  & \int_0^{\sigma(t)}
    \modulo{
    b(\mathcal{T}(0; t, x))
    \mathcal{E}_w(\mathcal{T}(0; t, x), t, x)
    - u_o(x)
    }
    \d{x}
  \\
  &
  & \quad\ + \int_0^{\sigma(t)}
    \int_{\mathcal{T}(0; t, x)}^t
    |
    q(\tau, \mathcal{X}(\tau; t, x), w)
    \mathcal{E}_w(\tau, t, x)
    | \d{\tau}
    \d{x}
  \\
  & =
  & \int_{t_o}^{t}
    v(y, 0)
    |
    b(y)
    e^{\int_{y}^t m(s, \mathcal{X}(s; y, 0), w) \d{s}}
    - u_o(\mathcal{X}(t; 0, y))
    e^{\int_y^t \partial_x v(s, \mathcal{X}(s; y, 0))
    \d{s}}
    |
    \d{y}
  \\
  &
  & \quad\ + Q_1 e^{M_\infty(t -t_o)}(t - t_o)
  \\
  & \leq
  & \hat{v}
    (
    B_1
    + \norma{u_o}_{\L\infty(\reali_+;\reali)}
    + Q_1
    ) e^{M_\infty (t - t_o)}(t - t_o)
  \\
  &
  & \quad\
    +\int_{t_o}^{t}
    v(y, 0)
    |u_o(\mathcal{X}(t; 0, y))|
    |
    e^{\int_{y}^t m(s, \mathcal{X}(s; y, 0), w) \d{s}}
    -  e^{\int_y^t \partial_x v(s, \mathcal{X}(s; y, 0)) \d{s}}
    |
    \d{y}
  \\
  & \leq
  & \hat{v}
    (
    B_1
    + R
    + Q_1
    )
    e^{M_\infty T}(t - t_o)
  \\
  &
  & \quad\
    +\hat{v} \norma{u_o}_{\L\infty(\reali_+; \reali)}
    \int_{t_o}^t
    (M_\infty + V_L) (t - y)
    e^{(M_\infty + V_L)(t - y)}
    \d{x}
  \\
  & \leq
  & \hat{v}
    (
    B_1
    + R
    + Q_1
    )
    e^{M_\infty T}(t - t_o)
    +\hat{v} R
    (M_\infty + V_L) (t - t_o)^2 e^{(M_\infty + V_L)(t - t_o)} \,.
\end{eqnarray*}
For the second term of~\eqref{eq:dUt1}, once again
from~\eqref{eq:rIBVP2},
\begin{eqnarray*}
  &
  & \norma{P^w(t, t_o)u_o - u_o}_{\L1([\sigma(t), +\infty[; \reali_+)}
  \\
  & \leq
  & \int_{\sigma(t)}^{+\infty} \!\!
    \modulo{u_o(\mathcal{X}(t_o; t, x))\mathcal{E}_w(t_o, t, x)
    - u_o(x)}
    \d{x}
        {+} \int_{\sigma(t)}^{+\infty} \!\!\!
        \int_{t_o}^t \!
        |q(\tau, \mathcal{X}(\tau; t, x), w)|
        \mathcal{E}_w(\tau, t, x)
        \d{\tau}
        \d{x}
  \\
  & \leq
  & \int_{\sigma(t)}^{+\infty}
    \modulo{u_o(\mathcal{X}(t_o; t, x)) - u_o(x)}
    \mathcal{E}_w(t_o, t, x)
    \d{x}
        + \int_{\sigma(t)}^{+\infty}
        |u_o(x)|
        \left|\mathcal{E}_w(t_o, t, x) - 1 \right|
        \d{x}
  \\
  &
  & \quad\ + \int_{t_o}^{t}
    \int_{\sigma(t)}^{+\infty}
    |q(\tau, \xi, w)|
    e^{\int_{\tau}^t m(s, \mathcal{X}(s; \tau, \xi), w) \d{s}}
    \d{\xi}
    \d{\tau}
  \\
  & \leq
  & \left[
    \hat{v} \tv(u_o; \reali_+)
    + M_\infty \norma{u_o}_{\L1(\reali_+;\reali)}
    + Q_1
    \right] e^{M_\infty (t - t_o)}(t - t_o)
  \\
  & \leq
  & \left[
    \hat{v} R
    + M_\infty R
    + Q_1
    \right] e^{M_\infty (t - t_o)}(t - t_o) \,,
\end{eqnarray*}
where we have made use of~\eqref{eq:3TV}.

Concluding, we thus have
\begin{displaymath}
  d_\mathcal{U}(P^w(t, t_o)u_o, u_o)
  \leq
  \left[
    \hat{v}(B_1 + 2R + R(M_\infty + V_L)T)
    + M_\infty R
    + Q_1
  \right]
  e^{M_\infty T} (t - t_o) \,.
\end{displaymath}

\paragraph{3. Lipschitz Continuity w.r.t. Parameters.}
Consider $u_o \in \mathcal{D}$, $t_o, t \in I$ and
$w_1, w_2 \in \mathcal{W}$.

We have
\begin{equation}\label{eq:dUw1}
  \begin{split}
    d_{\mathcal{U}} (P^{w_1}(t, t_o)u_o, P^{w_2}(t, t_o)u_o)
    & \leq \norma{P^{w_1}(t, t_o)u_o - P^{w_2}(t, t_o)u_o}_{\L1([0, \sigma(t)[; \reali_+)}\\
    & \quad\ + \norma{P^{w_1}(t, t_o)u_o - P^{w_2}(t,
      t_o)u_o}_{\L1([\sigma(t), +\infty[; \reali_+)}
  \end{split}
\end{equation}

For the first term of \eqref{eq:dUw1},
\begin{eqnarray}
  &
  & \norma{P^{w_1}(t, t_o)u_o - P^{w_2}(t, t_o)u_o}_{\L1([0, \sigma(t)[; \reali_+)} \nonumber
  \\
  & \leq
  & \int_0^{\sigma(t)}
    \modulo{b(\mathcal{T}(0; t, x))} \,
    \modulo{
    \mathcal{E}_{w_1}(\mathcal{T}(0; t, x), t, x)
    - \mathcal{E}_{w_2}(\mathcal{T}(0; t, x), t, x)
    }
    \d{x} \label{prf:Pw2}
  \\
  &
  & \quad\ + \int_0^{\sigma(t)} \int_{\mathcal{T}(0; t, x)}^t
    |
    q(\tau, \mathcal{X}(\tau; t, x), w_1)
    - q(\tau, \mathcal{X}(\tau; t, x), w_2)
    | \,
    \mathcal{E}_{w_1}(\tau, t, x)
    \d{x} \label{prf:Pw3}
  \\
  &
  & \quad\ + \int_0^{\sigma(t)} \int_{\mathcal{T}(0; t, x)}^t
    \modulo{q(\tau, \mathcal{X}(\tau; t, x), w_2)} \,
    \modulo{
    \mathcal{E}_{w_2}(\tau, t, x)
    - \mathcal{E}_{w_1}(\tau, t, x)
    }
    \d{x} \label{prf:Pw4} \,.
\end{eqnarray}
Focussing first on \eqref{prf:Pw2}, we use~\ref{item:IBVP2}, and get
\begin{eqnarray*}
  &
  & \int_0^{\sigma(t)}
    \modulo{b\left(\mathcal{T}(0; t, x)\right)} \,
    \modulo{
    \mathcal{E}_{w_1}(\mathcal{T}(0; t, x), t, x)
    - \mathcal{E}_{w_2}(\mathcal{T}(0; t, x), t, x)
    }
    \d{x}
  \\
  & =
  & \int_{t_o}^t v(y, 0) \, \modulo{b(y)}
    \modulo{
    \mathcal{E}_{w_1}(y, t, \mathcal{X}(t; 0, y))
    - \mathcal{E}_{w_2}(y, t, \mathcal{X}(t; 0, y))
    }
    \d{y}
  \\
  & \leq
  & B_\infty e^{M_\infty (t - t_o) }
    \int_{t_o}^{t}
    \int_{y}^t v(y, 0)
    \modulo{m(s, \mathcal{X}(s; y, 0), w_1) - m(s, \mathcal{X}(s; y, 0), w_2)}
    \d{s}
    \d{y}
  \\
  & =
  & B_\infty e^{M_\infty (t - t_o) }
    \int_{t_o}^{t}
    \int_{0}^{\sigma(s)}
    \modulo{m(s, \xi, w_1) - m(s, \xi, w_2)}
    \d{\xi}
    \d{s}
  \\
  & \leq
  & B_\infty M_L e^{M_\infty (t - t_o) }
    (t - t_o) d_\mathcal{W}(w_1, w_2) \,.
\end{eqnarray*}

For~\eqref{prf:Pw3}, using~\ref{item:IBVP3},
\begin{eqnarray*}
  &
  & \int_0^{\sigma(t)}
    \int_{\mathcal{T}(0; t, x)}^t
    \modulo{
    q(\tau, \mathcal{X}(\tau; t, x), w_1) - q(\tau, \mathcal{X}(\tau; t, x), w_2)
    } \,
    \mathcal{E}_{w_1}(\tau, t, x)
    \d{\tau}
    \d{x}
  \\
  & =
  & \int_{t_o}^{t}
    \int_{0}^{\sigma(\tau)}
    \modulo{v(y, \tau)} \,
    \modulo{
    q(\tau, y, w_1) - q(\tau, y, w_2)
    } \,
    e^{\int_{\tau}^t m(s, \mathcal{X}(s; \tau, y), w_1) \d{s}}
    \d{y}
    \d{\tau}
  \\
  & \leq
  & Q_L \, \hat{v} \, e^{M_\infty (t - t_o)} \, d_{\mathcal{W}}(w_1, w_2) \,.
\end{eqnarray*}

Finally, for \eqref{prf:Pw4}, we have
\begin{eqnarray*}
  &
  & \int_0^{\sigma(t)}
    \int_{\mathcal{T}(0; t, x)}^t
    \modulo{q(\tau, \mathcal{X}(\tau; t, x), w_2)} \,
    \modulo{
    \mathcal{E}_{w_2}(\tau, t, x)
    - \mathcal{E}_{w_1}(\tau, t, x)}
    \d{\tau}
    \d{x}
  \\
  & =
  & \int_{t_o}^{t}
    \int_{0}^{\sigma(\tau)}
    \modulo{q(\tau, \xi, w_2)} \,
    \modulo{
    e^{\int_{\tau}^t m(s, \mathcal{X}(s; \tau, \xi), w_2) \d{s}}
    -e^{\int_{\tau}^t m(s, \mathcal{X}(s; \tau, \xi), w_1) \d{s}}}
    \d{\xi}
    \d{\tau}
  \\
  & \leq
  & Q_\infty e^{M_\infty(t - t_o)}
    \int_{t_o}^t
    \int_{0}^{\sigma(\tau)}
    \int_{\tau}^t
    \modulo{m(s, \mathcal{X}(s; \tau, \xi), w_1)
    - m(s, \mathcal{X}(s; \tau, \xi), w_2)}
    \d{s}
    \d{\xi}
    \d{\tau}
  \\
  & \leq
  & Q_\infty e^{M_\infty(t - t_o)}
    \int_{t_o}^t
    \int_{\tau}^{t}
    \int_{\mathcal{X}(s; t_o, 0)}^{\mathcal{X}(s; \tau, 0)}
    \modulo{m(s, y, w_1) - m(s, y, w_2)}
    \d{s}
    \d{y}
    \d{\tau}
  \\
  & \leq
  & Q_\infty M_L e^{M_\infty(t - t_o)}
    \frac{1}{2} (t - t_o)^2 d_\mathcal{W}(w_1, w_2) \,.
\end{eqnarray*}
Thus,
\begin{equation}
  \label{prf:PwFin1}
  \begin{split}
    & \norma{P^{w_1}(t, t_o)u_o - P^{w_2}(t, t_o)u_o}_{\L1(J_1;
      \reali_+)}
    \\
    & \leq \left[ B_\infty M_L + \hat{v} Q_L + \frac{1}{2} Q_\infty
      M_L (t - t_o) \right] e^{M_\infty (t - t_o)}(t - t_o)
    d_\mathcal{W}(w_1, w_2) \,.
  \end{split}
\end{equation}

Focusing now on the second term of~\eqref{eq:dUw1}, we have
\begin{align}
  & \norma{P^{w_1}(t, t_o)u_o - P^{w_2}(t, t_o)u_o}_{\L1([\sigma(t), +\infty[;\reali)} \nonumber
  \\
  & \leq \int_{\sigma(t)}^{+\infty}
    \modulo{u_o\left(\mathcal{X}(t_o; t, x)\right)}
    |
    \mathcal{E}_{w_1} (t_o, t, x)
    - \mathcal{E}_{w_2} (t_o, t, x)
    |
    \d{x} \label{prf:Pw5}\\
  & \quad\ + \int_{\sigma(t)}^{+\infty}
    \int_{t_o}^t
    |
    q(\tau, \mathcal{X}(\tau; t, x), w_1)
    - q(\tau, \mathcal{X}(\tau; t, x), w_2)
    |
    \mathcal{E}_{w_1}(\tau, t, x)
    \d{\tau}
    \d{x} \label{prf:Pw6}\\
  & \quad\ + \int_{\sigma(t)}^{+\infty}
    \int_{t_o}^t
    |q(\tau, \mathcal{X}(\tau; t, x), w_2)|
    |
    \mathcal{E}_{w_1}(\tau, t, x)
    - \mathcal{E}_{w_2}(\tau, t, x)
    |
    \d{\tau}
    \d{x} \label{prf:Pw7}.
\end{align}

Looking at term~\eqref{prf:Pw5},
\begin{align*}
  & \int_{\sigma(t)}^{+\infty}
    \modulo{u_o(\mathcal{X}(t_o; t, x))} \;
    \modulo{
    \mathcal{E}_{w_1} (t_o, t, x)
    - \mathcal{E}_{w_2} (t_o, t, x)}
    \d{x}
  \\
  & \leq \norma{u_o}_{\L\infty(\reali_+;\reali)}
    e^{M_\infty (t - t_o)}
    \int_{0}^{+\infty}
    \int_{t_o}^t
    \modulo{m(s, \mathcal{X}(s; t_o, y) ,w_1)
    - m(s, \mathcal{X}(s; t_o, y), w_2)}
    \d{s}
    \d{x}
  \\
  & = \norma{u_o}_{\L\infty(\reali_+;\reali)}
    e^{M_\infty (t - t_o)}
    \int_{t_o}^t
    \int_{\sigma(s)}^{+\infty}
    \modulo{m(s, y, w_1) - m(s, y, w_2)}
    \d{y}
    \d{s}
  \\
  & \leq M_L R
    e^{M_\infty (t - t_o)} (t - t_o)
    d_\mathcal{W}(w_1, w_2) \,.
\end{align*}
Next, for the term~\eqref{prf:Pw6},
\begin{eqnarray*}
  &
  & \int_{\sigma(t)}^{+\infty}
    \int_{t_o}^t
    \modulo{
    q(\tau, \mathcal{X}(\tau; t, x), w_1)
    -
    q(\tau, \mathcal{X}(\tau; t, x), w_2)}
    \mathcal{E}_{w_1}(\tau, t, x)
    \d{\tau}
    \d{x}
  \\
  & \leq
  & e^{M_\infty (t - t_o)}
    \int_{t_o}^t
    \int_{\sigma(\tau)}^{+\infty}
    \modulo{q(\tau, y, w_1) - q(\tau, y, w_2)}
    \d{y}
    \d{\tau}
  \\
  & \leq
  & Q_L e^{M_\infty (t - t_o)} (t - t_o) d_{\mathcal{W}}(w_1, w_2).
\end{eqnarray*}
Finally, for term~\eqref{prf:Pw7},
\begin{eqnarray*}
  &
  & \int_{\sigma(t)}^{+\infty}
    \int_{t_o}^t
    |q(\tau, \mathcal{X}(\tau; t, x), w_2)|
    |
    \mathcal{E}_{w_1}(\tau, t, x)
    - \mathcal{E}_{w_2}(\tau, t, x)
    |
    \d{\tau}
    \d{x}
  \\
  & \leq
  & Q_\infty e^{M_\infty (t - t_o)}
    \int_{t_o}^t
    \int_{\sigma(t)}^{+\infty}
    \int_{\tau}^t
    \modulo{m(s, \mathcal{X}(s; \tau, \xi), w_1)
    - m(s, \mathcal{X}(s; \tau, \xi), w_2)}
    \d{s}
    \d{\xi}
    \d{\tau}
  \\
  & =
  & Q_\infty e^{M_\infty (t - t_o)}
    \int_{t_o}^t
    \int_{\tau}^t
    \int_{\sigma(s)}^{+\infty}
    \modulo{m(s, y, w_1) - m(s, y, w_2)}
    \d{y}
    \d{s}
    \d{\tau}
  \\
  & \leq
  & \frac{1}{2} M_L Q_\infty
    e^{M_\infty (t - t_o)} (t - t_o)^2
    d_\mathcal{W}(w_1, w_2) \,.
\end{eqnarray*}
Thus, combining these estimates together we have
\begin{equation}
  \label{prf:PwFin2}
  \begin{array}{cl}
    & \norma{P^{w_1}(t, t_o)u_o - P^{w_2}(t, t_o)u_o}_{\L1(J_1; \reali_+)}
    \\
    \leq
    & \left[ M_L R + Q_L + \frac{1}{2} M_L Q_\infty (t - t_o)
      \right] e^{M_\infty (t - t_o)} d_{\mathcal{W}}(w_1, w_2) \,.
  \end{array}
\end{equation}
Due to the assumption $u_o \in \mathcal{D}$, we have
$\norma{u_o}_{\L1(\reali_+;\reali)} \leq R$. Hence,
substituting~\eqref{prf:PwFin1} and~\eqref{prf:PwFin2}
into~\eqref{eq:dUw1}, and as $(t - t_o) < T$, we get
\begin{equation}
  d_\mathcal{U}(P^{w_1}(t, t_o)u_o, P^{w_2}(t, t_o)u_o)
  \leq C_w
  (t - t_o)
  d_\mathcal{W}(w_1, w_2)
\end{equation}
where $C_w$ is as in~\eqref{eq:IBVP2Const}, as required.
\end{proofof}

\begin{proofof}{Proposition~\ref{prop:coupleIBVP2}}
  For fixed $t_o \in I$, $u_o \in \mathcal{U}$, and
  $w \in \mathcal{W}$, define by
  $\Pi_{(t_o, u_o, w_o)}:\{(s, s_o) \in [t_o, T]^2 : s \geq s_o\}
  \times \mathcal{U} \to \mathcal{U}$ to be the process with
  $s \mapsto \Pi_{(t_o, u_o, w_o)}(s, s_o)\rho_o$ being the solution
  of
  \begin{equation}
    \label{eq:coupleIBVP2-1}
    \left\{
      \begin{array}{@{}l@{\qquad\quad}r@{\,}c@{\,}l@{}}
        \partial_t \rho + \partial_x \left(v (t,x) \,  \rho\right)
        =
        \bar m (t,x) \, \rho + \bar q (t,x)
        & (t,x)
        & \in
        & [s_o,T] \times \reali_+
        \\
        \rho (t,0) = b_o (t)
        & t
        & \in
        & [s_o,T]
        \\
        \rho (s_o,x) = \rho_o (x)
        & x
        & \in
        & \reali_+
      \end{array}
    \right.
  \end{equation}
  with $\bar{m}$ and $\bar{q}$ the given by \eqref{eq:IBVP2CompProb2}.
  For notational simplicity, we write $\Pi_{(t_o, u_o, w_o)} = \Pi$
  when the $(t_o, u_o, w_o)$ when no confusion arises.

  The mapping $\Pi$ is Lipschitz continuous with respect to time and
  initial data, for some constant $\mathcal{L}>0$, as $\bar{m}$ and
  $\bar{q}$ satisfy correspondingly \ref{item:IBVP2} and
  \ref{item:IBVP3}, which do not explicitly depend on $w$.

  By this construction, $t \mapsto \Pi_{(t_o, u_o, w_o)}(t, t_o)u_o$
  is the solution of \eqref{eq:IBVP2CompProb}.

  From~\cite[Theorem~2.9]{BressanLectureNotes}, we have
  \begin{displaymath}
    \begin{array}{cl}
      & \norma{u(t) - \Pi_{(t_o, u_o, w_o)}(t, t_o)u_o}_{\L1(\reali_+; \reali)}
      \\
      \leq
      & \mathcal{L}
        \int_{t_o}^t
        \liminf_{h \to 0+} \frac{1}{h}
        \norma{
        u(\tau + h) - \Pi_{(t_o, u_o, w_o)}(\tau + h, \tau)u(\tau)
        }_{\L1(\reali_+;\reali)}
        \d{\tau}
      \\
      =
      & \mathcal{L}
        \int_{t_o}^t
        \liminf_{h \to 0+} \frac{1}{h}
        \norma{
        P_1(\tau + h, \tau)P(\tau, t_o)(u_o, w_o)
        - \Pi_{(t_o, u_o, w_o)}(\tau + h, \tau)u(\tau)
        }_{\L1(\reali_+;\reali)}
        \d{\tau} .
    \end{array}
  \end{displaymath}
  Thus it suffices to show, for any $0 \leq t_o \leq \tau \in [0, T]$,
  that
  \begin{displaymath}
    \liminf_{h \to 0+} \frac{1}{h}
    \norma{
      P_1(\tau + h, \tau)P(\tau, t_o)(u_o, w_o)
      -
      \Pi_{(t_o, u_o, w_o)}(\tau + h, \tau)u(\tau)
    }_{\L1(\reali_+;\reali)} = 0 \,.
  \end{displaymath}
  The tangency condition~\eqref{eq:tangent} ensures that
  \begin{displaymath}
    \frac{1}{h}
    \norma{
      P_1(\tau + h, \tau) u(\tau)
      -
      P^{P_2(\tau, t_o)(u_o, w_o)}(\tau + h)u(\tau)}_{\L1(\reali_+;\reali)}
    \leq
    \O \int_0^h \frac{\omega(\xi)}{\xi}\ d\xi \to 0
  \end{displaymath}
  as $h \to 0$.

  Further, it can be shown, using formula~\eqref{eq:rIBVP2}, that
  \begin{displaymath}
    \norma{P^{P_2(\tau, t_o)(u_o, w_o)}(\tau + h, \tau) u(\tau) -
      \Pi_{(t_o, u_o, w_o)}(\tau + h, \tau) u(\tau)
    }_{\L1(\reali_+;\reali)} \leq \O h^2 \,,
  \end{displaymath}
  with the constant $\O$ depending on the constants laid out
  in~\ref{item:IBVP1}-\ref{item:IBVP4}, $R$ and $T$. Thus this also
  converges to zero as $h\to 0$, completing our proof.
\end{proofof}

\subsection{Proofs for \S~\ref{subsec:meas-valu-balance}}
\label{subsec:proof-MVBL}

{%
  \begin{lemma}
    \label{lem:BCapprox}
    Let $f \in \BC(\reali_+; \reali)$. For any
    $\eta \in \N \setminus \{0\}$ there exists a function
    $ f_\eta \in (\C1 \cap \W{1}{\infty})(\reali_+;\reali) $ such that
    \begin{itemize}
    \item
      $\norma{f_\eta'}_{\L\infty(\reali_+; \reali)} \leq
      \frac{2}{\eta} \, \norma{f}_{\L\infty(\reali_+; \reali)}$,
    \item $f_\eta \to f \text{ pointwise, as } \eta \to 0$,
    \item
      $\norma{f_\eta}_{\W{1}{\infty}(\reali_+; \reali)} \leq \left(1 +
        \frac{2}{\eta}\right) \norma{f}_{\L\infty(\reali_+; \reali)}$.
    \end{itemize}
  \end{lemma}

\begin{proofof}{Lemma~\ref{lem:BCapprox}}
  Consider $f_\eta(x) = \frac{1}{\eta} \int_0^\eta f(x + y)\ dy$.
\end{proofof}
}

\begin{lemma}
  \label{lem:muNarCont}
  The mapping $\mu$ defined by~\eqref{def:muSoln} in
  Proposition~\ref{prop:MEASBLCoup} is narrowly continuous.
\end{lemma}

{\begin{proofof}{Lemma~\ref{lem:muNarCont}}
    Choose $f \in \BC(\reali_+)$ and fix $t \in \mathbb{R}^+$. Let
    $\eps >0$ and for $\eta >0$ define
    $f_\eta \in (\C1 \cap \W{1}{\infty})(\reali_+;\reali) $ as in
    Lemma~\ref{lem:BCapprox}. Then, setting
    $M_\eta = \norma{f_\eta}_{\W1\infty (\reali;\reali)}$, so that
    $\Lip \left(\frac{f_\eta}{M_\eta}\right) \leq 1$, we have
    \begin{eqnarray*}
      &
      & \modulo{
        \int_{\reali_+} f(x) \d(P_1(t, t_o)\mu_o - P_1(s, t_o)\mu_o)(x)
        }
      \\
      & \leq
      & \modulo{
        \int_{\reali_+}
        \left(f(x) - f_\eta(x)\right)
        \d{\left(P_1(t, t_o)\mu_o - P_1(s, t_o)\mu_o \right)(x)}}
      \\
      &
      & + M_\eta \modulo{
        \int_{\reali_+}
        \frac{f_\eta (x)}{M_\eta}  \d{(P_1(t, t_o)\mu_o - P_1(s, t_o)\mu_o)(x)}}
      \\
      & \leq
      & \int_{\reali_+}
        \modulo{f(x) - f_\eta(x)}
        \d{\left(\modulo{P_1(t, t_o)\mu_o - P_1(s, t_o)\mu_o} \right)(x)}
      \\
      &
      & + M_\eta \; d_\mathcal{M}(P_1(t, t_o)\mu_o, P_1(s, t_o)\mu_o)
      \\
      & \leq
      & \int_{\reali_+}
        \modulo{f(x) - f_\eta(x)}
        \d{\left(\modulo{P_1(t, t_o)\mu_o - P_1(s, t_o)\mu_o} \right)(x)}
      \\
      &
      & + \Lip(P) \left(1 + \frac{2}{\eta}\right) \norma{f}_{\L\infty (\reali;\reali)}
        \modulo{t - s}
    \end{eqnarray*}
    By the Dominated Convergence Theorem, the first term can be
    bounded by $\epsilon/2$ for $\eta$ small. Then, choose $s$ so that
    also the latter summand above is bounded by $\varepsilon/2$.
  \end{proofof}
}

\begin{proofof}{Proposition~\ref{prop:MEASBLCoup}}

  \paragraph{The Narrow Continuity:} This is a consequence of
  Lemma~\ref{lem:muNarCont}.

  \paragraph{Distributional Solution:} To simplify calculations we
  define, for a test function
  $\varphi \in (\C1 \cap \W{1}{\infty})([t_o, T]\times \reali;
  \reali))$,
  \begin{align*}
    \mathcal{I}_\phi(\mu, w)
    & = \int_{\reali_+} \left(
      \partial_t \phi(\cdot, x) + b(\cdot, \mu, w)(x) \partial_x \phi(\cdot, x) - c(\cdot, \mu, w)(x)\phi(\cdot, x)
      \right) d\mu(\cdot, x)\\
    & \quad\ + \int_{\reali_+} \left(
      \int_{\reali_+} \phi(\cdot, x) d[\eta(\cdot, \mu, w)(y)](x)
      \right)	d\mu(\cdot, y).
  \end{align*}

  By a density argument, it suffices to check the integral equality in
  Definition~\ref{def:MVBL} for
  $\phi \in \Cc1([t_o, T]\times \reali_+; \reali))$.  We discretise
  the time domain. For a spacing $k \in \N$, and $i = 0,\dots,k$, we
  introduce the grid points $t_i = t_o + \frac{i(T - t_o)}{k}$, and
  the associated
  $(\tilde{\mu}_i, \tilde{w}_i) = P(t_{i-1}, t_o)(u_o, w_o)$. We then
  split the integral,
  \begin{eqnarray}
    \nonumber
    &
    & \int_{t_o}^T  \mathcal{I}_\phi \left(P(t, t_o)(\mu_o, w_o) \right)\d{t}
    \\
    \label{eqn:MEAS:T1}
    & =
    & \sum_{i=1}^k \int_{t_{i-1}}^{t_i}
      \underbrace{
      \left[
      \mathcal{I}_\phi \left(P(t, t_{i-1})(\tilde{\mu}_i, \tilde{w}_i)  \right)
      - \mathcal{I}_\phi \left(F(t - t_{i-1}, t_{i-1})(\tilde{\mu}_i, \tilde{w}_i)  \right)\
      \right]
      }_{A_{1, i}(t)} \d{t}
    \\
    \label{eqn:MEAS:T2}
    &
    & + \sum_{i=1}^k \int_{t_{i-1}}^{t_i}
      \underbrace{
      \mathcal{I}_\phi
      \left(F(t - t_{i-1}, t_{i-1})(\tilde{\mu}_i, \tilde{w}_i)  \right)
      }_{A_{2, i}(t)} \d{t}  \,.
  \end{eqnarray}

  Our first goal is to demonstrate that \eqref{eqn:MEAS:T1} vanishes
  in the limit $k \to \infty$. Focusing on $A_{1,i}$, we split the
  integral to get
  \begin{eqnarray}
    \nonumber
    &
    & A_{1, i}(t)
    \\
    \label{eqn:MEAS:A1i1}
    & =
    & \int_{\reali_+} \partial_t \phi(t, x)
      \d{
      \left(
      P_1(t, t_{i-1})(\tilde{\mu}_i, \tilde{w}_i)
      - F_1(t - t_{i-1}, t_{i-1})(\tilde{\mu}_i, \tilde{w}_i)
      \right)
      }(x)
    \\
    \nonumber
    &
    & + \int_{\reali_+}
      b\left(t, P(t, t_{i-1})(\tilde{\mu}_i, \tilde{w}_i)\right)(x)
      \partial_x \phi(t, x)
      \d{
      P_1(t, t_{i-1})(\tilde{\mu}_i, \tilde{w}_i)
      }(x)
    \\
    \label{eqn:MEAS:A1i2}
    &
    & - \int_{\reali_+} \!\!
      b(t, F(t - t_{i-1}, t_{i-1})(\tilde{\mu}_i, \tilde{w}_i))(x)
      \partial_x \phi(t, x)
      \d{
      F_1(t - t_{i-1}, t_{i-1})(\tilde{\mu}_i, \tilde{w}_i)
      }(x)
    \\
    \nonumber
    &
    & + \int_{\reali_+}
      c(t, F(t - t_{i-1}, t_{i-1})(\tilde{\mu}_i, \tilde{w}_i))(x)
      \phi(t, x) \d{
      F_1(t - t_{i-1}, t_{i-1})(\tilde{\mu}_i, \tilde{w}_i)
      }(x)
    \\
    \label{eqn:MEAS:A1i3}
    &
    & - \int_{\reali_+}
      c(t, P(t, t_{i-1})(\tilde{\mu}_i, \tilde{w}_i))(x)
      \phi(t, x)
      \d{
      P_1(t, t_{i-1})(\tilde{\mu}_i, \tilde{w}_i)
      }(x)
    \\
    \nonumber
    &
    & + \int_{\reali_+}
      \left(
      \int_{\reali_+} \phi(t, x)\ \
      \d{
      [\eta(t, P(t, t_{i-1})(\tilde{\mu}_i, \tilde{w}_i))(y)]
      }(x)
      \right) \d{
      P_1(t, t_{i-1})(\tilde{\mu}_i, \tilde{w}_i)
      }(y)
    \\
    \label{eqn:MEAS:A1i4}
    &
    & - \int_{\reali_+}
      \! \! \left(
      \int_{\reali_+} \!\! \phi(t, x)
      \d{
      [\eta(t, F(t - t_{i-1}, t_{i-1})(\tilde{\mu}_i, \tilde{w}_i))(y)]
      }(x)
      \right)
    \\
    \nonumber
    &
    & \qquad\qquad \d{
      F_1(t - t_{i-1}, t_{i-1})(\tilde{\mu}_i, \tilde{w}_i)
      }(y) \, .
  \end{eqnarray}
  We now deal with each of these terms separately. To simplify the
  notation we will set
  \begin{equation}
    \label{eq:37}
    \begin{array}{rclcl}
      P_i(t)
      & \equiv
      & (\mu_{i,P}(t), w_{i, P}(t))
      & =
      & P(t, t_{i-1})(\tilde{\mu}_i, \tilde{w}_i),
      \\
      F_i(t)
      & \equiv
      & (\mu_{i,F}(t), w_{i, F}(t))
      & =
      & F(t - t_{i-1}, t_{i-1})(\tilde{\mu}_i, \tilde{w}_i) \,.
    \end{array}
  \end{equation}
  We will make extensive use of the relation \eqref{eq:tangent}, which
  gives
  \begin{align}
    d(P_i(t), F_i(t))
    & \leq \frac{2L}{\ln{2}} (t - t_{i-1})
      \int_{0}^{t - t_{i-1}} \frac{w(\xi)}{\xi}\d\xi \label{eqn:MEAStang}
  \end{align}
  for $L$ as in~\eqref{eqn:tangM}.  For \eqref{eqn:MEAS:A1i1},
  \begin{eqnarray*}
    \modulo{
    \int_{\reali_+} \partial_t \phi(t, x)\
    d\left(
    \mu_{i,P}(t)
    - \mu_{i,F}(t)
    \right)(x)}
    & \leq
    & \norma{\partial_t \phi}_{\W{1}{\infty}(\reali_+;\reali)}
      d_\mathcal{M}\left( \mu_{i,P}(t) ,\mu_{i,F}(t) \right)
    \\
    & \leq
    & \norma{\partial_t \phi}_{\W{1}{\infty}(\reali_+;\reali)} \frac{2L}{\ln{2}}
      (t - t_{i-1}) \int_{0}^{t - t_{i-1}} \frac{\omega(\xi)}{\xi}\d\xi \,.
  \end{eqnarray*}
  Next, for \eqref{eqn:MEAS:A1i2}, calling
  $L_b = \sup_{t \in [0, T], w \in \mathcal{W}}\Lip(b(t, \cdot, w))$,
  \begin{eqnarray*}
    &
    &
      \modulo{
      \int_{\reali_+}
      b(t, P_i(t))(x)
      \partial_x \phi(t, x)\ d\mu_{i,P}(t)(x)
      - \int_{\reali_+}
      b(t, F_i(t))(x)
      \partial_x \phi(t, x)\ d\mu_{i,F}(t)(x)}
    \\
    & =
    & \bigg|
      \int_{\reali_+}
      \left[
      b(t, P_i(t))(x)
      - b(t, F(t, t_{i-1})(\tilde{\mu}_i, \tilde{w}_i))(x)
      \right]
      \partial_x \phi(t, x)\ d\mu_{i,P}(t)(x)
    \\
    &
    &	+ \int_{\reali_+}
      b(t, F_i(t))(x)
      \partial_x \phi(t, x)\
      d\left(
      \mu_{i,P}(t) - \mu_{i,F}(t)
      \right)(x)
      \bigg|
    \\
    & \leq
    & \norma{\partial_x \phi}_{\W{1}{\infty}(\reali_+;\reali)}
      (R L_b + R \hat{L} + B) \; d(P_i(t), F_i(t))
    \\
    & \leq
    & \norma{\partial_x \phi}_{\W{1}{\infty}(\reali_+;\reali)}
      (RL_b + R \hat{L} + B) \frac{2L}{\ln{2}}
      (t - t_{i-1})
      \int_{0}^{t - t_{i-1}} \frac{\omega(\xi)}{\xi}\d\xi \,.
  \end{eqnarray*}
  Repeat the same calculations for~\eqref{eqn:MEAS:A1i3} and set
  $L_c = \sup_{t \in [0, T], w \in \mathcal{W}} \Lip\left(c(t, \cdot,
    w)\right)$,
  \begin{eqnarray*}
    &
    & \modulo{
      \int_{\reali_+}
      c(t, F_i(t))(x)
      \phi(t, x)\ dw_{i,F}(t)(x)
      - \int_{\reali_+}
      c(t, P_i(t))(x)
      \phi(t, x)\ d\mu_{i,P}(t)(x)}
    \\
    & \leq
    & \norma{\phi}_{\W{1}{\infty}(\reali_+;\reali)} (RL_c + R \hat{L} + C)
      \frac{2L}{\ln{2}}
      (t - t_{i-1}) \int_{0}^{t - t_{i-1}} \frac{\omega(\xi)}{\xi}\d\xi \,.
  \end{eqnarray*}
  Finally, for the term \eqref{eqn:MEAS:A1i4}, we find
  \begin{align*}
    & \bigg|
      \int_{\reali_+}
      \left(
      \int_{\reali_+} \phi(t, x) \; d [\eta(t, P_i(t))(y)](x)
      \right)\ d \mu_{i,P}(t)(y)
    \\
    & \quad\ - \int_{\reali_+}
      \left(
      \int_{\reali_+} \phi(t, x) \; d [\eta(t, F_i(t))(y)](x)
      \right)\ d w_{i,F}(t)(y)
      \bigg|
    \\
    & = \bigg|
      \int_{\reali_+}
      \left(
      \int_{\reali_+} \phi(t, x)\ d [\eta(t, P_i(t))(y) - \eta(t, F_i(t))(y)](x)
      \right)\ d \mu_{i,P}(t)(y)
    \\
    & \quad\ + \int_{\reali_+}
      \left(
      \int_{\reali_+} \phi(t, x)\ d [\eta(t, F_i(t))(y)](x)
      \right)\ d \left(
      \mu_{i,P}(t) - w_{i,F}(t)
      \right)(y)
      \bigg|
    \\
    & \leq \norma{\phi}_{\W{1}{\infty}(\reali_+;\reali)} R
      \left(
      \sup_{\stackrel{t\in [0, T]}{w \in \mathcal{W}}}
      \Lip(\eta(t, \cdot, w))
      + \hat{L}
      + E
      \right)
      \frac{2L}{\ln{2}} (t - t_{i-1}) \int_{0}^{t - t_{i-1}}
      \frac{\omega(\xi)}{\xi}\d\xi.
  \end{align*}
  Combining these four estimates together, we have for a constant
  $\mathcal{C}$, independent of $k$,
  \begin{displaymath}
    \modulo{\sum_{i=1}^k \int_{t_{i-1}}^{t_i} A_{1,i}(t)\d{t}}
    \leq
    \mathcal{C} \sum_{i=1}^k \frac{(t_i - t_{i-1})^2}{2}
    \int_0^{\frac{T-t_o}{k}} \frac{\omega(\xi)}{\xi} \d\xi
    \to 0 \mbox{ as } k \to +\infty \,.
  \end{displaymath}
  Now,
  \begin{align*}
    A_{2,i}(t)
    & = \mathcal{I}_\phi \left(F(t - t_{i-1}, t_{i-1})(\tilde{\mu}_i, \tilde{w}_i) \right)
    \\
    & = \mathcal{I}_\phi \left(\mu_{i, F}(t), \tilde{w}_i \right)
    \\
    & \quad\ + \int_{\reali_+}
      (b(t, \mu_{i,F}(t), w_{i, F}(t))(x) - b(t, \mu_{i,F}(t), \tilde{w}_i)(x)) \partial_x \phi(t, x)\ d\mu_{i, F}(t)(x)
    \\
    & \quad\ + \int_{\reali_+}
      (c(t, \mu_{i,F}(t), \tilde{w}_i)(x) - c(t, \mu_{i,F}(t), w_{i,F}(t))(x)) \phi(t, x)\ d\mu_{i, F}(t)(x)
    \\
    & \quad\ + \int_{\reali_+} \!\!
      \left(
      \int_{\reali_+} \phi(t ,x) \,
      d[\eta(t, \mu_{i, F}(t), w_{i,F}(t))(y) - \eta(t, \mu_{i, F}(t), \tilde{w}_i)(y)](x)
      \right) d\mu_{i,F}(t)(x)
  \end{align*}
  and hence
  \begin{align}
    \nonumber
    A_{2,i}(t)
    &\leq \mathcal{I}_\phi \left(\mu_{i, F}(t), \tilde{w}_i\right)
    \\
    \label{eqn:A2i}
    &+ \hat{L} R
      \left(
      2\norma{\phi}_{\W{1,\infty}(\reali_+;\reali)}
      + \norma{\partial_x \phi}_{\W{1,\infty}(\reali_+;\reali)}
      \right)
      \frac{2L}{\ln{2}} (t - t_{i-1})
      \int_{0}^{t - t_{i-1}} \frac{\omega(\xi)}{\xi}\d\xi.
  \end{align}
  The second term will thus converge to zero in the summation.  Hence
  we concentrate on the summation of the first term.

  In the next calculation, we will use the fact
  \begin{align*}
    & \int_{\reali_+} \phi(T, x) \d{\left(\mu_{k,F}(T) - P_1(T, t_o)(u_o, w_o)\right)}(x)\\
    & = \int_{\reali_+} \phi(T, x)
      \d{\left(
      F_1(T - t_{k-1}, t_{k-1})P(t_{k-1}, t_o)(u_o, w_o)
      - P_1(T, t_{k-1})P(t_{k-1}, t_o)(u_o, w_o)
      \right)}(x)
    \\
    & \leq \norma{\phi(T)}_{\W{1}{\infty}(\reali_+;\reali)} \frac{2L}{\ln{2}}
      \frac{T - t_o}{k} \int_0^{\frac{T - t_o}{k}} \frac{\omega(\xi)}{\xi}\d\xi
      \to 0, \mbox{ as } k \to \infty.
  \end{align*}
  Focusing on the summation of the first term in~\eqref{eqn:A2i}
  \begin{align*}
    \sum_{i=1}^k
    \int_{t_{i-1}}^{t_i} \mathcal{I}_\phi
    \left(
    \mu_{i, F}(t),
    \tilde{w}_i
    \right)\d{t}
    & =\sum_{i=1}^k
      \left(
      \int_{\reali_+} \phi(t_i, x)\ d\mu_{i, F}(t_i)(x)
      - \int_{\reali_+} \phi(t_{i-1}, x)\ d\tilde{\mu}_i(x)
      \right)\\
    & = \int_{\reali_+} \phi(T, x)\ d\mu_{T, F}(T)(x)
      - \int_{\reali_+} \phi(t_o, x)\ d\mu_o(x) \\
    & \quad\ + \sum_{i=1}^k
      \left(
      \int_{\reali_+} \phi(t_i, x)\ d(\mu_{i, F}(t_i) - \tilde{\mu}_{i+1})(x)
      \right)\\
    & \underset{k \to +\infty}{\longrightarrow}
      \int_{\reali_+} \phi(T, x) \,
      d(P_1(T, t_o)(u_o, w_o))(x) - \int_{\reali_+} \phi(t_o, x)\ d\mu_o(x),
  \end{align*}
  where we use that
  \begin{align*}
    \sum_{i=1}^k
    \left(
    \int_{\reali_+} \phi(t_i, x)\ d(\mu_{i, F}(t_i) - \tilde{\mu}_{i+1})(x)
    \right)
    \leq \norma{\phi}_{\W{1}{\infty}(\reali_+;\reali)} \frac{2L}{\ln{2}} T \int_0^{\frac{T - t_o}{k}} \frac{\omega(\xi)}{\xi}\d\xi
    \underset{k \to +\infty}{\longrightarrow} 0\,,
  \end{align*}
  completing the proof.
\end{proofof}

\subsection{Proofs for \S~\ref{subsec:scal-cons-laws}}
\label{sec:proofs-s-refs-1}

\begin{proofof}{Proposition~\ref{prop:CLcoupling}}
  We assume for simplicity that both processes $P^u$ and $P^w$ share
  the same constants $C_u,C_w,C_t$
  in~\eqref{eq:6}--\eqref{eq:7}--\eqref{eq:8}.

  The properties of $P$ ensured by \Cref{thm:Metric} show that
  $P_1 \in \C0([t_o, T]; \L1(\reali^n;\reali))$ as required by
  Definition~\ref{def:CL}.

  Introduce the following notation. For any $k \in \reali$ and
  $\phi \in \Cc{\infty}(\hat I \times \reali; \reali_+)$, denote
  \begin{eqnarray*}
    \mathcal{I}_{\phi, k}(u, w)
    & =
    & \int_\reali
      \left[
      \modulo{u - k} \, \partial_t \phi
      + q_k(u, w)
      \, \partial_x \phi
      \right] \d{x} \,,
    \\
    q_k(u, w)
    & =
    & \sign(u - k) \, \left( f(u, w) - f(k, w) \right) \,.
  \end{eqnarray*}
  Fix $N \in \naturali \setminus\left\{0\right\}$ and, for every
  $i \in \left\{0, \ldots, N\right\}$, define
  $t_i = t_o + i \frac{T - t_o}{N}$ and, for $t \in [t_{i-1}, T]$,
  \begin{equation}
    \label{eq:42}
    \begin{array}{rcccl}
      &
      & \left(\tilde u_i, \tilde w_i\right)
      & =
      & P (t_{i-1}, t_o)(u_o, w_o) \,,
      \\
      \bar P_i(t, x)
      & \equiv
      & (u_{i,P}(t, x), w_{i, P}(t))
      & =
      & P(t, t_{i-1})(\tilde{u}_i, \tilde{w}_i)(x) \,,
      \\
      \bar F_i(t, x)
      & \equiv
      & (u_{i,F}(t, x), w_{i, F}(t))
      & =
      & \left(P^{\tilde w_i}\left(t, t_{i-1}\right)\tilde u_i(x),
        P^{\tilde u_i}\left(t, t_{i-1}\right)\tilde w_i \right) \,.
    \end{array}
  \end{equation}

  We now prove in $2$ steps that
  \begin{equation}
    \label{eq:ineq-1}
    \begin{split}
      \int_{t_o}^T \mathcal{I}_{\phi, k}(P\left(t, t_o\right)(u_o,
      w_o)) \d{t} & \geq \int_{\reali} \modulo{P_1\left(T,
          t_o\right)(u_o, w_o) - k} \, \phi(T, x) \d{x}
      \\
      & \quad - \int_{\reali} \modulo{u_o(x) - k} \, \phi(0, x) \d{x}
      \,.
    \end{split}
  \end{equation}

  \paragraph{Step~1:} We prove the inequality
  \begin{equation}
    \label{eq:50}
    \int_{t_o}^T \mathcal{I}_{\phi, k}(P\left(t,
      t_o\right)(u_o, w_o)) \d{t}
    \geq
    \limsup_{N\to +\infty} \sum_{i=1}^N \int_{t_{i-1}}^{t_i}
    \mathcal{I}_{\phi,k} (u_{i.F} (t), \tilde{w}_i) \d{t} \,.
  \end{equation}
  To this aim, write
  \begin{eqnarray}
    \int_{t_o}^T \mathcal{I}_{\phi, k}(P\left(t,
    t_o\right)(u_o, w_o)) \d{t}
    \label{eq:43}
    & =
    & \int_{t_o}^{T} \int_{\reali}
      \modulo{P_1\left(t, t_o\right)(u_o, w_o)(x) - k}
      \partial_t \phi(t, x) \d{x}\, \d{t}
    \\
    \label{eq:44}
    & +
    & \int_{t_o}^{T} \int_{\reali}
      q_k\left(P\left(t, t_o\right)(u_o, w_o)(x)\right)
      \partial_x \phi(t, x)
      \d{x}  \, \d{t}
  \end{eqnarray}
  We proceed towards the estimate of~\eqref{eq:43}. For every
  $i \in \{1, \ldots, N\}$ and $k \in \reali$,
  using~\eqref{eq:tangent} with $L$ and $\omega$ given
  by~\eqref{eqn:tangM}, we have
  \begin{align*}
    & \quad \modulo{\int_{t_{i-1}}^{t_i} \int_{\reali}
      \left[\modulo{u_{i,P} (t, x) - k} \partial_t \phi(t,x) -
      \modulo{u_{i, F}(t, x) - k} \partial_t \phi(t,x)\right] \d{x}\, \d{t}}
    \\
    & \leq \int_{t_{i-1}}^{t_i} \int_{\reali}
      \modulo{u_{i,P} (t, x) - u_{i,F} (t, x)} \partial_t \phi(t,x)
      \d{x}\, \d{t}
    \\
    & \leq \frac{2L}{\ln(2)} \norma{\partial_t \phi}_{\L\infty\left([t_o, T]
      \times \reali; \reali\right)}
      \int_{t_{i-1}}^{t_i} \left(t - t_{i-1}\right)
      \int_0^{t - t_{i-1}} \frac{\omega(\xi)}{\xi} \d{\xi}\, \d{t}
    \\
    & \leq \frac{L}{\ln(2)}\, \frac{\left(T - t_o\right)^2}{N^2}
      \norma{\partial_t \phi}_{\L\infty\left([t_o, T]
      \times \reali; \reali\right)}\int_0^{\frac{T - t_o}{N}}
      \frac{\omega(\xi)}{\xi} \d{\xi}.
  \end{align*}
  Therefore, the term~\eqref{eq:43} is estimated as:
  \begin{align*}
    & \int_{t_o}^{T} \int_{\reali}
      \modulo{P_1\left(t, t_o\right)(u_o, w_o)(x) - k} \,
      \partial_t \phi(t, x) \d{x}\, \d{t}
    \\
    = &
        \sum_{i=1}^N \int_{t_{i-1}}^{t_i} \int_{\reali}
        \modulo{u_{i, P}(t, x) - k} \,
        \partial_t \phi(t, x) \d{x}\, \d{t}
    \\
    \geq
    & \sum_{i=1}^N
      \left[
      \int_{t_{i-1}}^{t_i} \int_{\reali}
      \modulo{u_{i, F}(t, x) - k} \,
      \partial_t \phi(t, x) \d{x} \, \d{t}
      \right]
    \\
    & \qquad\qquad\qquad\qquad
      - \frac{L}{\ln(2)} \, \frac{\left(T - t_o\right)^2}{N} \,
      \norma{\partial_t \phi}_{\L\infty([t_o, T]\times\reali; \reali)}
      \int_0^{\frac{T - t_o}{N}}
      \frac{\omega(\xi)}{\xi} \d{\xi}
  \end{align*}
  and the last term converges to $0$ as $N \to +\infty$. Thus, the
  term~\eqref{eq:43} is estimated as follows:
  \begin{equation}
    \label{eq:51}
    [\mbox{\eqref{eq:43}}]
    \geq \limsup_{N\to+\infty} \sum_{i=1}^N
    \int_{t_{i-1}}^{t_i}
    \int_{\reali}
    \modulo{u_{i, F}(t, x) - k} \,
    \partial_t \phi(t, x)
    \d{x}\, \d{t} \,.
  \end{equation}
  We pass now to the term~\eqref{eq:44}. For every
  $i \in \left\{1, \ldots, N\right\}$ and $k \in \reali$, since $q_k$
  is Lipschitz continuous~\cite[Lemma~3]{Kruzkov} and
  using~\eqref{eq:tangent}, $L_f$ from~\ref{it:CL2}, $L$ and $\omega$
  from~\eqref{eqn:tangM},
  \begin{eqnarray*}
    &
    & \int_{t_{i-1}}^{t_i} \int_{\reali}
      q_k\left(\bar P_i(t, x)\right) \partial_x \phi(t,x)
      \d{x}  \, \d{t}
      - \int_{t_{i-1}}^{t_i} \int_{\reali} q_k(u_{i, F}(t,x), \tilde{w}_i) \,
      \partial_x \phi(t,x) \d{x}  \, \d{t}
    \\
    & =
    & \int_{t_{i-1}}^{t_i} \int_{\reali}
      \left[q_k\left(\bar P_i(t, x)\right) - q_k\left(u_{i, F}(t,x), w_{i,P}(t)
      \right)\right]\partial_x \phi(t,x)
      \d{x}  \, \d{t}
    \\
    &
    & + \int_{t_{i-1}}^{t_i} \int_{\reali}
      \left[q_k\left(u_{i, F} (t,x), w_{i, P}(t)\right)
      - q_k(u_{i, F}(t, x), \tilde{w}_i)  \right]
      \partial_x \phi(t,x) \d{x}  \, \d{t}
    \\
    & \leq
    & L_f \norma{\partial_x \phi}
      _{\L\infty([t_o, T]\times\reali; \reali)}
      \int_{t_{i-1}}^{t_i} \int_{\reali} \modulo{u_{i,P}(t,x) - u_{i, F}(t,x)}
      \d{x}  \, \d{t}
    \\
    &
    & +
      \int_{t_{i-1}}^{t_i} \int_\reali
      \left|
      f\left(u_{i,F}(t,x), w_{i,P}(t)\right)
      -
      f(u_{i,F}(t,x), \tilde{w}_i)
      \right.
    \\
    &
    & \hphantom{\quad + \int_{t_{i-1}}^{t_i} \int_\reali}
      \left.
      -
      \left(
      f\left(k, w_{i,P}(t)\right) - f(k, \tilde{w}_i)
      \right)
      \right| \,
      \modulo{\partial_x \phi(t,x)}\d{x}  \, \d{t}
    \\
    & \leq
    & L_f \, \frac{2L}{\ln(2)} \,
      \norma{\partial_x \phi}_{\L\infty([t_o, T]\times\reali; \reali)}
      \int_{t_{i-1}}^{t_i} (t - t_{i-1})
      \int_0^{t-t_{i-1}}
      \frac{\omega(\xi)}{\xi}
      \d{\xi}  \, \d{t}
    \\
    &
    & + L_f \int_{t_{i-1}}^{t_i} \int_\reali
      \modulo{u_{i,F}(t, x) - k}
      \cdot d_\mathcal{W}(w_{i,P}(t), \tilde{w}_i)
      \cdot \modulo{\partial_x \phi(t, x)} \,
      \d{x}\,\d{t}
    \\
    & \leq
    & L_f \frac{2L}{\ln(2)} \norma{\partial_x \phi}
      _{\L\infty([t_o, T]\times\reali; \reali)}
      \frac{\left(t_i - t_{i-1}\right)^2}{2}
      \int_0^{t_i - t_{i-1}} \frac{\omega(\xi)}{\xi}
      \d{\xi}
    \\
    &
    & + L_f\, C_t (R + k)
      \norma{\partial_x \phi}_{\L\infty([t_o, T] \times \reali; \reali)}
      \int_{t_{i-1}}^{t_i}
      \left(t - t_{i-1}\right) \d{t}
    \\
    & \leq
    & \frac{L_f}{2}
      \norma{\partial_x \phi}_{\L\infty([t_o, T] \times \reali; \reali)}
      \left(L_f\, C_t (R + k)
      + \frac{2L}{\ln(2)} \int_0^{\frac{T - t_o}{N}}
      \frac{\omega(\xi)}{\xi} \d{\xi}\right)
      \frac{\left(T - t_o\right)^2}{N^2} \,.
  \end{eqnarray*}
  Therefore, \eqref{eq:44} is estimated as
  \begin{align*}
    & \int_{t_o}^{T} \int_{\reali}
      q_k\left(P\left(t, t_o\right)(u_o, w_o)(x)\right)
      \partial_x \phi(t, x)
      \d{x}  \, \d{t}
    \\
    \ge
    & \sum_{i=1}^N \int_{t_{i-1}}^{t_i} \int_{\reali}
      q_k(u_{i, F}(t,x), \tilde{w}_i) \,
      \partial_x \phi(t, x)
      \d{x}  \, \d{t}
    \\
    & - \frac{L_f}{2}
      \norma{\partial_x \phi}_{\L\infty([t_o, T] \times \reali; \reali)}
      \left(L_f\, C_t (R + k)
      + \frac{2L}{\ln(2)} \int_0^{\frac{T - t_o}{N}}
      \frac{\omega(\xi)}{\xi} \d{\xi}\right) \frac{\left(T - t_o\right)^2}{N^2}
  \end{align*}
  and the last term converges to $0$ as $N \to +\infty$. Thus,
  \begin{equation}
    \label{eq:52}
    [\mbox{\eqref{eq:44}}]
    \geq \limsup_{N\to+\infty}
    \sum_{i=1}^N
    \int_{t_{i-1}}^{t_i}
    \int_{\reali}
    q_k(u_{i, F}(t,x), \tilde{w}_i) \, \partial_x
    \phi(t, x)
    \d{x} \, \d{t} \,.
  \end{equation}
  Combining~\eqref{eq:51} and~\eqref{eq:52}, the proof of Step~1,
  namely~\eqref{eq:50}, is completed.

  \paragraph{Step~2:} Now we prove that
  \begin{equation}
    \label{eq:49}
    \begin{split}
      & \liminf_{N\to +\infty} \sum_{i=1}^N \int_{t_{i-1}}^{t_i}
      \mathcal{I}_{\phi,k} (u_{i.F} (t), \tilde{w}_i) \d{t}
      \\
      \geq & \int_\reali \modulo{P_1(T, t_o)(u_o, w_o) (x) - k} \,
      \phi(T,x) \d{x} -\int_\reali \modulo{u_o(x) - k} \, \phi(t_o,x)
      \d{x}
    \end{split}
  \end{equation}

  Fix $i \in \left\{1, \ldots, N\right\}$. For $\eps>0$ sufficiently
  small, consider
  $\chi_\eps \in \Cc{\infty}\left( ]t_{i-1}, t_i[; [0,1]\right)$ such
  that $\chi_\eps(t) = 1$ for $t \in [t_{i-1} + \eps, t_i -\eps]$ and
  define $\phi_\eps = \phi \cdot \chi_\eps$. Then, by \Cref{def:CL}
  and the choice of $\chi_{\eps}$, we have that for every $\eps > 0$
  sufficiently small,
  \begin{equation*}
    \int_{t_{i-1}}^{t_i}
    \mathcal{I}_{\phi_\varepsilon,k} (u_{i,F} (t,x),\tilde{w}_i)
    \d{t}
    \ge
    0 \,.
  \end{equation*}
  This implies that
  {\begin{eqnarray}
      \nonumber
      &
      & \int_{t_{i-1}}^{t_i} \mathcal{I}_{\phi,k} (u_{i.F} (t),
        \tilde{w}_i) \d{t}
      \\
      \nonumber
      & =
      & \int_{t_{i-1}}^{t_i} \mathcal{I}_{\phi-\phi_\varepsilon,k} (u_{i.F}
        (t), \tilde{w}_i) \d{t} + \int_{t_{i-1}}^{t_i}
        \mathcal{I}_{\phi_\varepsilon,k} (u_{i.F} (t), \tilde{w}_i)
        \d{t}
      \\
      \nonumber
      & \geq
      & \int_{t_{i-1}}^{t_i} \mathcal{I}_{\phi-\phi_\varepsilon,k}
        (u_{i.F} (t), \tilde{w}_i) \d{t}
      \\
      \label{eq:47}
      & =
      & \int_{t_{i-1}}^{t_i} \int_{\reali} \modulo{u_{i,F}(t,x) -
        k} \partial_t \left(\phi - \phi_\eps\right)(t, x) \d{x}\, \d{t}
      \\
      \label{eq:48}
      &
      & + \int_{t_{i-1}}^{t_i} \int_{\reali}
        q_k(u_{i,F}(t,x), \tilde{w}_i) \,
        \partial_x (\phi - \phi_\eps)(t,x) \, \d{x} \, \d{t}
    \end{eqnarray}}
  for every $\eps >0$ sufficiently small. Moreover the continuity in
  time of $u_{i,F}$ implies that
  \begin{displaymath}
    \lim_{\eps \to 0^+} [\mbox{\eqref{eq:47}}]
    =
    \int_{\reali}
    \modulo{u_{i,F}(t_i, x) - k} \,
    \phi (t_i, x) \, \d{x}
    - \int_{\reali}
    \modulo{u_{i,F}(t_{i-1}, x) - k} \,
    \phi (t_{i-1}, x) \, \d{x},
  \end{displaymath}
  while, by the Dominated Convergence Theorem, we deduce that
  \begin{displaymath}
    \lim_{\eps \to 0^+} [\mbox{\eqref{eq:48}}]
    =
    \lim_{\eps \to 0^+} \int_{t_{i-1}}^{t_i}
    \int_{\reali} q_k(u_{i, F}(t,x), \tilde{w}_i) \,
    \partial_x (\phi - \phi_\eps)(t,x) \d{x}  \, \d{t}
    = 0 \,.
  \end{displaymath}
  Therefore, we get
  \begin{eqnarray*}
    &
    & \int_{t_{i-1}}^{t_i} \mathcal{I}_{\phi,k} (u_{i.F} (t), \tilde{w}_i) \,
      \d{t}
    \\
    & \geq
    & \int_{\reali}
      \modulo{u_{i,F}(t_i, x) - k} \,
      \phi (t_i, x) \d{x} - \int_{\reali}
      \modulo{u_{i,F}(t_{i-1}, x) - k} \,
      \phi (t_{i-1}, x) \, \d{x} \,.
  \end{eqnarray*}
  Summing over $i$, we obtain that
  \begin{eqnarray}
    \nonumber
    &
    & \sum_{i=1}^N
      \int_{t_{i-1}}^{t_i} \mathcal{I}_{\phi,k} (u_{i.F} (t), \tilde{w}_i) \d{t}
    \\
    \nonumber
    & \geq
    & \sum_{i=1}^N \int_{\reali}
      \modulo{u_{i,F}(t_i, x) - k} \, \phi (t_i, x) \d{x}
      - \sum_{i=1}^N \int_{\reali} \modulo{u_{i,F}(t_{i-1}, x) - k} \,
      \phi (t_{i-1}, x) \d{x}
    \\
    \label{eq:40}
    & =
    & \int_\reali \modulo{u_{N,F}(T,x) - k} \, \phi(T,x) \d{x}
      - \int_\reali \modulo{u_o(x) - k} \, \phi(t_o,x) \d{x}
    \\
    \label{eq:41}
    &
    & \qquad + \sum_{i=1}^{N-1} \int_\reali
      \left(
      \modulo{u_{i,F}(t_i,x) - k} - \modulo{u_{i+1,F}(t_i,x) - k}
      \right) \, \phi(t_i,x) \d{x} \,.
  \end{eqnarray}
  We now estimate the first term in~\eqref{eq:40}:
  \begin{eqnarray*}
    &
    & \int_\reali \modulo{u_{N,F}(T,x) - k} \, \phi(T,x) \d{x}
      - \int_\reali \modulo{P_1(T, t_o)(u_o, w_o) (x) - k} \, \phi(T,x) \d{x}
    \\
    & =
    & \int_\reali
      \left(
      \modulo{
      F_1(T - t_{N-1}, t_{N-1})(\tilde{u}_{N-1}, \tilde{w}_{N-1}) (x) - k}
      - \modulo{P_1(T, t_o)(u_o, w_o) (x) - k}
      \right) \phi(T,x) \d{x}
  \end{eqnarray*}
  and, using $L$ and $\omega$ as in~\eqref{eqn:tangM}, we get
  \begin{eqnarray*}
    &
    & \modulo{
      \int_\reali
      \left(
      \modulo{
      F_1(T - t_{N-1}, t_{N-1})(\tilde{u}_{N-1}, \tilde{w}_{n-1}) (x) - k}
      {-}
      \modulo{P_1(T, t_o)(u_o, w_o) (x) - k}
      \right)\phi(T,x) \d{x}}
    \\
    & \leq
    & \int_\reali
      \left|
      F_1(T - t_{N-1}, t_{N-1})P(t_{N-1}, t_o)(u_o, w_o) (x)
      \right.
    \\
    &
    & \qquad
      \left.
      -
      P_1(T, t_{N-1})P(t_{N-1}, t_o)(u_o, w_o) (x)
      \right|
      \phi(T,x) \d{x}
    \\
    & \leq
    & \frac{2L}{\ln(2)} \; \frac{T - t_o}{N} \; \int_0^{\frac{T - t_o}{N}}
      \frac{\omega(\xi)}{\xi} \d\xi
    \\
    \!\!\!
    & \to
    & 0 \quad \mbox{ as } N \to +\infty \,.
  \end{eqnarray*}
  We now estimate~\eqref{eq:41} using~\eqref{eq:42}
  and~\eqref{eq:tangent}
  \begin{eqnarray*}
    &
    &  \sum_{i=1}^{N-1} \int_\reali
      \modulo{
      \modulo{u_{i,F}(t_i,x) - k} - \modulo{u_{i+1,F}(t_i,x) - k}
      } \,
      \phi(t_i,x) \d{x}
    \\
    & \leq
    & \sum_{i=1}^{N-1} \int_\reali
      \modulo{u_{i,F}(t_i,x) - u_{i+1,F}(t_i,x)} \, \phi(t_i,x) \d{x}
    \\
    & =
    & \sum_{i=1}^{N-1} \int_\reali
      \modulo{
      P^{\tilde w_i}(t_i, t_{i-1}) \tilde u_i(x)
      -
      P_1 (t_{i}, t_{i-1}) \tilde u_i(x)
      }\,
      \phi(t_i,x) \d{x}
    \\
    & \leq
    & \norma{\phi}_{\L\infty ([t_o,T]\times\reali; \reali)}
      \sum_{i=1}^{N-1} \norma{P^{\tilde w_i}(t_i, t_{i-1})
      \tilde u_i
      - P_1 (t_{i}, t_{i-1}) \tilde u_i}_{\L1 (\reali;\reali)}
    \\
    & \leq
    & \frac{2\, L}{\ln 2}
      \norma{\phi}_{\L\infty ([t_o,T]\times\reali; \reali)}
      \sum_{i=1}^{N-1} (t_i - t_{i-1}) \int_0^{t_i-t_{i-1}} \dfrac{\omega (\tau)}{\tau} \d\tau
    \\
    & \leq
    & \frac{2\, L}{\ln 2}
      \norma{\phi}_{\L\infty ([t_o,T]\times\reali; \reali)}
      (T-t_o) \int_0^{(T-t_o)/N} \dfrac{\omega (\tau)}{\tau} \d\tau
    \\
    & \to
    & 0 \quad \mbox{ as } N \to +\infty \,.
  \end{eqnarray*}
  The obtained estimates for~\eqref{eq:40} and~\eqref{eq:41}, as
  $N \to +\infty$, proved Step~2, namely~\eqref{eq:49}.
\end{proofof}

{%
  \appendix
  \section{Appendix: \textbf{BV} Estimates}

  We gather here a few estimates on $\BV$ functions used in the
  proofs.

  \begin{lemma}
    \label{lem:BVresults}
    Recall the following elementary estimates on $\BV$ functions, see
    also~\cite[\S~4.2]{MR4283943} or~\cite{AmbrosioFuscoPallara}:
    \begin{eqnarray}
      \label{eq:TV1}
      \left.
      \begin{array}{@{}r@{}c@{}l@{}}
        u
        & \in
        & \BV (\reali_+; \reali)
        \\
        w
        & \in
        & \BV (\reali_+; \reali)
      \end{array}
          \right\}
        & \Rightarrow
        & \tv (u \, w)
          \leq
          \tv (u) \norma{w}_{\L\infty (\reali_+; \reali)}
          {+}
          \norma{u}_{\L\infty (\reali_+; \reali)} \tv (w)
      \\
      \label{eq:TV2}
      \left.
      \begin{array}{r@{\,}c@{\,}l@{}}
        \varphi
        & \in
        & \C{0,1} (\reali^n; \reali)
        \\
        u
        & \in
        & \BV (\reali_+; \reali^n)
      \end{array}
          \right\}
        & \Rightarrow
        & \tv (\varphi\circ u) \leq \Lip (\varphi) \, \tv (u)
      \\
      \label{eq:TV3}
      \!\!\!\!\!\!\!\!\!\!\!\!\!\!\!\left.
      \begin{array}{@{}l@{}}
        u
        \in
        \L1 (\hat I; \L1(\reali_+;\reali))
        \\
        u (t)
        \in
        \BV (\reali_+; \reali)
      \end{array}
      \right\}
        & \Rightarrow
        & \tv\left(\int_{t_o}^t u (\tau,\cdot) \, \d\tau\right)
          \leq
          \int_{t_o}^t \tv\left(u (\tau)\right) \, \d\tau
      \\
      \label{eq:3TV}
      \left.
      \begin{array}{r@{\,}c@{\,}l@{}}
        u
        & \in
        & \BV (\reali_+; \reali)
        \\
        \delta
        & \in
        & \L\infty(\reali; \reali_+)
      \end{array}
          \right\}
        & \Rightarrow
        & \int_{\reali_+}
          \modulo{u\left(x + \delta (x)\right) - u (x)}  \d{x}
          \leq
          \tv (u) \, \norma{\delta}_{\L\infty (\reali_+; \reali)}
    \end{eqnarray}
    and in~\eqref{eq:TV3} we have $t_o, t \in \hat I$ with
    $t_o \leq t$.
  \end{lemma}

\begin{proofof}{Lemma~\ref{lem:BVresults}}
  Inequality~\eqref{eq:TV1} follows
  from~\cite[Formula~(3.10)]{AmbrosioFuscoPallara}.  The one
  dimensional proof follows. For any partition $ (x_i)_{i=0}^N $ of
  $\reali_+$, we have
  \begin{eqnarray*}
    &
    & \sum_{i=1}^N \modulo{u(x_i)w(x_i) - u(x_{i-1})w(x_{i-1})}
    \\
    & \leq
    & \sum_{i=1}^N \modulo{u(x_i) - u(x_{i-1})} \, \modulo{w(x_i)}
      + \sum_{i=1}^N \modulo{w(x_i) - w(x_{i-1})} \, \modulo{u(x_{i-1})}
    \\
    & \leq
    & \norma{w}_{\L\infty(\reali_+;\reali)}
      \sum_{i=1}^N \modulo{u(x_i) - u(x_{i-1})}
      + \norma{u}_{\L\infty(\reali_+;\reali)}
      \sum_{i=1}^N \modulo{w(x_i) - w(x_{i-1})}
    \\
    & \leq
    & \tv(u) \, \norma{w}_{\L\infty(\reali_+;\reali)}
      + \norma{u}_{\L\infty(\reali_+;\reali)} \, \tv(w) \, ,
  \end{eqnarray*}
  and taking the supremum over all such sequence, we get our required
  result.

  The definition of total variation directly
  implies~\eqref{eq:TV2}
  and~\eqref{eq:TV3}.  For a proof of~\eqref{eq:3TV} see for
  instance~\cite[Lemma~2.3]{BressanLectureNotes}.
\end{proofof}
}

\section*{Acknowledgments}
RMC and MG were partly supported by the GNAMPA~2022 project
\emph{"Evolution Equations: Well Posedness, Control and
  Applications"}.  MT acknowledges the grant \emph{Wave Phenomena and
  Stability - a Shocking Combination} (WaPheS), by the Research
Council of Norway, and was partly supported by the fund for
international cooperation of the University of Brescia.

\section*{Conflict  of Interest }
The author declares no conflicts of interest in this paper.

{\small

  \bibliographystyle{abbrv}

  \bibliography{ColomboGaravelloTandy}

}

\end{document}